\pdfoutput=1
\documentclass[hidelinks, reqno]{amsart}
\usepackage[english,french]{babel}
\usepackage{graphicx}
\usepackage{subcaption}
\usepackage{tabularx}
\usepackage{booktabs}
\usepackage{array}
\usepackage{amsmath}
\usepackage{amsfonts}
\usepackage{amssymb}
\usepackage{amsthm}
\usepackage[scr]{rsfso}
\usepackage{enumitem}
\usepackage{permute}
\usepackage{slashed}
\usepackage[usenames,dvipsnames]{xcolor}
\usepackage[pagebackref = true, colorlinks, linkcolor = Red, citecolor = Green, bookmarksdepth=2, linktocpage=true]{hyperref}
\usepackage[capitalize]{cleveref}
\usepackage{comment}

\usepackage[T1]{fontenc}
\usepackage{makecell}

\usepackage{tikz}
\usetikzlibrary{calc}

\allowdisplaybreaks

\DeclareMathOperator{\interior}{int}

\DeclareMathOperator{\Id}{Id}

\DeclareMathOperator{\Stab}{Stab}

\DeclareMathOperator{\supp}{supp}

\DeclareMathOperator{\SO}{SO}
\DeclareMathOperator{\U}{U}
\DeclareMathOperator{\PSL}{PSL}

\DeclareMathOperator{\Lip}{Lip}

\DeclareMathOperator{\proj}{proj}

\DeclareMathOperator{\T}{T}
\DeclareMathOperator{\F}{F}

\DeclareMathOperator{\Isom}{Isom}

\DeclareMathOperator{\diam}{diam}

\DeclareMathOperator{\len}{len}
\let\Pr\relax
\DeclareMathOperator{\Pr}{Pr}

\DeclareMathOperator{\ad}{ad}
\DeclareMathOperator{\Ad}{Ad}
\DeclareMathOperator{\BP}{BP}

\DeclareMathOperator{\Span}{span}

\DeclareMathOperator{\li}{li}
\DeclareMathOperator{\Core}{Core}
\DeclareMathOperator{\Hull}{Hull}

\DeclareFontFamily{U}{mathb}{\hyphenchar\font45}
\DeclareFontShape{U}{mathb}{m}{n}{
      <5> <6> <7> <8> <9> <10> gen * mathb
      <10.95> mathb10 <12> <14.4> <17.28> <20.74> <24.88> mathb12
      }{}
\DeclareSymbolFont{mathb}{U}{mathb}{m}{n}
\DeclareMathSymbol{\bigast}{2}{mathb}{"06}

\def\XXint#1#2#3{{\setbox0=\hbox{$#1{#2#3}{\int}$}
     \vcenter{\hbox{$#2#3$}}\kern-.5\wd0}}

\theoremstyle{plain}
\newtheorem{theorem}{Theorem}[section]
\newtheorem{proposition}[theorem]{Proposition}
\newtheorem{lemma}[theorem]{Lemma}
\newtheorem{corollary}[theorem]{Corollary}

\theoremstyle{definition}
\newtheorem{definition}[theorem]{Definition}

\newtheoremstyle{remark}
{}   % ABOVESPACE, \topsep for no space
{}   % BELOWSPACE, \topsep for no space
{\normalfont}  % BODYFONT
{}       % INDENT (empty value is the same as 0pt)
{\itshape} % HEADFONT
{.}         % HEADPUNCT
{5pt plus 1pt minus 1pt} % HEADSPACE
{}          % CUSTOM-HEAD-SPEC

\theoremstyle{remark}
\newtheorem*{remark}{Remark}

\setlist[enumerate,1]{ref=(\arabic*)}
\setlist[enumerate,2]{ref=(\theenumi)(\alph*)}
\setlist[enumerate,3]{ref=(\theenumi)(\theenumii)(\roman*)}
\setlist[enumerate,4]{ref=(\theenumi)(\theenumii)(\theenumiii)(\Alph*)}

\newlist{alternative}{enumerate}{4}     % this creates a dedicated counter named 'subtaski'
\setlist[alternative,1]{label=(\arabic*), ref=(\arabic*)}
\setlist[alternative,2]{label=(\alph*), ref=(\thealternativei)(\alph*)}
\setlist[alternative,3]{label=(\roman*), ref=(\thealternativei)(\thealternativeii)(\roman*)}
\setlist[alternative,4]{label=(\Alph*), ref=(\thealternativei)(\thealternativeii)(\thealternativeiii)(\Alph*)}

\Crefname{enumi}{Property}{Properties}
\Crefname{alternativei}{Alternative}{Alternatives}
\Crefname{subsection}{Subsection}{Subsections}

\begin{document}
\selectlanguage{english}

%\setlength{\parindent}{0 in}
%\setlength{\mathindent}{0.5 in}
%\numberwithin{equation}{section}
%\renewcommand*{\abstractname}{Introduction}

\title[Exponential Mixing of Frame Flows]{Exponential Mixing of Frame Flows for Convex Cocompact Hyperbolic Manifolds}
\author{Pratyush Sarkar}
\address{Department of Mathematics, Yale University, New Haven, Connecticut 06511}
\email{pratyush.sarkar@yale.edu}

\author{Dale Winter}
\email{dale.alan.winter@gmail.com}

\date{\today}

\begin{abstract}
The aim of this paper is to establish exponential mixing of frame flows for convex cocompact hyperbolic manifolds of arbitrary dimension with respect to the Bowen--Margulis--Sullivan measure. Some immediate applications include an asymptotic formula for matrix coefficients with an exponential error term as well as the exponential equidistribution of holonomy of closed geodesics. The main technical result is a spectral bound on transfer operators twisted by \emph{holonomy}, which we obtain by building on Dolgopyat's method.
\end{abstract}

\maketitle

\selectlanguage{english}

\setcounter{tocdepth}{1}
\tableofcontents

\section{Introduction}
\label{sec:Introduction}
Let $\mathbb H^n$ be the $n$-dimensional hyperbolic space for $n \geq 2$. Let $G = \SO(n, 1)^\circ$ which we recognize as the group of orientation preserving isometries of $\mathbb H^n$. Let $\Gamma < G$ be a Zariski dense torsion-free discrete subgroup. Let $X = \Gamma \backslash \mathbb H^n$. We assume that $\Gamma$ is convex cocompact, i.e., the convex core $\Core(X) \subset X$, which is the smallest closed convex subset containing all closed geodesics, is compact. We identify $X$, its unit tangent bundle $\T^1(X)$, and its frame bundle $\F(X)$ with $\Gamma \backslash G/K$, $\Gamma \backslash G/M$, and $\Gamma \backslash G$ respectively where $M \cong \SO(n - 1) < K \cong \SO(n)$ are appropriate compact subgroups of $G$. Let $A = \{a_t: t \in \mathbb R\} < G$ be a one-parameter subgroup of semisimple elements such that its right translation action corresponds to the geodesic flow on $\Gamma \backslash G/M$ and the frame flow on $\Gamma \backslash G$. Let $\mathsf{m}$ be the Bowen--Margulis--Sullivan probability measure on $\Gamma \backslash G$ which is an $M$-invariant lift of the one on $\Gamma \backslash G/M$, which is known to be the measure of maximal entropy. Since $\Gamma$ is convex cocompact, we note that $\mathsf{m}$ is compactly supported. If $\Gamma$ is cocompact, then $\mathsf{m}$ is simply the $G$-invariant probability measure. By the works of Babillot \cite{Bab02} and Winter \cite{Win15}, the frame flow is known to be mixing with respect to $\mathsf{m}$.

Whether the frame flow is \emph{exponentially} mixing  with respect to $\mathsf{m}$ is a fundamental question in the study of dynamics because it has many applications including orbit counting, equidistribution, prime geodesic theorems, and shrinking target problems (see for example \cite{DRS93,EM93,BO12,MMO14,MO15,KO21}).

For lattices, exponential mixing of the geodesic flow is due to Moore \cite{Moo87} and Ratner \cite{Rat87}. The proof is based on the $L^2$ spectral gap of the Laplacian. For a general convex cocompact $\Gamma$, this approach does not work. However, in this case, Stoyanov \cite{Sto11} was able to use Dolgopyat's method \cite{Dol98} to prove exponential mixing of the geodesic flow.

The main aim of this paper is to prove the following theorem about exponential mixing of the \emph{frame flow} extending Stoyanov's work.
 
\begin{theorem}
\label{thm:TheoremExponentialMixingOfFrameFlow}
There exist $\eta > 0$, $C > 0$, and $r \in \mathbb N$ such that for all $\phi \in C_{\mathrm{c}}^r(\Gamma \backslash G, \mathbb R)$, $\psi \in C_{\mathrm{c}}^1(\Gamma \backslash G, \mathbb R)$, and $t > 0$, we have
\begin{align*}
\left|\int_{\Gamma \backslash G} \phi(xa_t)\psi(x) \, d\mathsf{m}(x) - \mathsf{m}(\phi) \cdot \mathsf{m}(\psi)\right| \leq Ce^{-\eta t} \|\phi\|_{C^r} \|\psi\|_{C^1}.
\end{align*}
\end{theorem}

Let $\delta_\Gamma \in (0, n - 1]$ be the critical exponent of $\Gamma$. If $\delta_\Gamma > \frac{n - 1}{2}$ for $n \in \{2, 3\}$, or if $\delta_\Gamma > n - 2$ for $n \geq 4$, then \cref{thm:TheoremExponentialMixingOfFrameFlow} has been established for geometrically finite groups by Mohammadi--Oh \cite{MO15} using spectral gap. Recently, for $M$-invariant functions, Edwards--Oh \cite{EO21} have improved the condition to $\delta_\Gamma > \frac{n - 1}{2}$ for all $n \geq 2$. Hence, the novelty of \cref{thm:TheoremExponentialMixingOfFrameFlow} lies in the treatment of all convex cocompact subgroups $\Gamma < G$ regardless of the magnitude of $\delta_\Gamma$.

Fix a right $G$-invariant measure on $\Gamma \backslash G$ induced from some fixed Haar measure on $G$. We denote by $m^{\mathrm{BR}}$ and $m^{\mathrm{BR}_*}$ the unstable and stable Burger--Roblin measures on $\Gamma \backslash G$ respectively, compatible with the choice of the Haar measure. Using Roblin's transverse intersection argument \cite{Rob03,OS13,OW16}, we can derive the following theorem regarding decay of matrix coefficients from \cref{thm:TheoremExponentialMixingOfFrameFlow}.

\begin{theorem}
\label{thm:TheoremDecayOfMatrixCoefficients}
There exist $\eta > 0$ and $r \in \mathbb N$ such that for all $\phi \in C_{\mathrm{c}}^r(\Gamma \backslash G, \mathbb R)$ and $\psi \in C_{\mathrm{c}}^1(\Gamma \backslash G, \mathbb R)$, there exists $C > 0$ such that for all $t > 0$, we have
\begin{align*}
\left|e^{(n - 1 - \delta_\Gamma)t}\int_{\Gamma \backslash G} \phi(xa_t)\psi(x) \, dx - m^{\mathrm{BR}}(\phi) \cdot m^{\mathrm{BR}_*}(\psi)\right| \leq Ce^{-\eta t} \|\phi\|_{C^r} \|\psi\|_{C^1}
\end{align*}
where $C$ depends only on $\supp(\phi)$ and $\supp(\psi)$.
\end{theorem}

\begin{remark}
\Cref{thm:TheoremExponentialMixingOfFrameFlow,thm:TheoremDecayOfMatrixCoefficients} in fact hold for H\"{o}lder functions using the appropriate H\"{o}lder norms. The decay exponent $\eta$ then depends on the H\"{o}lder exponent. For the first theorem, this is derived by a convolutional argument, originally by Moore \cite{Moo87} and Ratner \cite{Rat87} and generalized by Kleinbock--Margulis \cite[Appendix]{KM96}. For the second theorem, the convolutional argument does not apply directly. Instead, it must be derived by going through Roblin's transverse intersection argument from the H\"{o}lder version of \cref{thm:TheoremExponentialMixingOfFrameFlow}.
\end{remark}

For all $T > 0$, define
\begin{align*}
\mathcal{G}(T) ={}&\#\{\gamma: \gamma \text{ is a primitive closed geodesic in } \Gamma \backslash \mathbb H^n \text{ with length at most } T\}.
\end{align*}
For all primitive closed geodesics $\gamma$ in $\Gamma \backslash \mathbb H^n$, its \emph{holonomy} is a conjugacy class $h_\gamma$ in $M$ induced by parallel transport along $\gamma$. Fix the probability Haar measure on $M$. Recall the function $\li: (2, \infty) \to \mathbb R$ defined by $\li(x) = \int_2^x \frac{1}{\log(t)} \, dt$ for all $x \in (2, \infty)$. We can also derive the following theorem regarding exponential equidistribution of holonomy of closed geodesics as in \cite{MMO14} from \cref{thm:TheoremDecayOfMatrixCoefficients}.

\begin{remark}
When following the proof of \cite[Theorem 1.2]{MMO14}, the source of the hypothesis $\delta_\Gamma > n - 2$ (as $\Gamma$ has no parabolic elements) is actually two-fold and we can dispense with the hypothesis for both sources. The first source is of course \cite[Theorem 4.2]{MMO14}, quoted from \cite{MO15}, which we simply replace with \cref{thm:TheoremDecayOfMatrixCoefficients}. The second source is \cite[Remark 4.5]{MMO14} which can be replaced with \cite[Lemma 3.8]{DFSU20} for the purposes of the proof of \cite[Theorem 4.9]{MMO14}.
\end{remark}

\begin{theorem}
\label{thm:TheoremEquidistributionOfHolonomy}
There exist $\eta > 0$ and $C > 0$ such that for all class functions $\phi \in C^\infty(M, \mathbb R)$ and $T > \frac{\log(2)}{\delta_\Gamma}$, we have
\begin{align*}
\left|\sum_{\gamma \in \mathcal{G}(T)} \phi(h_\gamma) - \li\bigl(e^{\delta_\Gamma T}\bigr) \int_M \phi(m) \, dm\right| \leq Ce^{(\delta_\Gamma - \eta)T}.
\end{align*}
\end{theorem}

For lattices, \cref{thm:TheoremEquidistributionOfHolonomy} was obtained by Sarnak--Wakayama \cite{SW99} using the Selberg trace formula. We also remark that the analogue of \cref{thm:TheoremEquidistributionOfHolonomy} for \emph{hyperbolic} rational maps on the Riemann sphere was obtained by Oh--Winter \cite{OW17} and hence adding to Sullivan's dictionary: holonomies are exponentially equidistributed both for hyperbolic rational maps on the Riemann sphere and closed geodesics in convex cocompact hyperbolic manifolds.

\subsection{Outline of the proof of \texorpdfstring{\cref{thm:TheoremExponentialMixingOfFrameFlow}}{\autoref{thm:TheoremExponentialMixingOfFrameFlow}}}
As $\Gamma$ is convex cocompact, we have existence of a Markov section on $\supp(\mathsf{m})$ by the works of Bowen and Ratner \cite{Bow70,Rat73}. This gives a coding for the geodesic flow and immediately provides tools from symbolic dynamics and thermodynamic formalism at our disposal. In particular, denoting $U$ to be the union of the strong unstable leaves of the Markov section, we have the \emph{transfer operators} $\mathcal{L}_\xi: C(U, \mathbb C) \to C(U, \mathbb C)$ for $\xi = a + ib \in \mathbb C$ defined by
\begin{align*}
\mathcal{L}_\xi(h)(u) = \sum_{u' \in \sigma^{-1}(u)} e^{-(a + \delta_\Gamma - ib)\tau(u')} h(u')
\end{align*}
where $\tau$ is the first return time map. Using techniques originally observed by Pollicott \cite{Pol85}, we can prove exponential mixing of the \emph{geodesic flow} if we obtain appropriate spectral bounds for the transfer operators. For small frequencies $|b| \ll 1$, the spectral bounds follow from the Ruelle--Perron--Frobenius theorem together with perturbation theory of operators. For large frequencies $|b| \gg 1$, the spectral bounds are much more difficult to obtain, but it was achieved by the important work of Dolgopyat \cite{Dol98} and later generalized by Stoyanov \cite{Sto11}. A key ingredient in Dolgopyat's method is the \emph{local non-integrability condition} (LNIC) from which we can infer that $\tau$ is highly oscillating.

We wish to now follow the above framework to prove exponential mixing of the \emph{frame flow}. In this case, we need to consider instead the \emph{transfer operators with holonomy} which are twisted by irreducible representations of the compact subgroup $M$. That is, for a given irreducible representation $\rho: M \to \U(V_\rho)$ and $\xi = a + ib \in \mathbb C$, we consider $\mathcal{M}_{\xi, \rho}: C\bigl(U, V_\rho^{\oplus \dim(\rho)}\bigr) \to C\bigl(U, V_\rho^{\oplus \dim(\rho)}\bigr)$ defined by
\begin{align*}
\mathcal{M}_{\xi, \rho}(H)(u) &= \sum_{u' \in \sigma^{-1}(u)} e^{-(a + \delta_\Gamma - ib)\tau(u')} \rho(\vartheta(u')^{-1}) H(u')
\end{align*}
where $\vartheta$ is the holonomy. Now we must overcome certain difficulties when following Dolgopyat's method.

The first difficulty is that we need to prove a more general LNIC which deals with both $\tau$ and $\vartheta$ combined together into an $AM$-valued map $\Phi$ instead of just the $A$-valued map $\tau$. Working with $\Phi$, we need not worry about the competing oscillations of $\tau$ and $\vartheta$ interfering with each other. We are able to prove this LNIC using Lie theoretic techniques. The arguments also crucially rely on the Zariski density of $\Gamma < G$ which is expected as it was already required to show mixing of the frame flow \cite{Win15}. The high oscillations of $\Phi$ are then carried through by large $b \in \mathbb R$ or nontrivial irreducible representations $\rho: M \to \U(V_\rho)$.

The second difficulty is that we require a new ingredient which we call the \emph{non-concentration property} (NCP) which was not required to prove exponential mixing of the geodesic flow \cite{Sto11}. This property roughly says that given that $\Gamma < G$ is Zariski dense, its limit set does not concentrate along any particular direction. Note that if $\Gamma$ is a lattice, then the limit set is all of $\partial_\infty(\mathbb H^n)$, in which case the NCP is trivial.

We also have technical difficulties to overcome in order to execute the argument carefully because we use Riemannian geometry and Lie theory while the limit set and the Markov section at hand are of fractal nature. After these details are taken care of, Dolgopyat's method runs smoothly with the LNIC and NCP and we obtain the desired bounds for the transfer operators with holonomy, completing the proof.

\begin{remark}
Similar twisted transfer operators have also been considered by Dolgopyat \cite{Dol02} but in the context of compact extensions of hyperbolic diffeomorphisms rather than flows.
\end{remark}

\subsection{Organization of the paper}
\label{subsec:OrganizationOfThePaper}
We start with covering the necessary background and important constructions in \cref{sec:Preliminaries,sec:CodingTheGeodesicFlow,sec:HolonomyAndRepresentationTheory,sec:TransferOperatorsWithHolonomy}. Next, we prepare for Dolgopyat's method by covering the necessary ingredients in \cref{sec:LNIC&NCP,PreliminaryLemmasAndConstants}. In \cref{sec:ConstructionOfDolgopyatOperators,sec:ProofOfFrameFlowDolgopyat}, we construct the Dolgopyat operators and go through Dolgopyat's method to obtain spectral bounds for large frequencies or nontrivial irreducible representations of $M$. Finally in \cref{sec:ExponentialMixingOfTheFrameFlow}, we use the obtained spectral bounds to carefully go through arguments by Pollicott along with Paley--Wiener theory to prove exponential mixing of the frame flow.

\subsection*{Acknowledgements}
First and foremost, Winter would like to thank Ralf Spatzier, without whom this paper could in no way have been completed. He was a large part of the development of the underlying ideas, provided very extensive technical help, and was a source of encouragement and good humor throughout. Winter is also particularly grateful to Mark Pollicott both for several useful conversations and for his patience during them. Sarkar and Winter also thank others for technical advice and helpful conversations, amongst them Hee Oh, Michael Magee, Ilya Gekhtman, and Wenyu Pan. We are extremely thankful to Hee Oh for her patient and careful reading of the manuscript. Her suggestions have improved the quality of our exposition.

\section{Preliminaries}
\label{sec:Preliminaries}
We will first introduce the basic setup and fix notations for the rest of the paper.

Let $\mathbb H^n$ be the $n$-dimensional hyperbolic space for $n \geq 2$, i.e., the unique complete simply connected $n$-dimensional Riemannian manifold with constant negative sectional curvature. We denote by $\langle \cdot, \cdot\rangle$ and $\|\cdot\|$ the inner product and norm respectively on any tangent space of $\mathbb H^n$ induced by the hyperbolic metric. Similarly, we denote by $d$ the distance function on $\mathbb H^n$ induced by the hyperbolic metric. Let $G = \SO(n, 1)^\circ \cong \Isom^+(\mathbb H^n)$ and $\Gamma < G$ be a Zariski dense torsion-free discrete subgroup. Fix a reference point $o \in \mathbb H^n$ and a reference tangent vector $v_o \in \T^1(\mathbb H^n)$ at $o$. Then, we have the stabilizer subgroups $K = \Stab_G(o)$ and $M = \Stab_G(v_o) < K$. Note that $K \cong \SO(n)$ and it is a maximal compact subgroup of $G$ and $M \cong \SO(n - 1)$. Our base hyperbolic manifold is $X = \Gamma \backslash \mathbb H^n \cong \Gamma \backslash G/K$, its unit tangent bundle is $\T^1(X) \cong \Gamma \backslash G/M$ and its frame bundle is $\F(X) \cong \Gamma \backslash G$ which is a principal $\SO(n)$-bundle over $X$ and a principal $\SO(n - 1)$-bundle over $\T^1(X)$. Let $A = \{a_t: t \in \mathbb R\} < G$ be a one parameter subgroup of semisimple elements, where $C_G(A) = AM$, parametrized such that its right translation action on $G/M$ and $G$ corresponds to the geodesic flow and the frame flow respectively. We choose a left $G$-invariant and right $K$-invariant Riemannian metric on $G$ \cite{Sas58,Mok78} which descends down to the previous hyperbolic metric on $\mathbb H^n \cong G/K$, and again use the notations $\langle \cdot, \cdot\rangle$, $\|\cdot\|$, and $d$ on $G$ and any of its quotient spaces.

\subsection{Limit set}
Let $\partial_\infty(\mathbb H^n)$ denote the boundary at infinity and $\overline{\mathbb H^n} = \mathbb H^n \cup \partial_\infty(\mathbb H^n)$ denote the compactification of $\mathbb H^n$.

\begin{definition}[Limit set]
The \emph{limit set} of $\Gamma$ is the set of limit points $\Lambda(\Gamma) = \lim(\Gamma o) \subset \partial_\infty(\mathbb H^n) \subset \overline{\mathbb H^n}$.
\end{definition}

\begin{definition}[Critical exponent]
The \emph{critical exponent} $\delta_\Gamma$ of $\Gamma$ is the abscissa of convergence of the Poincar\'{e} series $\mathscr{P}_\Gamma(s) = \sum_{\gamma \in \Gamma} e^{-s d(o, \gamma o)}$.
\end{definition}

\begin{remark}
It is well known that the above definitions are independent of the choice of $o \in \mathbb H^n$.
\end{remark}

\begin{definition}[Convex cocompact]
A torsion-free discrete subgroup $\Gamma < G$ is called \emph{convex cocompact} if the \emph{convex core} $\Core(X) = \Gamma \backslash \Hull(\Lambda(\Gamma)) \subset X$, where $\Hull$ denotes the convex hull, is compact.
\end{definition}

We assume that $\Gamma$ is convex cocompact in the entire paper.

\begin{remark}
In our case, $\delta_\Gamma \in (0, n - 1]$ and coincides with the Hausdorff dimension of $\Lambda(\Gamma)$.
\end{remark}

\subsection{Patterson--Sullivan density}
\label{subsec:Patterson--SullivanDensity}
Let $\{\mu^{\mathrm{PS}}_x: x \in \mathbb H^n\}$ denote the \emph{Patterson--Sullivan density} of $\Gamma$ \cite{Pat76,Sul79}, i.e., the set of finite Borel measures on $\partial_\infty(\mathbb H^n)$ supported on $\Lambda(\Gamma)$ such that
\begin{enumerate}
\item	$g_*\mu^{\mathrm{PS}}_x = \mu^{\mathrm{PS}}_{gx}$ for all $g \in \Gamma$ and $x \in \mathbb H^n$;
\item	$\frac{d\mu^{\mathrm{PS}}_x}{d\mu^{\mathrm{PS}}_y}(\xi) = e^{\delta_\Gamma \beta_{\xi}(y, x)}$ for all $\xi \in \partial_\infty(\mathbb H^n)$ and $x, y \in \mathbb H^n$
\end{enumerate}
where $\beta_{\xi}$ denotes the \emph{Busemann function} at $\xi \in \partial_\infty(\mathbb H^n)$ defined by $\beta_{\xi}(y, x) = \lim_{t \to \infty} (d(\xi(t), y) - d(\xi(t), x))$, where $\xi: \mathbb R \to \mathbb H^n$ is any geodesic such that $\lim_{t \to \infty} \xi(t) = \xi$. We allow tangent vector arguments for the Busemann function as well in which case we will use their basepoints in the definition. Since $\Gamma$ is convex cocompact, for all $x \in \mathbb H^n$, the measure $\mu^{\mathrm{PS}}_x$ is the $\delta_\Gamma$-dimensional Hausdorff measure on $\partial_\infty(\mathbb H^n)$ supported on $\Lambda(\Gamma)$ corresponding to the spherical metric on $\partial_\infty(\mathbb H^n)$ with respect to $x$, up to scalar multiples.

\subsection{Bowen--Margulis--Sullivan measure}
\label{subsec:BMS_Measure}
For all $u \in \T^1(\mathbb H^n)$, let $u^+$ and $u^-$ denote its forward and backward limit points. Using the Hopf parametrization via the homeomorphism $G/M \cong \T^1(\mathbb H^n) \to \{(u^+, u^-) \in \partial_\infty(\mathbb H^n) \times \partial_\infty(\mathbb H^n): u^+ \neq u^-\} \times \mathbb R$ defined by $u \mapsto (u^+, u^-, t = \beta_{u^-}(o, u))$, we define the \emph{Bowen--Margulis--Sullivan (BMS) measure} $\mathsf{m}$ on $G/M$ \cite{Mar04,Bow71,Kai90} by
\begin{align*}
d\mathsf{m}(u) = e^{\delta_\Gamma \beta_{u^+}(o, u)} e^{\delta_\Gamma \beta_{u^-}(o, u)} \, d\mu^{\mathrm{PS}}_o(u^+) \, d\mu^{\mathrm{PS}}_o(u^-) \, dt.
\end{align*}
Note that this definition only depends on $\Gamma$ and not on the choice of reference point $o \in \mathbb H^n$. Moreover, $\mathsf{m}$ is left $\Gamma$-invariant. We now define induced measures on other spaces, all of which we call the BMS measures and denote by $\mathsf{m}$ by abuse of notation. By left $\Gamma$-invariance, $\mathsf{m}$ descends to a measure on $\Gamma \backslash G/M$. We normalize it to a probability measure so that $\mathsf{m}(\Gamma \backslash G/M) = 1$. Since $M$ is compact, we can then use the probability Haar measure on $M$ to lift $\mathsf{m}$ to a right $M$-invariant measure on $\Gamma \backslash G$. It can be checked that the BMS measures are invariant with respect to the geodesic flow or the frame flow as appropriate, i.e., they are right $A$-invariant. We denote the right $A$-invariant subset $\Omega = \supp(\mathsf{m}) \subset \Gamma \backslash G/M$ which is compact since $\Gamma$ is convex cocompact.

\section{Coding the geodesic flow}
\label{sec:CodingTheGeodesicFlow}
In this section, we will review the required background for Markov sections, symbolic dynamics, and thermodynamic formalism.

\subsection{Markov sections}
\label{subsec:MarkovSections}
We will use a Markov section on $\Omega \subset \T^1(X) \cong \Gamma \backslash G/M$, as developed by Bowen and Ratner \cite{Bow70,Rat73}, to obtain a symbolic coding of the dynamical system at hand. Recall that the geodesic flow on $\T^1(X)$ is Anosov. Let $W^{\mathrm{su}}(w) \subset \T^1(X)$ and $W^{\mathrm{ss}}(w) \subset \T^1(X)$ denote the leaves through $w \in \T^1(X)$ of the strong unstable and strong stable foliations, and $W_{\epsilon}^{\mathrm{su}}(w) \subset W^{\mathrm{su}}(w)$ and $W_{\epsilon}^{\mathrm{ss}}(w) \subset W^{\mathrm{ss}}(w)$ denote the open balls of radius $\epsilon > 0$ with respect to the induced distance functions $d_{\mathrm{su}}$ and $d_{\mathrm{ss}}$, respectively. We use similar notations for the weak unstable and weak stable foliations by replacing `su' with `wu' and `ss' with `ws' respectively. The Anosov property provides a constant $C_{\mathrm{Ano}} > 0$ such that for all $w \in \T^1(X)$, we have
\begin{align*}
d_{\mathrm{su}}(ua_{-t}, va_{-t}) &\leq C_{\mathrm{Ano}}e^{-t}d_{\mathrm{su}}(u, v); & d_{\mathrm{ss}}(ua_t, va_t) &\leq C_{\mathrm{Ano}}e^{-t}d_{\mathrm{ss}}(u, v)
\end{align*}
for all $t \geq 0$, for all $u, v \in W^{\mathrm{su}}(w)$ or for all $u, v \in W^{\mathrm{ss}}(w)$ respectively. We recall that there exist $\epsilon_0, \epsilon_0' > 0$ such that for all $w \in \T^1(X)$, $u \in W_{\epsilon_0}^{\mathrm{wu}}(w)$, and $s \in W_{\epsilon_0}^{\mathrm{ss}}(w)$, there exists a unique intersection denoted by
\begin{align}
\label{eqn:BracketOfUandS}
[u, s] = W_{\epsilon_0'}^{\mathrm{ss}}(u) \cap W_{\epsilon_0'}^{\mathrm{wu}}(s)
\end{align}
and moreover, $[\cdot, \cdot]$ defines a homeomorphism from $W_{\epsilon_0}^{\mathrm{wu}}(w) \times W_{\epsilon_0}^{\mathrm{ss}}(w)$ onto its image \cite{Rat73}. Subsets $U \subset W_{\epsilon_0}^{\mathrm{su}}(w) \cap \Omega$ and $S \subset W_{\epsilon_0}^{\mathrm{ss}}(w) \cap \Omega$ for some $w \in \Omega$ are called \emph{proper} if $U = \overline{\interior(U)}$ and $S = \overline{\interior(S)}$, where the interiors and closures are taken in the topology of $W^{\mathrm{su}}(w) \cap \Omega$ and $W^{\mathrm{ss}}(w) \cap \Omega$ respectively. For any $\hat{\delta} > 0$ and proper sets $U \subset W_{\epsilon_0}^{\mathrm{su}}(w) \cap \Omega$ and $S \subset W_{\epsilon_0}^{\mathrm{ss}}(w) \cap \Omega$ containing some $w \in \Omega$, we call
\begin{align*}
R = [U, S] = \{[u, s] \in \Omega: u \in U, s \in S\} \subset \Omega
\end{align*}
a \emph{rectangle of size $\hat{\delta}$} if $\diam_{d_{\mathrm{su}}}(U), \diam_{d_{\mathrm{ss}}}(S) \leq \hat{\delta}$, and we call $w$ the \emph{center} of $R$. For any rectangle $R = [U, S]$, we generalize the notation and define $[v_1, v_2] = [u_1, s_2]$ for all $v_1 = [u_1, s_1] \in R$ and $v_2 = [u_2, s_2] \in R$.

\begin{definition}[Complete set of rectangles]
Let $\hat{\delta} > 0$ and $N \in \mathbb N$. A set $\mathcal{R} = \{R_1, R_2, \dotsc, R_N\} = \{[U_1, S_1], [U_2, S_2], \dotsc, [U_N, S_N]\}$ consisting of rectangles in $\Omega$ is called a \emph{complete set of rectangles of size $\hat{\delta}$} if
\begin{enumerate}
\item \label{itm:MarkovProperty1} $R_j \cap R_k = \varnothing$ for all $1 \leq j, k \leq N$ with $j \neq k$;
\item \label{itm:MarkovProperty2} $\diam_{d_{\mathrm{su}}}(U_j), \diam_{d_{\mathrm{ss}}}(S_j) \leq \hat{\delta}$ for all $1 \leq j \leq N$;
\item \label{itm:MarkovProperty3} $\Omega = \bigcup_{j = 1}^N \bigcup_{t \in [0, \hat{\delta}]} R_j a_t$.
\end{enumerate}
\end{definition}

Henceforth, we fix
\begin{align}
\label{eqn:DeltaHatCondition}
0 < \hat{\delta} < \min(1, \epsilon_0, \epsilon_0')
\end{align}
where $\epsilon_0$ and $\epsilon_0'$ are from \cref{eqn:BracketOfUandS}. We also fix
\begin{align*}
\mathcal{R} = \{R_1, R_2, \dotsc, R_N\} = \{[U_1, S_1], [U_2, S_2], \dotsc, [U_N, S_N]\}
\end{align*}
to be a complete set of rectangles of size $\hat{\delta}$ in $\Omega$. We denote
\begin{align*}
R &= \bigsqcup_{j = 1}^N R_j; & U &= \bigsqcup_{j = 1}^N U_j.
\end{align*}
We introduce the distance function $d$ on $U$ defined by
\begin{align*}
d(u, v) =
\begin{cases}
d_{\mathrm{su}}(u, v), & u, v \in U_j \text{ for some } 1 \leq j \leq N \\
1, & \text{otherwise.}
\end{cases}
\qquad
\text{for all $u, v \in U$}.
\end{align*}
We will use $d_{\mathrm{su}}$ whenever further clarity is required. Denote $\tau: R \to \mathbb R$ to be the first return time map defined by
\begin{align*}
\tau(u) = \inf\{t > 0: ua_t \in R\} \qquad \text{for all $u \in R$}.
\end{align*}
Note that $\tau$ is \emph{constant} on $[u, S_j]$ for all $u \in U_j$ and $1 \leq j \leq N$. Denote $\mathcal{P}: R \to R$ to be the Poincar\'{e} first return map defined by
\begin{align*}
\mathcal{P}(u) = ua_{\tau(u)} \qquad \text{for all $u \in R$}.
\end{align*}
Let $\sigma = (\proj_U \circ \mathcal{P})|_U: U \to U$ be its projection where $\proj_U: R \to U$ is the projection defined by $\proj_U([u, s]) = u$ for all $[u, s] \in R$. Define the \emph{cores}
\begin{align*}
\hat{R} &= \{u \in R: \mathcal{P}^k(u) \in \interior(R) \text{ for all } k \in \mathbb Z\}; \\
\hat{U} &= \{u \in U: \sigma^k(u) \in \interior(U) \text{ for all } k \in \mathbb Z_{\geq 0}\}
\end{align*}
which are both residual subsets (complements of meager sets) of $R$ and $U$ respectively.

\begin{definition}[Markov section]
Let $\hat{\delta} > 0$ and $N \in \mathbb N$. We call a complete set of rectangles $\mathcal{R}$ of size $\hat{\delta}$ a \emph{Markov section} if in addition to \cref{itm:MarkovProperty1,itm:MarkovProperty2,itm:MarkovProperty3}, the following property
\begin{enumerate}
\setcounter{enumi}{3}
\item $[\interior(U_k), \mathcal{P}(u)] \subset \mathcal{P}([\interior(U_j), u])$ and $\mathcal{P}([u, \interior(S_j)]) \subset [\mathcal{P}(u), \interior(S_k)]$ for all $u \in R$ such that $u \in \interior(R_j) \cap \mathcal{P}^{-1}(\interior(R_k)) \neq \varnothing$, for all $1 \leq j, k \leq N$
\end{enumerate}
called the \emph{Markov property}, is satisfied. This can be understood pictorially in \cref{fig:MarkovProperty}.
\end{definition}

\begin{figure}
\definecolor{front}{RGB}{31, 38, 62}
\definecolor{middle}{RGB}{63,91,123}
\definecolor{back}{RGB}{98,145,166}
\begin{tikzpicture}
\coordinate (E) at (2.7,1.5);
\coordinate (F) at (2.7,3);
\coordinate (G) at (4.8,3);
\coordinate (H) at (4.8,1.5);

%Markov rectangle EFGH
\draw[back, fill = back, fill opacity=0.7] (E) to[out=80,in=-80] (F) to[out=2,in=178] (G) to[out=-80,in=80] (H) to[out=178,in=2] (E) -- cycle;

\coordinate (A) at (0.2,0.2);
\coordinate (B) at (0.2,1.8);
\coordinate (C) at (2.8,1.8);
\coordinate (D) at (2.8,0.2);

%Markov rectangle ABCD
\draw[middle, fill = middle, fill opacity=0.7] (A) to[out=80,in=-80] (B) to[out=2,in=178] (C) to[out=-80,in=80] (D) to[out=178,in=2] (A) -- cycle;

\coordinate (AB_Mid) at (0.28,1);
\coordinate (CD_Mid) at (2.88,1);

\coordinate (Uj_Label) at (3.8, 0);
\coordinate (Sj_Label) at (3.8, -0.5);

%w of Markov rectangle ABCD
\node[above right, font=\tiny] at (1.5, 0.95) {$w_j$};

%U_j of Markov rectangle ABCD
\draw[middle,thick] (AB_Mid) to[out=2,in=178] (CD_Mid);
\node[right, font=\tiny] at (Uj_Label) {$U_j \subset W^{\mathrm{su}}(w_j)$};
\draw[->, very thin] (Uj_Label) to[out=180,in=-75] ($0.8*(CD_Mid) + 0.2*(AB_Mid) + (0, -0.05)$);

\coordinate (BC_Mid) at (1.5,1.825);
\coordinate (AD_Mid) at (1.5,0.225);

%S_j of Markov rectangle ABCD
\draw[middle,thick] (AD_Mid) to[out=80,in=-80] (BC_Mid);
\node[right, font=\tiny] at (Sj_Label) {$S_j \subset W^{\mathrm{ss}}(w_j)$};
\draw[->, very thin] (Sj_Label) to[out=180,in=-15] ($0.8*(AD_Mid) + 0.2*(BC_Mid) + (0.13, 0)$);

\coordinate (A') at (1.4,1.6);
\coordinate (B') at (1.4,2.4);
\coordinate (C') at (7,2.4);
\coordinate (D') at (7,1.6);

%Forward flow of ABCD
\draw[middle,dashed] (A') to[out=80,in=-80] (B') to[out=2,in=178] (C') to[out=-80,in=80] (D') to[out=178,in=2] (A') -- cycle;

\coordinate (A'') at (-2,-2);
\coordinate (B'') at (-2,2);
\coordinate (C'') at (-0.5,2);
\coordinate (D'') at (-0.5,-2);

%Backward flow of ABCD
\draw[middle,dashed] (A'') to[out=80,in=-80] (B'') to[out=2,in=178] (C'') to[out=-80,in=80] (D'') to[out=178,in=2] (A'') -- cycle;

%Flow lines
\draw[middle,dashed] (A'') to[out=60,in=215] (A) to[out=35,in=240] (A');
\draw[middle,dashed] (B'') to[out=-15,in=190] (B) to[out=10,in=215] (B');
\draw[middle,dashed] (C'') to[out=-15,in=185] (C) to[out=5,in=180] (C');
\draw[middle,dashed] (D'') to[out=40,in=210] (D) to[out=30,in=190] (D');

\coordinate (I) at (-2.3,-1);
\coordinate (J) at (-2.3,1.1);
\coordinate (K) at (0.7,1.1);
\coordinate (L) at (0.7,-1);

%Markov rectangle IJKL
\draw[front, fill = front, fill opacity=0.7] (I) to[out=80,in=-80] (J) to[out=2,in=178] (K) to[out=-80,in=80] (L) to[out=178,in=2] (I) -- cycle;
\end{tikzpicture}
\caption{The Markov property.}
\label{fig:MarkovProperty}
\end{figure}
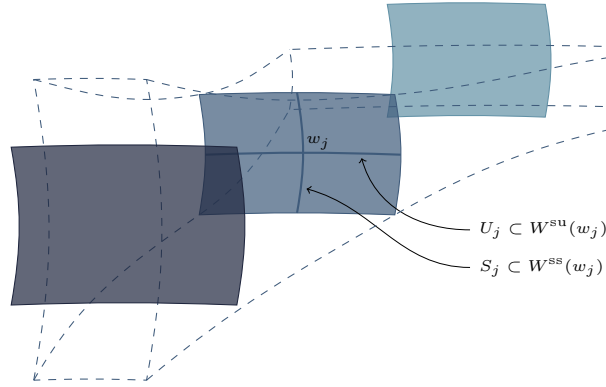

The existence of Markov sections of arbitrarily small size for Anosov flows was proved by Bowen and Ratner \cite{Bow70,Rat73}. Thus, we assume henceforth that $\mathcal{R}$ is a Markov section.

\subsection{Symbolic dynamics}
\label{subsec:SymbolicDynamics}
Let $\mathcal A = \{1, 2, \dotsc, N\}$ be the \emph{alphabet} for the coding corresponding to the Markov section. Define the $N \times N$ \emph{transition matrix} $T$ by
\begin{align*}
T_{j, k} =
\begin{cases}
1, & \interior(R_j) \cap \mathcal{P}^{-1}(\interior(R_k)) \neq \varnothing \\
0, & \text{otherwise}
\end{cases}
\qquad
\text{for all $1 \leq j, k \leq N$}.
\end{align*}
The transition matrix $T$ is \emph{topologically mixing} \cite[Theorem 4.3]{Rat73}, i.e., there exists $N_T \in \mathbb N$ such that all the entries of $T^{N_T}$ are positive. This definition is equivalent to the one in \cite{Rat73} in the setting of Markov sections. Define the spaces of bi-infinite and infinite \emph{admissible sequences} by
\begin{align*}
\Sigma &= \{(\dotsc, x_{-1}, x_0, x_1, \dotsc) \in \mathcal A^{\mathbb Z}: T_{x_j, x_{j + 1}} = 1 \text{ for all } j \in \mathbb Z\}; \\
\Sigma^+ &= \{(x_0, x_1, \dotsc) \in \mathcal A^{\mathbb Z_{\geq 0}}: T_{x_j, x_{j + 1}} = 1 \text{ for all } j \in \mathbb Z_{\geq 0}\}
\end{align*}
respectively. We will also use the term \emph{admissible sequences} for finite sequences in the natural way. For any $\theta \in (0, 1)$, we can endow $\Sigma$ with the distance function $d_\theta$ defined by $d_\theta(x, y) = \theta^{\inf\{|j| \in \mathbb Z_{\geq 0}: x_j \neq y_j\}}$ for all $x, y \in \Sigma$. We can similarly endow $\Sigma^+$ with a distance function which we also denote by $d_\theta$.

\begin{definition}[Cylinder]
For all $k \in \mathbb Z_{\geq 0}$ and for all admissible sequences $x = (x_0, x_1, \dotsc, x_k)$, we define the corresponding \emph{cylinder} to be
\begin{align*}
\mathtt{C}[x] = \{u \in U: \sigma^j(u) \in \interior(U_{x_j}) \text{ for all } 0 \leq j \leq k\}
\end{align*}
with \emph{length} $\len(\mathtt{C}[x]) = k$. We will denote cylinders simply by $\mathtt{C}$ (or other typewriter style letters) when we do not need to specify the corresponding admissible sequence.
\end{definition}

Although $\sigma$ and $\tau$ are not even continuous, we note that for all admissible pairs $(j, k)$, the restricted maps $\sigma|_{\mathtt{C}[j, k]}: \mathtt{C}[j, k] \to \interior(U_k)$, $(\sigma|_{\mathtt{C}[j, k]})^{-1}: \interior(U_k) \to \mathtt{C}[j, k]$, and $\tau|_{\mathtt{C}[j, k]}: \mathtt{C}[j, k] \to \mathbb R$ are Lipschitz in our setting.

By a slight abuse of notation, let $\sigma$ also denote the shift map on $\Sigma$ or $\Sigma^+$. There exist natural continuous surjections $\zeta: \Sigma \to R$ and $\zeta^+: \Sigma^+ \to U$ defined by $\zeta(x) = \bigcap_{j = -\infty}^\infty \overline{\mathcal{P}^{-j}(\interior(R_{x_j}))}$ for all $x \in \Sigma$ and $\zeta^+(x) = \bigcap_{j = 0}^\infty \overline{\sigma^{-j}(\interior(U_{x_j}))}$ for all $x \in \Sigma^+$. Define $\hat{\Sigma} = \zeta^{-1}(\hat{R})$ and $\hat{\Sigma}^+ = (\zeta^+)^{-1}(\hat{U})$. Then the restrictions $\zeta|_{\hat{\Sigma}}: \hat{\Sigma} \to \hat{R}$ and $\zeta^+|_{\hat{\Sigma}^+}: \hat{\Sigma}^+ \to \hat{U}$ are bijective and satisfy $\zeta|_{\hat{\Sigma}} \circ \sigma|_{\hat{\Sigma}} = \mathcal{P}|_{\hat{R}} \circ \zeta|_{\hat{\Sigma}}$ and $\zeta^+|_{\hat{\Sigma}^+} \circ \sigma|_{\hat{\Sigma}^+} = \sigma|_{\hat{U}} \circ \zeta^+|_{\hat{\Sigma}^+}$.

For $\theta \in (0, 1)$ sufficiently close to $1$, the maps $\zeta$ and $\zeta^+$ are Lipschitz \cite[Lemma 2.2]{Bow73} with some Lipschitz constant $C_\theta > 0$. We now fix $\theta$ to be any such constant. Let $C^{\Lip(d_\theta)}(\Sigma, \mathbb R)$ denote the space of Lipschitz functions $f: \Sigma \to \mathbb R$. We use similar notations for domain space $\Sigma^+$ or target space $\mathbb C$.

Since $(\tau \circ \zeta)|_{\hat{\Sigma}}$ and $(\tau \circ \zeta^+)|_{\hat{\Sigma}^+}$ are Lipschitz, there exist unique Lipschitz extensions $\tau_\Sigma: \Sigma \to \mathbb R$ and $\tau_{\Sigma^+}: \Sigma^+ \to \mathbb R$ respectively. Note that the resulting maps are distinct from $\tau \circ \zeta$ and $\tau \circ \zeta^+$ because they may differ precisely on $x \in \Sigma$ for which $\proj_U(\zeta(x)) \in \partial(\mathtt{C})$ and $x \in \Sigma^+$ for which $\zeta^+(x) \in \partial(\mathtt{C})$ respectively, for some cylinder $\mathtt{C} \subset U$ with $\len(\mathtt{C}) = 1$. Then the previous properties extend to $\zeta(\sigma(x)) = \zeta(x)a_{\tau_\Sigma(x)}$ for all $x \in \Sigma$ and $\zeta^+(\sigma(x)) = \proj_U(\zeta^+(x)a_{\tau_{\Sigma^+}(x)})$ for all $x \in \Sigma^+$.

\subsection{Thermodynamics}
\label{subsec:Thermodynamics}
\begin{definition}[Pressure]
\label{def:Pressure}
For all $f \in C^{\Lip(d_\theta)}(\Sigma, \mathbb R)$, called the \emph{potential}, the \emph{pressure} is defined by
\begin{align*}
\Pr_\sigma(f) = \sup_{\nu \in \mathcal{M}^1_\sigma(\Sigma)}\left\{\int_\Sigma f \, d\nu + h_\nu(\sigma)\right\}
\end{align*}
where $\mathcal{M}^1_\sigma(\Sigma)$ is the set of $\sigma$-invariant Borel probability measures on $\Sigma$ and $h_\nu(\sigma)$ is the measure theoretic entropy of $\sigma$ with respect to $\nu$.
\end{definition}

For all $f \in C^{\Lip(d_\theta)}(\Sigma, \mathbb R)$, there is in fact a unique $\sigma$-invariant Borel probability measure $\nu_f$ on $\Sigma$ which attains the supremum in \cref{def:Pressure} called the \emph{$f$-equilibrium state} \cite[Theorems 2.17 and 2.20]{Bow08} and it satisfies $\nu_f(\hat{\Sigma}) = 1$ \cite[Corollary 3.2]{Che02}. In particular, we will consider the probability measure $\nu_{-\delta_\Gamma\tau_\Sigma}$ on $\Sigma$ which we will denote simply by $\nu_\Sigma$ and has corresponding pressure $\Pr_\sigma(-\delta_\Gamma\tau_\Sigma) = 0$. According to above, $\nu_\Sigma(\hat{\Sigma}) = 1$. Define the corresponding probability measure $\nu_R = \zeta_*(\nu_\Sigma)$ on $R$ and note that $\nu_R(\hat{R}) = 1$. Now consider the suspension space $R^\tau = (R \times \mathbb R_{\geq 0})/\mathord{\sim}$ where $\sim$ is the equivalence relation on $R \times \mathbb R_{\geq 0}$ defined by $(u, t + \tau(u)) \sim (\mathcal{P}(u), t)$. Then we have a bijection $R^\tau \to \Omega$ defined by $(u, t) \mapsto ua_t$. We can define the measure $\nu^\tau$ on $R^\tau$ as the product measure $\nu_R \times m^{\mathrm{Leb}}$ on $\{(u, t) \in R \times \mathbb R_{\geq 0}: 0 \leq t < \tau(u)\}$. Then using the aforementioned bijection we have the pushforward measure which, by abuse of notation, we also denote by $\nu^\tau$ on $\T^1(X)$ supported on $\Omega$. By \cite{Sul84} and \cite[Theorem 4.4]{Che02}, we have $\mathsf{m} = \frac{\nu^\tau}{\nu_R(\tau)}$ because they are the unique measure of maximal entropy for the geodesic flow on $\T^1(X)$. Finally, we define the probability measure $\nu_U = (\proj_U)_*(\nu_R)$ and note that $\nu_U(\hat{U}) = 1$ and $\nu_U(\tau) = \nu_R(\tau)$.

\section{Holonomy and representation theory}
\label{sec:HolonomyAndRepresentationTheory}
In this section, we define holonomy which is required in addition to the Markov section to deal with the frame flow. Since the holonomy is $M$-valued, we naturally need to consider $L^2(M, \mathbb C)$ and so we also cover the required representation theory.

We do not have a Markov section available for the frame flow. Thus, similar to $\tau$, we need a map $\vartheta$ which ``keeps track of the $M$-coordinate''. We first require an appropriate choice of section $F$ on $R$ of the frame bundle $\F(X)$ over $\T^1(X)$. Let $w_j$ be the center of $R_j$ for all $j \in \mathcal{A}$. For convenience later on, we will actually define a \emph{smooth} section
\begin{align*}
F: \bigsqcup_{j = 1}^N [W_{\epsilon_0}^{\mathrm{su}}(w_j), W_{\epsilon_0}^{\mathrm{ss}}(w_j)] \to \F(X)
\end{align*}
where without loss of generality we assume $\epsilon_0$ is sufficiently small so that the union is indeed a disjoint union. Define $N^+ < G$ and $N^- < G$ to be the expanding and contracting horospherical subgroups, i.e.,
\begin{align*}
N^\pm = \Big\{n^\pm \in G: \lim_{t \to \pm\infty} a_tn^\pm a_{-t} = e\Big\}.
\end{align*}
First we choose arbitrary frames $F(w_j) \in \F(X)$ based at the tangent vector $w_j \in \T^1(X)$ for all $j \in \mathcal{A}$. Then we extend the section $F$ such that for all $j \in \mathcal{A}$ and $u, u' \in W_{\epsilon_0}^{\mathrm{su}}(w_j)$, we have that the frames $F(u)$ and $F(u')$ are backwards asymptotic, i.e., $\lim_{t \to -\infty} d(F(u)a_t, F(u')a_t) = 0$. Then we must have $F(u') = F(u)n^+$ for some unique $n^+ \in N^+$. We again extend the section $F$ such that for all $j \in \mathcal{A}$, $u \in W_{\epsilon_0}^{\mathrm{su}}(w_j)$, and $s, s' \in W_{\epsilon_0}^{\mathrm{ss}}(w_j)$, we have that the frames $F([u, s])$ and $F([u, s'])$ are forwards asymptotic, i.e., $\lim_{t \to +\infty} d(F([u, s])a_t, F([u, s'])a_t) = 0$. Then we must have $F([u, s']) = F([u, s])n^-$ for some unique $n^- \in N^-$. This completes the construction.

\begin{definition}[Holonomy]
The \emph{holonomy} is a map $\vartheta: R \to M$ such that for all $u \in R$, we have $F(u)a_{\tau(u)} = F(\mathcal{P}(u))\vartheta(u)^{-1}$.
\end{definition}

Just as $\tau$ is constant on the strong stable leaves of the rectangles, the following lemma shows that the same is true for $\vartheta$. This allows us to work solely on $U$.

\begin{lemma}
\label{lem:HolonomyConstantOnStrongStableLeaves}
The holonomy $\vartheta$ is constant on $[u, S_j]$ for all $u \in U_j$ and $j \in \mathcal{A}$.
\end{lemma}

\begin{proof}
Let $j \in \mathcal{A}$ and $u \in U_j$. Let $s \in S_j$ and $u' = [u, s]$. Recall that $F(u') = F(u)n^-$ for some $n^- \in N^-$. From the definition of the holonomy map, we have $F(\mathcal{P}(u)) = F(u)a_{\tau(u)}\vartheta(u)$ and $F(\mathcal{P}(u')) = F(u')a_{\tau(u')}\vartheta(u') = F(u)n^-a_{\tau(u)}\vartheta(u')$ since $\tau(u') = \tau(u)$. Let $F(u) = \Gamma g \in \Gamma \backslash G$. Using left $G$-invariance and right $K$-invariance of the distance function $d$ on $G$, we have
\begin{align*}
d(F(\mathcal{P}(u))a_t, F(\mathcal{P}(u'))a_t) &= d(F(u)a_{\tau(u) + t}\vartheta(u), F(u)n^-a_{\tau(u) + t}\vartheta(u')) \\
&= d(ga_{\tau(u) + t}\vartheta(u), gn^-a_{\tau(u) + t}\vartheta(u')) \\
&= d(\vartheta(u), a_{-(\tau(u) + t)}n^-a_{\tau(u) + t}\vartheta(u')) \\
&\geq d(\vartheta(u), \vartheta(u')) - d(\vartheta(u'), a_{-(\tau(u) + t)}n^-a_{\tau(u) + t}\vartheta(u')) \\
&= d(\vartheta(u), \vartheta(u')) - d(e, a_{-(\tau(u) + t)}n^-a_{\tau(u) + t})
\end{align*}
for all $t \geq 0$. The second equality holds assuming that $\epsilon_0$ is sufficiently small without loss of generality. Then the equations $\lim_{t \to +\infty} d(F(\mathcal{P}(u))a_t, F(\mathcal{P}(u'))a_t) = 0$ and $\lim_{t \to +\infty} d(e, a_{-(\tau(u) + t)}n^-a_{\tau(u) + t}) = 0$ from definitions imply $d(\vartheta(u), \vartheta(u')) = 0$. Thus $\vartheta(u') = \vartheta(u)$.
\end{proof}

Denote $\Omega_{\F} = \supp(\mathsf{m}) \subset \Gamma \backslash G$ which is compact since $\Gamma$ is convex cocompact. Define $R^\vartheta \subset \F(X)$ to be the subset of frames over $R$ and similarly define $U^\vartheta$. Via the section $F$, we have the natural identifications $R^\vartheta \cong R \times M$ and $U^\vartheta \cong U \times M$. We define the measure $\nu_{R^\vartheta}$ on $R^\vartheta$ simply by lifting the measure $\nu_{R}$ using the probability Haar measure on $M$. Using the holonomy $\vartheta$, we can define the suspension space $R^{\vartheta, \tau} = R^\vartheta \times \mathbb R_{\geq 0}/{\sim}$ where $\sim$ is the equivalence relation on $R^\vartheta \times \mathbb R_{\geq 0}$ defined by $(u, m, t + \tau(u)) \sim (\mathcal{P}(u), \vartheta(u)^{-1}m, t)$. Like $\nu^\tau$, we can now define the measure $\nu^{\vartheta, \tau}$ on $R^{\vartheta, \tau}$. As in \cref{subsec:Thermodynamics}, we can use the natural bijection $R^{\vartheta, \tau} \to \Omega_{\F}$ defined by $(u, m, t) \mapsto F(u)a_t m$, to obtain the pushforward measure which, by abuse of notation, we also denote by $\nu^{\vartheta, \tau}$ on $\F(X)$ supported on $\Omega_{\F}$. Then $\mathsf{m} = \frac{\nu^{\vartheta, \tau}}{\nu_R(\tau)}$.

We need to deal with the function space $C(U^\vartheta, \mathbb C)$. We note that
\begin{align*}
C(U^\vartheta, \mathbb C) \cong C(U \times M, \mathbb C) \cong C(U, C(M, \mathbb C)) \subset C(U, L^2(M, \mathbb C)).
\end{align*}
Define $\varrho: M \to \U(L^2(M, \mathbb C))$ to be the unitary left regular representation, i.e., $\varrho(h)(\phi)(m) = \phi(h^{-1}m)$ for all $m \in M$, $\phi \in L^2(M, \mathbb C)$, and $h \in M$. Denote the unitary dual of $M$ by $\widehat{M}$. Denote the trivial irreducible representation by $1 \in \widehat{M}$. Define $\widehat{M}_0 = \widehat{M} \setminus \{1\}$. By the Peter--Weyl theorem, we obtain an orthogonal Hilbert space decomposition
\begin{align*}
L^2(M, \mathbb C) = \operatorname*{\widehat{\bigoplus}}_{\rho \in \widehat{M}} V_\rho^{\oplus \dim(\rho)}
\end{align*}
corresponding to the decomposition $\varrho = \operatorname*{\widehat{\bigoplus}}_{\rho \in \widehat{M}} \rho^{\oplus \dim(\rho)}$.

For all $b \in \mathbb R$ and $\rho \in \widehat{M}$, we define the tensored unitary representation $\rho_b: AM \to \U(V_\rho)$ by
\begin{align*}
\rho_b(a_tm)(z) = e^{-ibt}\rho(m)(z) \qquad \text{for all $z \in V_\rho$, $t \in \mathbb R$, and $m \in M$}.
\end{align*}

We introduce some notations related to Lie algebras. We denote Lie algebras corresponding to Lie groups by the corresponding Fraktur letters, e.g., $\mathfrak{a} = \T_e(A), \mathfrak{m} = \T_e(M), \mathfrak{n}^+ = \T_e(N^+)$, and $\mathfrak{n}^- = \T_e(N^-)$. For any unitary representation $\rho: M \to \U(V)$ for some Hilbert space $V$, we denote the differential at $e \in M$ by $d\rho = (d\rho)_e: \mathfrak{m} \to \mathfrak{u}(V)$, and define the norm
\begin{align*}
\|\rho\| = \sup_{\substack{z \in \mathfrak{m}\\ \text{such that } \|z\| = 1}} \|d\rho(z)\|_{\mathrm{op}}
\end{align*}
and similarly for any unitary representation $\rho: AM \to \U(V)$.

\begin{remark}
The norms remain the same if we replace $V_\rho$ with $V_\rho^{\oplus \dim(\rho)}$ since the $M$-action is identical across all components.
\end{remark}

\Cref{lem:LieTheoreticNormBounds} records some useful facts regarding the Lie theoretic norms.

\begin{lemma}
\label{lem:LieTheoreticNormBounds}
For all $b \in \mathbb R$ and $\rho \in \widehat{M}$, we have
\begin{align*}
\sup_{a \in A, m \in M} \sup_{\substack{z \in \T_{am}(AM)\\ \textnormal{such that } \|z\| = 1}} \|(d\rho_b)_{am}(z)\|_{\mathrm{op}} = \|\rho_b\|
\end{align*}
and $\max(|b|, \|\rho\|) \leq \|\rho_b\| \leq |b| + \|\rho\|$.
\end{lemma}

\begin{proof}
Let $b \in \mathbb R$ and $\rho \in \widehat{M}$. We first show the equality. Let $a \in A$, $m \in M$, and $z \in \T_{am}(AM)$ with $\|z\| = 1$. Let $m^{\mathrm{L}}_g: G \to G$ be the left multiplication map by $g \in G$. By the unitarity of $\rho_b$ and the left $G$-invariance of the norm on $G$, we have
\begin{align*}
\|(d\rho_b)_{am}(z)\|_{\mathrm{op}} &= \left\|\left((d\rho_b)_{am} \circ \big(dm_{am}^{\mathrm{L}}\big)_e \circ \big(dm_{(am)^{-1}}^{\mathrm{L}}\big)_{am}\right)(z)\right\|_{\mathrm{op}} \\
&= \left\|\left(\big(dm_{\rho_b(am)}^{\mathrm{L}}\big)_e \circ (d\rho_b)_e \circ \big(dm_{(am)^{-1}}^{\mathrm{L}}\big)_{am}\right)(z)\right\|_{\mathrm{op}} \leq \|\rho_b\|.
\end{align*}
Taking the supremum and recognizing that we have equality for $am = e \in AM$, the first equality follows.

Now we show the inequality. The first part is trivial so we focus on the second part. By construction of the Riemannian metric on $G$, we have $\langle w_1, w_2 \rangle = 0$ for all $w_1 \in \T_g(gA)$, $w_2 \in \T_g(gM)$, and $g \in G$. Hence $AM \cong A \times M$ not only as Lie groups but also as Riemannian manifolds with the canonical product Riemannian metric. Let $z = z_\mathfrak{a} + z_\mathfrak{m} \in \mathfrak{a} \oplus \mathfrak{m}$ with $\|z\|^2 = \|z_\mathfrak{a}\|^2 + \|z_\mathfrak{m}\|^2 = 1$. We have
\begin{align*}
\|d\rho_b(z)\|_{\mathrm{op}} &= \left\|{ib\|z_\mathfrak{a}\|\Id_{\U(V_\rho)}} + d\rho(z_\mathfrak{m})\right\|_{\mathrm{op}} \\
&\leq |b| \cdot \|z_\mathfrak{a}\| + \|\rho\| \cdot \|z_\mathfrak{m}\| \\
&\leq |b| + \|\rho\|
\end{align*}
and so by taking the supremum, the inequality follows.
\end{proof}

It turns out that the \emph{source} of the oscillations needed in Dolgopyat's method is provided by the \emph{local non-integrability condition (LNIC)} which will be introduced in \cref{subsec:LNIC} and the oscillations themselves are \emph{propagated} when $\|\rho_b\|$ is sufficiently large. But this occurs precisely when $|b|$ is sufficiently large or $\rho \in \widehat{M}$ is nontrival. Let $b_0 > 0$ which we fix later. This motivates us to define
\begin{align*}
\widehat{M}_0(b_0) = \{(b, \rho) \in \mathbb R \times \widehat{M}: |b| > b_0 \text{ or } \rho \neq 1\}.
\end{align*}
We fix some related constants. Fix $\delta_{\varrho} = \inf_{b \in \mathbb R, \rho \in \widehat{M}_0} \|\rho_b\| = \inf_{\rho \in \widehat{M}_0} \|\rho\|$. Note that $\delta_{\varrho} > 0$ as $M$ is a compact connected Lie group. Furthermore, we can deduce that $\inf_{(b, \rho) \in \widehat{M}_0(b_0)} \|\rho_b\| \geq \min(b_0, \delta_{\varrho})$. Hence we fix $\delta_{1, \varrho} = \min(1, \delta_{\varrho})$.

The Killing form $B$ on $\mathfrak{m}$ is nondegenerate and negative definite because $M$ is a compact semisimple Lie group. We denote the corresponding inner product and norm on both $\mathfrak{m}$ and $\mathfrak{m}^*$ by $\langle \cdot, \cdot \rangle_B$ and $\|\cdot\|_B$. By construction of the Riemannian metric on $G$, the induced inner product on $\mathfrak{m}$ satisfies $\langle \cdot, \cdot \rangle_B = C_B\langle \cdot, \cdot \rangle$ for some constant $C_B > 0$.

\begin{lemma}
\label{lem:maActionLowerBound}
There exists $\delta > 0$ such that for all $b \in \mathbb R$, $\rho \in \widehat{M}$, and $\omega \in V_\rho^{\oplus \dim(\rho)}$ with $\|\omega\|_2 = 1$, there exists $z \in \mathfrak{a} \oplus \mathfrak{m}$ with $\|z\| = 1$ such that $\|d\rho_b(z)(\omega)\|_2 \geq \delta \|\rho_b\|$.
\end{lemma}

\begin{proof}
Fix $\delta = \frac{1}{2}$ if $M$ is trivial and $\delta = \frac{1}{2 \dim(\mathfrak{m})}$ otherwise. Let $b \in \mathbb R$, $\rho \in \widehat{M}$, and $\omega \in V_\rho^{\oplus \dim(\rho)}$ with $\|\omega\|_2 = 1$. For any $z \in \mathfrak{a} \subset \mathfrak{a} \oplus \mathfrak{m}$ with $\|z\| = 1$, we have
\begin{align*}
\|d\rho_b(z)(\omega)\|_2 = \|ib\omega\|_2 = |b|.
\end{align*}
If $M$ is trivial, then $|b| = \|\rho_b\| \geq \delta \|\rho_b\|$ so the lemma follows. Otherwise, first consider the case $|b| \geq \|\rho\|$. By \cref{lem:LieTheoreticNormBounds}, we have $|b| \geq \frac{1}{2}(|b| + \|\rho\|) \geq \delta \|\rho_b\|$, which proves the lemma in this case.

Now consider the case $|b| \leq \|\rho\|$. By \cref{lem:LieTheoreticNormBounds}, we have $\|\rho_b\| \leq 2\|\rho\|$. Let $\Phi_\rho$ be the set of weights corresponding to the Lie algebra representation $d\rho$ and $\lambda \in \Phi_\rho$ be the highest weight. We first show that $\|\rho\| \leq C_B \|\lambda\|_B$. Let $z \in \mathfrak{m} \subset \mathfrak{a} \oplus \mathfrak{m}$ with $\|z\| = 1$. Then consider the Cartan subalgebra $\mathfrak{h} \subset \mathfrak{m}$ containing $z$, guaranteed by Cartan's maximal tori theorem on $M$, and assume without loss of generality that $\Phi_\rho \subset \mathfrak{h}^*$. We have $d\rho(z)(\omega') = \sum_{\eta \in \Phi_\rho} d\rho(z)(\omega'_\eta) = \sum_{\eta \in \Phi_\rho} \eta(z)\omega'_\eta$ for all $\omega' \in V_\rho$, where we write $\omega' = \sum_{\eta \in \Phi_\rho} \omega'_\eta$ using the weight space decomposition $V_\rho = \bigoplus_{\eta \in \Phi_\rho} V_{\rho, \eta}$. Note that this decomposition is in fact orthogonal because $d\rho: \mathfrak{m} \to \mathfrak{u}(V_\rho)$ is diagonalizable by a unitary operator. Thus, we can use the formula $\|d\rho(z)\|_{\mathrm{op}} = \max_{\eta \in \Phi_\rho} |\eta(z)|$ for the operator norm to get the bound
\begin{align}
\label{eqn:dRhoZ_OperatorNormBound}
\|d\rho(z)\|_{\mathrm{op}} \leq \max_{\eta \in \Phi_\rho} \|\eta\|_B \|z\|_B \leq C_B \|\lambda\|_B
\end{align}
since $\lambda \in \Phi_\rho$ is the highest weight. Since this bound holds for all $z \in \mathfrak{m} \subset \mathfrak{a} \oplus \mathfrak{m}$ with $\|z\| = 1$, taking the supremum gives $\|\rho\| \leq C_B \|\lambda\|_B$ as desired. Hence $\|\rho_b\| \leq 2C_B\|\lambda\|_B$. Now, with respect to the inner product $\langle \cdot, \cdot \rangle_B$, let $(z_1, z_2, \dotsc, z_{\dim(\mathfrak{m})})$ be an orthonormal basis of $\mathfrak{m}$ so that it is its own dual basis. Then the negative Casimir element in the center of the universal enveloping algebra of $\mathfrak{m}$ is given by $\varsigma = \sum_{j = 1}^{\dim(\mathfrak{m})} z_j^2 \in Z(\mathfrak{m}) \subset U(\mathfrak{m})$. Its action on $V_\rho$ via $d\rho$ and hence also via $d\rho_b$ is simply by the scalar $\|\lambda\|_B^2 + 2\langle \lambda, \upsilon \rangle_B$ where $\upsilon = \frac{1}{2}\sum_{\eta \in R^+} \eta$ and $R^+$ is the set of positive roots. But $\langle \lambda, \upsilon \rangle_B \geq 0$ since $\lambda \in \Phi_\rho$ is the highest weight. Thus, we have
\begin{align*}
\sum_{j = 1}^{\dim(\mathfrak{m})} \|d\rho_b(z_j^2)(\omega)\|_2 \geq \|d\rho_b(\varsigma)(\omega)\|_2 \geq \|\lambda\|_B^2.
\end{align*}
Hence, there exists $z_0 \in \{z_1, z_2, \dotsc, z_{\dim(\mathfrak{m})}\}$ such that $\|d\rho_b(z_0^2)(\omega)\|_2 \geq \frac{\|\lambda\|_B^2}{\dim(\mathfrak{m})}$. Using $\|z_0\|_B = 1$ and a similar bound as in \cref{eqn:dRhoZ_OperatorNormBound}, we have $\|d\rho_b(z_0)(\omega)\|_2 \geq \frac{\|\lambda\|_B}{\dim(\mathfrak{m})}$. Let $z = \frac{z_0}{\|z_0\|} \in \mathfrak{m} \subset \mathfrak{a} \oplus \mathfrak{m}$ so that $\|z\| = 1$. Along with the above bound $\|\rho_b\| \leq 2C_B\|\lambda\|_B$, we have
\begin{align*}
\|d\rho_b(z)(\omega)\|_2 \geq \frac{\|\lambda\|_B}{\dim(\mathfrak{m})\|z_0\|} \geq \frac{1}{2 \dim(\mathfrak{m})}\|\rho_b\| \geq \delta \|\rho_b\|
\end{align*}
which proves the lemma in this case also.
\end{proof}

Fix $\varepsilon_1 > 0$ to be the $\delta$ provided by \cref{lem:maActionLowerBound}.

\section{Transfer operators with holonomy and their spectral bounds}
\label{sec:TransferOperatorsWithHolonomy}
In this section, our goal is to define transfer operators with holonomy which are the main objects of study in this paper and then present the main technical theorem regarding their spectral bounds. We start with some preparation.

\subsection{Modified constructions using the smooth structure on $G$}
\label{subsec:ModifiedConstructionsUsingTheSmoothStructureOnG}
We need to use the smooth structure on $G$ to apply Lie theoretic arguments to derive the LNIC later in \cref{subsec:LNIC}. However, the smooth structure is not readily available on $U$ since it is fractal in nature. Thus, we need an appropriately enlarged open set $\tilde{U}$ of the strong unstable foliation containing $U$. Since the strong unstable foliation is smooth, $\tilde{U} \subset \T^1(X)$ would then be a smooth submanifold and provide a smooth structure at our disposal. At the same time, we would like to extend $\sigma$ to a map on $\tilde{U}$ but this is difficult due to the expanding nature. Conveniently, we can avoid this problem altogether by extending the local inverses in the following sense. Let $w_j$ be the center of $R_j$ for all $j \in \mathcal{A}$. Using arguments of \cite[Lemma 1.2]{Rue89} with a sufficiently small $\delta > 0$, and increasing $\hat{\delta}$ if necessary while ensuring that \cref{eqn:DeltaHatCondition} still holds, there exist open sets $U_j \subset \tilde{U}_j$ such that $\overline{\tilde{U}_j} \subset W_{\epsilon_0}^{\mathrm{su}}(w_j)$ with $\diam_{d_{\mathrm{su}}}(\tilde{U}_j) \leq \hat{\delta}$ for all $j \in \mathcal{A}$ such that for all admissible pairs $(j, k)$, we can naturally extend the inverse $(\sigma|_{\mathtt{C}[j, k]})^{-1}: \interior(U_k) \to \mathtt{C}[j, k]$ to a smooth injective map $\sigma^{-(j, k)}: \tilde{U}_k \to \tilde{U}_j$. More specifically, assuming that $\epsilon_0$ and $\delta$ are sufficiently small without loss of generality, taking any $u_0 \in U_j$ such that $\sigma(u_0) \in U_k$, we can define $\sigma^{-(j, k)}(u)$ to be the unique intersection
\begin{align*}
\sigma^{-(j, k)}(u) = \left(\bigcup_{t \in (-\tau(u_0) - \inf(\tau), -\tau(u_0) + \inf(\tau))} W_{\epsilon_0}^{\mathrm{ss}}(u)a_t\right) \cap W_{\epsilon_0}^{\mathrm{su}}(w_j)
\end{align*}
for all $u \in \tilde{U}_k$. We define $\tilde{U} = \bigsqcup_{j = 1}^N \tilde{U}_j$. Also define the measure $\nu_{\tilde{U}}$ on $\tilde{U}$ simply by $\nu_{\tilde{U}}(B) = \nu_U(B \cap U)$ for all Borel subsets $B \subset \tilde{U}$. Let $j \in \mathbb Z_{\geq 0}$ and $\alpha = (\alpha_0, \alpha_1, \dotsc, \alpha_j)$ be an admissible sequence. Define $\sigma^{-\alpha} = \sigma^{-(\alpha_0, \alpha_1)} \circ \sigma^{-(\alpha_1, \alpha_2)} \circ \cdots \circ \sigma^{-(\alpha_{j - 1}, \alpha_j)}: \tilde{U}_{\alpha_j} \to \tilde{U}_{\alpha_0}$ if $j > 0$ and $\sigma^{-\alpha} = \Id_{\tilde{U}_{\alpha_0}}$ if $j = 0$. Define the cylinder $\tilde{\mathtt{C}}[\alpha] = \sigma^{-\alpha}(\tilde{U}_{\alpha_j}) \supset \mathtt{C}[\alpha]$. Define the smooth maps $\sigma^\alpha = (\sigma^{-\alpha})^{-1}: \tilde{\mathtt{C}}[\alpha] \to \tilde{U}_{\alpha_j}$. These maps are sufficient for our purposes in defining transfer operators. For convenience we define $\tilde{R}_j = [\tilde{U}_j, S_j]$ for all $j \in \mathcal{A}$.

We define more extended maps. Let $(j, k)$ be an admissible pair. The maps $\tau|_{\mathtt{C}[j, k]}$ and $\vartheta|_{\mathtt{C}[j, k]}$ naturally extend to smooth maps $\tau_{(j, k)}: \tilde{\mathtt{C}}[j, k] \to \mathbb R$ and $\vartheta^{(j, k)}: \tilde{\mathtt{C}}[j, k] \to M$ as follows. In light of the above definition of $\sigma^{-(j, k)}$, using the same notations and writing $u' = \sigma^{(j, k)}(u)$, we define $\tau_{(j, k)}(u) \in (\tau(u_0) - \inf(\tau), \tau(u_0) + \inf(\tau))$ uniquely such that $W_{\epsilon_0}^{\mathrm{ss}}(u')a_{-\tau_{(j, k)}(u)} \cap W_{\epsilon_0}^{\mathrm{su}}(w_k) \neq \varnothing$ for all $u \in \tilde{\mathtt{C}}[j, k]$. Similar to before $\vartheta^{(j, k)}(u)$ is such that $F(u)a_{\tau_{(j, k)}(u)} = F\bigl(ua_{\tau_{(j, k)}(u)}\bigr)\vartheta^{(j, k)}(u)^{-1}$ for all $u \in \tilde{\mathtt{C}}[j, k]$. Now for all $k \in \mathbb N$ and admissible sequences $\alpha = (\alpha_0, \alpha_1, \dotsc, \alpha_k)$, we define the smooth maps $\tau_\alpha: \tilde{\mathtt{C}}[\alpha] \to \mathbb R$, $\vartheta^\alpha: \tilde{\mathtt{C}}[\alpha] \to M$, and $\Phi^\alpha: \tilde{\mathtt{C}}[\alpha] \to AM$ by
\begin{align*}
\tau_\alpha(u) &= \sum_{j = 0}^{k - 1} \tau_{(\alpha_j, \alpha_{j + 1})}(\sigma^{(\alpha_0, \alpha_1, \dotsc, \alpha_j)}(u)); \\
\vartheta^\alpha(u) &= \prod_{j = 0}^{k - 1} \vartheta^{(\alpha_j, \alpha_{j + 1})}(\sigma^{(\alpha_0, \alpha_1, \dotsc, \alpha_j)}(u)); \\
\Phi^\alpha(u) &= a_{\tau_\alpha(u)}\vartheta^\alpha(u) = \prod_{j = 0}^{k - 1} \Phi^{(\alpha_j, \alpha_{j + 1})}(\sigma^{(\alpha_0, \alpha_1, \dotsc, \alpha_j)}(u))
\end{align*}
for all $u \in \tilde{\mathtt{C}}[\alpha]$, where the terms of the products are to be in \emph{ascending} order from left to right. For all admissible sequences $\alpha$ with $\len(\alpha) = 0$, we define $\tau_\alpha(u) = 0$ and $\vartheta^\alpha(u) = \Phi^\alpha(u) = e \in AM$ for all $u \in \tilde{\mathtt{C}}[\alpha]$. For all $u \in U$, there is a corresponding unique admissible sequence in $\Sigma^+$ and hence we can instead use the notations $\tau_k(u)$, $\vartheta^k(u)$, and $\Phi^k(u)$ for all $k \in \mathbb Z_{\geq 0}$.

\subsection{Transfer operators}
\label{subsec:TransferOperators}
We will use the notation $\xi = a + ib \in \mathbb C$ for the complex parameter for the transfer operators and use the convention that sums over sequences are actually sums over \emph{admissible} sequences, throughout the paper.

\begin{definition}[Transfer operator with holonomy]
For all $\xi \in \mathbb C$ and $\rho \in \widehat{M}$, the \emph{transfer operator with holonomy} $\tilde{\mathcal{M}}_{\xi\tau, \rho}: C\bigl(\tilde{U}, V_\rho^{\oplus \dim(\rho)}\bigr) \to C\bigl(\tilde{U}, V_\rho^{\oplus \dim(\rho)}\bigr)$ is defined by
\begin{align*}
\tilde{\mathcal{M}}_{\xi\tau, \rho}(H)(u) = \sum_{\substack{(j, k)\\ u' = \sigma^{-(j, k)}(u)}} e^{\xi\tau_{(j, k)}(u')} \rho(\vartheta^{(j, k)}(u')^{-1}) H(u')
\end{align*}
for all $u \in \tilde{U}$ and $H \in C\bigl(\tilde{U}, V_\rho^{\oplus \dim(\rho)}\bigr)$.
\end{definition}

Let $\xi \in \mathbb C$. We denote $\tilde{\mathcal{L}}_{\xi\tau} = \tilde{\mathcal{M}}_{\xi\tau, 1}$ and simply call it the \emph{transfer operator}. For any $\rho \in \widehat{M}$, denote $|_U: C\bigl(\tilde{U}, V_\rho^{\oplus \dim(\rho)}\bigr) \to C\bigl(U, V_\rho^{\oplus \dim(\rho)}\bigr)$ to be the restriction map. Then for all $\rho \in \widehat{M}$, we also define the \emph{transfer operator with holonomy} $\mathcal{M}_{\xi\tau, \rho} = |_U \circ \tilde{\mathcal{M}}_{\xi\tau, \rho} \circ (|_U)^{-1}$ where $(|_U)^{-1}$ denotes taking any preimage using Tietze extension theorem and denote the \emph{transfer operator} $\mathcal{L}_{\xi\tau} = \mathcal{M}_{\xi\tau, 1}$.

\begin{remark}
Let $\xi \in \mathbb C$ and $\rho \in \widehat{M}$. Then $\tilde{\mathcal{M}}_{\xi\tau, \rho}$ preserves $C^k\bigl(\tilde{U}, V_\rho^{\oplus \dim(\rho)}\bigr)$ for all $k \in \mathbb Z_{\geq 0}$ and $\mathcal{M}_{\xi\tau, \rho}$ preserves $C^{\Lip(d)}\bigl(U, V_\rho^{\oplus \dim(\rho)}\bigr)$. Here we regard the target space as a real vector space.
\end{remark}

We recall the Ruelle--Perron--Frobenius (RPF) theorem along with the theory of Gibbs measures in this setting \cite{Bow08,PP90}.

\begin{theorem}
\label{thm:RPFonU}
For all $a \in \mathbb R$, the operator $\mathcal{L}_{a\tau}: C(U, \mathbb C) \to C(U, \mathbb C)$ and its dual $\mathcal{L}_{a\tau}^*: C(U, \mathbb C)^* \to C(U, \mathbb C)^*$ has eigenvectors with the following properties. There exist a unique positive function $h \in C^{\Lip(d)}(U, \mathbb R)$ and a unique Borel probability measure $\nu$ on $U$ such that
\begin{enumerate}
\item	$\mathcal{L}_{a\tau}(h) = e^{\Pr_\sigma(a\tau_{\Sigma})}h$;
\item	$\mathcal{L}_{a\tau}^*(\nu) = e^{\Pr_\sigma(a\tau_{\Sigma})}\nu$;
\item	the eigenvalue $e^{\Pr_\sigma(a\tau_{\Sigma})}$ is maximal simple and the rest of the spectrum of $\mathcal{L}_{a\tau}|_{C^{\Lip(d)}(U, \mathbb C)}$ is contained in a disk of radius strictly less than $e^{\Pr_\sigma(a\tau_{\Sigma})}$;
\item	$\nu(h) = 1$ and the Borel probability measure $\mu$ defined by $d\mu = h \, d\nu$ is $\sigma$-invariant and is the projection of the $a\tau_{\Sigma}$-equilibrium state to $U$, i.e., $\mu = (\proj_U \circ \zeta)_*(\nu_{a\tau_{\Sigma}})$.
\end{enumerate}
\end{theorem}

In light of \cref{thm:RPFonU}, it is convenient to normalize the transfer operators defined above. Let $a \in \mathbb R$. Define $\lambda_a = e^{\Pr_\sigma(-(\delta_\Gamma + a)\tau_{\Sigma})}$ which is the maximal simple eigenvalue of $\mathcal{L}_{-(\delta_\Gamma + a)\tau}$ by \cref{thm:RPFonU} and recall that $\lambda_0 = 1$. Define the eigenvectors, the unique positive function $h_a \in C^{\Lip(d)}(U, \mathbb R)$ and the unique probability measure $\nu_a$ on $U$ with $\nu_a(h_a) = 1$ such that
\begin{align*}
\mathcal{L}_{-(\delta_\Gamma + a)\tau}(h_a) &= \lambda_a h_a; & \mathcal{L}_{-(\delta_\Gamma + a)\tau}^*(\nu_a) &= \lambda_a \nu_a
\end{align*}
provided by \cref{thm:RPFonU}. Note that $d\nu_U = h_0 \, d\nu_0$. Now by \cref{thm:SmoothRPF}, the eigenvector $h_a \in C^{\Lip(d)}(U, \mathbb R)$ extends to an eigenvector $h_a \in C^\infty(\tilde{U}, \mathbb R)$ with bounded derivatives for $\tilde{\mathcal{L}}_{-(\delta_\Gamma + a)\tau}$. For all admissible pairs $(j, k)$, we define the smooth map
\begin{align*}
f_{(j, k)}^{(a)} = -(\delta_\Gamma + a)\tau_{(j, k)} + \log(h_0) - \log(h_0 \circ \sigma^{(j, k)}) - \log(\lambda_a).
\end{align*}
For all $k \in \mathbb N$ and admissible sequences $\alpha = (\alpha_0, \alpha_1, \dotsc, \alpha_k)$, we define the smooth map $f_\alpha^{(a)}: \tilde{\mathtt{C}}[\alpha] \to \mathbb R$ by
\begin{align*}
f_\alpha^{(a)}(u) = \sum_{j = 0}^{k - 1} f_{(\alpha_j, \alpha_{j + 1})}^{(a)}(\sigma^{(\alpha_0, \alpha_1, \dotsc, \alpha_j)}(u)) \qquad \text{for all $u \in \tilde{\mathtt{C}}[\alpha]$}.
\end{align*}
For all admissible sequences $\alpha$ with $\len(\alpha) = 0$, we define $f_\alpha^{(a)}(u) = 0$. As before, for all $u \in U$, we can also use the notation $f_k^{(a)}(u)$ for any $k \in \mathbb Z_{\geq 0}$.

We now normalize the transfer operators. Let $\xi \in \mathbb C$ and $\rho \in \widehat{M}$. We define $\tilde{\mathcal{M}}_{\xi, \rho}: C\bigl(\tilde{U}, V_\rho^{\oplus \dim(\rho)}\bigr) \to C\bigl(\tilde{U}, V_\rho^{\oplus \dim(\rho)}\bigr)$ by
\begin{align*}
\tilde{\mathcal{M}}_{\xi, \rho}(H)(u) = \sum_{\substack{(j, k)\\ u' = \sigma^{-(j, k)}(u)}} e^{(f_{(j, k)}^{(a)} + ib\tau_{(j, k)})(u')} \rho(\vartheta^{(j, k)}(u')^{-1}) H(u')
\end{align*}
for all $u \in \tilde{U}$ and $H \in C\bigl(\tilde{U}, V_\rho^{\oplus \dim(\rho)}\bigr)$. For all $k \in \mathbb N$, its $k$\textsuperscript{th} iteration is
\begin{align}
\label{eqn:k^thIterationOfCongruenceTransferOperatorOfType_rho}
\tilde{\mathcal{M}}_{\xi, \rho}^k(H)(u) = \sum_{\substack{\alpha: \len(\alpha) = k\\ u' = \sigma^{-\alpha}(u)}} e^{f_\alpha^{(a)}(u')} \rho_b(\Phi^\alpha(u')^{-1}) H(u')
\end{align}
for all $u \in \tilde{U}$ and $H \in C\bigl(\tilde{U}, V_\rho^{\oplus \dim(\rho)}\bigr)$. Again, we denote $\tilde{\mathcal{L}}_\xi = \tilde{\mathcal{M}}_{\xi, 1}$ and using the restriction map $|_U$, we get the corresponding normalized operators $\mathcal{M}_{\xi, \rho}: C\bigl(U, V_\rho^{\oplus \dim(\rho)}\bigr) \to C\bigl(U, V_\rho^{\oplus \dim(\rho)}\bigr)$ and $\mathcal{L}_\xi: C(U, \mathbb C) \to C(U, \mathbb C)$. With this normalization, for all $a \in \mathbb R$, the maximal simple eigenvalue of $\mathcal{L}_a$ is $1$ with eigenvector $\frac{h_a}{h_0}$. Moreover, we have $\mathcal{L}_0^*(\nu_U) = \nu_U$.

We fix some related constants. By perturbation theory of operators as in \cite[Chapter 7]{Kat95} and \cite[Proposition 4.6]{PP90}, we can fix $a_0' > 0$ such that the map $[-a_0', a_0'] \to \mathbb R$ defined by $a \mapsto \lambda_a$ and the map $[-a_0', a_0'] \to C(\tilde{U}, \mathbb R)$ defined by $a \mapsto h_a$ are Lipschitz. We then fix $A_f > 0$ such that $\left|f_{(j, k)}^{(a)}(u) - f_{(j, k)}^{(0)}(u)\right| \leq A_f|a|$ for all admissible pairs $(j, k)$, $u \in \tilde{\mathtt{C}}[j, k]$, and $|a| \leq a_0'$. Fix $\overline{\tau} = \max_{(j, k)} \sup_{u \in \tilde{\mathtt{C}}[j, k]} \tau_{(j, k)}(u)$ and $\underline{\tau} = \min_{(j, k)} \inf_{u \in \tilde{\mathtt{C}}[j, k]} \tau_{(j, k)}(u)$. Fix
\begin{align*}
T_0 >{}&\max\bigg(\max_{(j, k)} \|\tau_{(j, k)}\|_{C^1}, \max_{(j, k)} \sup_{|a| \leq a_0'} \left\|f_{(j, k)}^{(a)}\right\|_{C^1}, \max_{(j, k)} \bigl\|\vartheta^{(j, k)}\bigr\|_{C^1}\bigg)
\end{align*}
which is possible by \cite[Lemma 4.1]{PS16}.

\subsection{Spectral bounds with holonomy}
\label{subsec:SpectralBoundsWithHolonomy}
We first introduce some norms and seminorms. Let $\rho \in \widehat{M}$, and $H \in C\bigl(U, V_\rho^{\oplus \dim(\rho)}\bigr)$. We will denote $\|H\| \in C(U, \mathbb R)$ to be the function defined by $\|H\|(u) = \|H(u)\|_2$ for all $u \in U$, and if $\rho = 1$, we will denote $|H| \in C(U, \mathbb R)$ to be the function defined by $|H|(u) = |H(u)| \in \mathbb R$ for all $u \in U$. We define $\|H\|_\infty = \sup \|H\|$. We use similar notations if the domain is $\tilde{U}$. We define the Lipschitz seminorm and the Lipschitz norm by
\begin{align*}
\Lip_d(H) &= \sup_{\substack{u, u' \in U\\ \text{such that } u \neq u'}} \frac{\|H(u) - H(u')\|_2}{d(u, u')}; & \|H\|_{\Lip(d)} &= \|H\|_\infty + \Lip_d(H)
\end{align*}
respectively.

Since we will mostly use the $C^1$ norm, we avoid defining the $C^k$ norm for a general $k \in \mathbb N$. Let $Y$ be a Riemannian manifold and $H \in C^1(\tilde{U}, Y)$. We define the $C^1$ seminorm and the $C^1$ norm by
\begin{align*}
|H|_{C^1} &= \sup_{u \in \tilde{U}}\|(dH)_u\|_{\mathrm{op}}; & \|H\|_{C^1} &= \|H\|_\infty + |H|_{C^1}
\end{align*}
respectively. We can define another useful norm by
\begin{align*}
\|H\|_{1, b} &= \|H\|_\infty + \frac{1}{\max(1, |b|)} |H|_{C^1}.
\end{align*}
Henceforth, by differentiable function spaces on $\tilde{U}$ or its derived suspension spaces, such as $C^1(\tilde{U}, Y)$, we will always mean the space of $C^1$ functions whose $C^1$ norm is \emph{bounded}.

For all $\rho \in \widehat{M}$, we define the Banach spaces
\begin{align*}
\mathcal{V}_\rho(U) &= C^{\Lip(d)}\bigl(U, V_\rho^{\oplus \dim(\rho)}\bigr); & \mathcal{V}_\rho(\tilde{U}) &= C^1\bigl(\tilde{U}, V_\rho^{\oplus \dim(\rho)}\bigr).
\end{align*}

Now we can state \cref{thm:TheoremFrameFlow} regarding spectral bounds of transfer operators with holonomy.

\begin{theorem}
\label{thm:TheoremFrameFlow}
There exist $\eta > 0$, $C > 0$, $a_0 > 0$, and $b_0 > 0$ such that for all $\xi \in \mathbb C$ with $|a| < a_0$, if $(b, \rho) \in \widehat{M}_0(b_0)$, then for all $k \in \mathbb N$ and $H \in \mathcal{V}_\rho(\tilde{U})$, we have
\begin{align*}
\big\|\tilde{\mathcal{M}}_{\xi, \rho}^k(H)\big\|_2 \leq Ce^{-\eta k} \|H\|_{1, \|\rho_b\|}.
\end{align*}
\end{theorem}

We reduce \cref{thm:TheoremFrameFlow} to \cref{thm:FrameFlowDolgopyat} which captures the mechanism of Dolgopyat's method  in our setting. Similar theorems have appeared in \cite{Dol98,Nau05,Sto11,OW16}. The main difference with previous works is that we deal with holonomy.

We define the cone
\begin{align*}
K_B(\tilde{U}) = \{h \in C^1(\tilde{U}, \mathbb R): h > 0, \|(dh)_u\|_{\mathrm{op}} \leq Bh(u) \text{ for all } u \in \tilde{U}\}.
\end{align*}

\begin{remark}
It is useful to note that we can easily derive the equivalent $\log$-Lipschitz characterization given by $K_B(\tilde{U}) = \{h \in C^1(\tilde{U}, \mathbb R): h > 0, |\log \circ h|_{C^1} \leq B\}$.
\end{remark}

\begin{theorem}
\label{thm:FrameFlowDolgopyat}
There exist $m \in \mathbb N$, $\eta \in (0, 1)$, $E > \max\left(1, \frac{1}{b_0}, \frac{1}{\delta_\varrho}\right)$, $a_0 > 0$, $b_0 > 0$, and a set of operators $\{\mathcal{N}_{a, J}^H: C^1(\tilde{U}, \mathbb R) \to C^1(\tilde{U}, \mathbb R): H \in \mathcal{V}_\rho(\tilde{U}), |a| < a_0, J \in \mathcal{J}(b, \rho), \text{ for some } (b, \rho) \in \widehat{M}_0(b_0)\}$, where $\mathcal{J}(b, \rho)$ is some finite set for all $(b, \rho) \in \widehat{M}_0(b_0)$, such that
\begin{enumerate}
\item\label{itm:FrameFlowDolgopyatProperty1}	$\mathcal{N}_{a, J}^H(K_{E\|\rho_b\|}(\tilde{U})) \subset K_{E\|\rho_b\|}(\tilde{U})$ for all $H \in \mathcal{V}_\rho(\tilde{U})$, $|a| < a_0$, $J \in \mathcal{J}(b, \rho)$, and $(b, \rho) \in \widehat{M}_0(b_0)$;
\item\label{itm:FrameFlowDolgopyatProperty2}	$\left\|\mathcal{N}_{a, J}^H(h)\right\|_2 \leq \eta \|h\|_2$ for all $h \in K_{E\|\rho_b\|}(\tilde{U})$, $H \in \mathcal{V}_\rho(\tilde{U})$, $|a| < a_0$, $J \in \mathcal{J}(b, \rho)$, and $(b, \rho) \in \widehat{M}_0(b_0)$;
\item\label{itm:FrameFlowDolgopyatProperty3}	for all $\xi \in \mathbb C$ with $|a| < a_0$, if $(b, \rho) \in \widehat{M}_0(b_0)$, and if $H \in \mathcal{V}_\rho(\tilde{U})$ and $h \in K_{E\|\rho_b\|}(\tilde{U})$ satisfy
\begin{enumerate}[label=(1\alph*), ref=\theenumi(1\alph*)]
\item\label{itm:FrameFlowDominatedByh}	$\|H(u)\|_2 \leq h(u)$ for all $u \in \tilde{U}$;
\item\label{itm:FrameFlowLogLipschitzh}	$\|(dH)_u\|_{\mathrm{op}} \leq E\|\rho_b\|h(u)$ for all $u \in \tilde{U}$;
\end{enumerate}
then there exists $J \in \mathcal{J}(b, \rho)$ such that
\begin{enumerate}[label=(2\alph*), ref=\theenumi(2\alph*)]
\item\label{itm:FrameFlowDominatedByDolgopyat}	$\big\|\tilde{\mathcal{M}}_{\xi, \rho}^m(H)(u)\big\|_2 \leq \mathcal{N}_{a, J}^H(h)(u)$ for all $u \in \tilde{U}$;
\item\label{itm:FrameFlowLogLipschitzDolgopyat}	$\left\|\left(d\tilde{\mathcal{M}}_{\xi, \rho}^m(H)\right)_u\right\|_{\mathrm{op}} \leq E\|\rho_b\|\mathcal{N}_{a, J}^H(h)(u)$ for all $u \in \tilde{U}$.
\end{enumerate}
\end{enumerate}
\end{theorem}

\begin{proof}[Proof that \cref{thm:FrameFlowDolgopyat} implies \cref{thm:TheoremFrameFlow}]
Fix $m \in \mathbb N, a_0 > 0, b_0 > 0, E > 0$ to be the ones from \cref{thm:FrameFlowDolgopyat} and $\tilde{\eta} \in (0, 1)$ to be the $\eta$ from \cref{thm:FrameFlowDolgopyat}. Fix
\begin{align*}
B = \sup_{|a| \leq a_0, \rho \in \widehat{M}} \big\|\tilde{\mathcal{M}}_{\xi, \rho}\big\|_{\mathrm{op}} \leq \sup_{|a| \leq a_0} \big\|\tilde{\mathcal{L}}_\xi\big\|_{\mathrm{op}} \leq Ne^{T_0}
\end{align*}
viewing the transfer operators as operators on $L^2\bigl(\tilde{U}, V_\rho^{\oplus \dim(\rho)}\bigr)$ and $L^2(\tilde{U}, \mathbb R)$ respectively. Fix $\eta = \frac{-\log(\tilde{\eta})}{m}$ and $C = B^m \tilde{\eta}^{-1}$. Let $\xi \in \mathbb C$ with $|a| < a_0$. Suppose $(b, \rho) \in \widehat{M}_0(b_0)$. Let $k \in \mathbb N$ and $H \in \mathcal{V}_\rho(\tilde{U})$. The theorem is trivial if $H = 0$, so suppose that $H \neq 0$. First set $h_0 \in K_{E\|\rho_b\|}(\tilde{U})$ to be the positive constant function defined by $h_0(u) = \|H\|_{1, \|\rho_b\|}$ for all $u \in \tilde{U}$. Denote $H_0 = H$. Then $H_0$ and $h_0$ satisfy \cref{itm:FrameFlowDominatedByh,itm:FrameFlowLogLipschitzh} in \cref{thm:FrameFlowDolgopyat}. Thus, \cref{thm:FrameFlowDolgopyat} allows us to inductively obtain $h_j = \mathcal{N}_{a, J_{j - 1}}^{H_{j - 1}}(h_{j - 1}) \in K_{E\|\rho_b\|}(\tilde{U})$ for some $J_{j - 1} \in \mathcal{J}(b, \rho)$ and $H_j = \tilde{\mathcal{M}}_{\xi, \rho}^m(H_{j - 1})$ which satisfy \cref{itm:FrameFlowDominatedByDolgopyat,itm:FrameFlowLogLipschitzDolgopyat} in \cref{thm:FrameFlowDolgopyat}, for all $j \in \mathbb N$. Now using \cref{itm:FrameFlowDolgopyatProperty2} in \cref{thm:FrameFlowDolgopyat}, we have $\big\|\tilde{\mathcal{M}}_{\xi, \rho}^{jm}(H)\big\|_2 \leq \|h_j\|_2 \leq \tilde{\eta}^j\|h_0\|_2 = \tilde{\eta}^j\|H\|_{1, \|\rho_b\|}$ for all $j \in \mathbb Z_{\geq 0}$. Writing $k = jm + l$ for some $j \in \mathbb Z_{\geq 0}$ and $0 \leq l < m$, we have
\begin{align*}
\big\|\tilde{\mathcal{M}}_{\xi, \rho}^k(H)\big\|_2 \leq B^l\big\|\tilde{\mathcal{M}}_{\xi, \rho}^{jm}(H)\big\|_2 \leq B^l\tilde{\eta}^j\|H\|_{1, \|\rho_b\|} \leq Ce^{-\eta k} \|H\|_{1, \|\rho_b\|}.
\end{align*}
\end{proof}

\section{Local non-integrability condition and non-concentration property}
\label{sec:LNIC&NCP}
This section is devoted to the main tools needed for the proof of \cref{thm:FrameFlowDolgopyat} in \cref{sec:ProofOfFrameFlowDolgopyat}. Non-integrability type conditions have appeared in all previous works employing Dolgopyat's method. We will prove \cref{pro:FrameFlowLNIC} which is the appropriate formulation in our setting and call it the \emph{local non-integrability condition (LNIC)} as in \cite{Sto11}. Running Dolgopyat's method with holonomy also requires \cref{pro:NonConcentrationProperty} which we call the \emph{non-concentration property (NCP)}.

\subsection{Local non-integrability condition}
\label{subsec:LNIC}
First, we will define a map related to Brin--Pesin moves \cite{BP74,Bri82} which will be needed for the LNIC in our setting.

We choose unique isometric lifts $\tilde{\mathsf{R}}_j = \big[\tilde{\mathsf{U}}_j, \mathsf{S}_j\big] \subset \T^1(\mathbb H^n)$ of $\tilde{R}_j$ for all $j \in \mathcal{A}$. Define $\tilde{\mathsf{R}} = \bigsqcup_{j \in \mathcal{A}} \tilde{\mathsf{R}}_j$ and $\tilde{\mathsf{U}} = \bigsqcup_{j \in \mathcal{A}} \tilde{\mathsf{U}}_j$. For all $u \in \tilde{R}$, let $\tilde{u} \in \tilde{\mathsf{R}}$ denote the unique lift in $\tilde{\mathsf{R}}$. We then lift the section $F$ to $\mathsf{F}: \bigsqcup_{\gamma \in \Gamma} \gamma\tilde{\mathsf{R}} \to \F(\mathbb H^n)$ in the natural way.

\begin{definition}[Associated sequence of frames]
Let $z_1 \in \tilde{R}_1$ be the center. Consider some sequence of tangent vectors $(z_1, z_2, z_3, z_4, z_1) \in (\tilde{R}_1)^5$ such that $z_2 \in S_1$, $z_4 \in \tilde{U}_1$ and $z_3 = [z_4, z_2]$. Its lift to the universal cover is $(\tilde{z}_1, \tilde{z}_2, \tilde{z}_3, \tilde{z}_4, \tilde{z}_1) \in (\tilde{\mathsf{R}}_1)^5 \subset \T^1(\mathbb H^n)^5 \cong (G/M)^5$. We define an \emph{associated sequence of frames} to be the unique sequence $(g_1, g_2, \dotsc, g_5) \in \F(\mathbb H^n)^5 \cong G^5$ where
\begin{align*}
g_1 &= \mathsf{F}(\tilde{z}_1); \\
g_2 &= \mathsf{F}(\tilde{z}_2) \in g_1N^- \text{ such that } g_2M = \tilde{z}_2 \in \T^1(\mathbb H^n) \cong G/M; \\
g_3 &\in g_2N^+ \text{ such that } g_3a_tM = \tilde{z}_3 \in \T^1(\mathbb H^n) \cong G/M \text{ for some } t \in (-\underline{\tau}, \underline{\tau}); \\
g_4 &\in g_3N^- \text{ such that } g_4a_tM = \tilde{z}_4 \in \T^1(\mathbb H^n) \cong G/M \text{ for some } t \in (-\underline{\tau}, \underline{\tau}); \\
g_5 &\in g_4N^+ \text{ such that } g_5a_tM = \tilde{z}_1 \in \T^1(\mathbb H^n) \cong G/M \text{ for some } t \in (-\underline{\tau}, \underline{\tau}).
\end{align*}
\end{definition}

\begin{remark}
The sequence of frames is obtained by ``moving the frame $\mathsf{F}(\tilde{z}_1)$ only along the strong unstable and strong stable directions'' corresponding to the path represented by the sequence $(\tilde{z}_1, \tilde{z}_2, \tilde{z}_3, \tilde{z}_4, \tilde{z}_1)$. Using properties of the strong unstable and strong stable leaves, we see that $t \in (-\underline{\tau}, \underline{\tau})$ must be the same throughout the sequence in the definition above.
\end{remark}

We continue using the notation in the above definition. Define the open set $N_1^+ = \{n^+ \in N^+: F(z_1)n^+ \in F(\tilde{U}_1)\} \subset N^+$ and the compact set $N_1^- = \{n^- \in N^-: F(z_1)n^- \in F(S_1)\} \subset N^-$. Now, if the above sequence $(z_1, z_2, z_3, z_4, z_1)$ is corresponding to some $n^+ \in N_1^+$ and some $n^- \in N_1^-$ such that $F(z_4) = F(z_1)n^+$ and $F(z_2) = F(z_1)n^-$ respectively, then we can define the map $\Xi: N_1^+ \times N_1^- \to AM$ by
\begin{align*}
\Xi(n^+, n^-) = g_5^{-1}g_1 \in AM.
\end{align*}
To view it as a function of only the first coordinate for a fixed $n^- \in N_1^-$, we write $\Xi_{n^-}: N_1^+ \to AM$.

Let $z_1 \in \tilde{R}_1$ be the center. Let $j \in \mathbb N$ and $\alpha = (\alpha_0, \alpha_2, \dotsc, \alpha_{j - 1}, 1)$ be an admissible sequence. By following definitions, there exists an element which we denote by $n_\alpha \in N_1^-$ such that
\begin{align*}
F(\mathcal{P}^j(\sigma^{-\alpha}(z_1))) = F(z_1)n_\alpha.
\end{align*}
This is well-defined since $\sigma^{-\alpha}(z_1) \in \mathtt{C}[\alpha] \subset U$.

In order to derive the LNIC in \cref{pro:FrameFlowLNIC}, we first start with a few useful lemmas regarding $\Xi$.

\begin{lemma}
\label{lem:BrinPesinInTermsOfHolonomy}
Let $j \in \mathbb N$, $\alpha = (\alpha_0, \alpha_1, \dotsc, \alpha_{j - 1}, 1)$ be an admissible sequence, and $n^- = n_\alpha \in N_1^-$. Let $u \in \tilde{U}_1$ and $n^+ \in N_1^+$ such that $F(u) = F(z_1)n^+$ where $z_1 \in \tilde{R}_1$ is the center. Then, we have
\begin{align*}
\Xi(n^+, n^-) = \Phi^\alpha(\sigma^{-\alpha}(z_1))^{-1}\Phi^\alpha(\sigma^{-\alpha}(u)).
\end{align*}
\end{lemma}

\begin{proof}
Let $z_1 \in \tilde{R}_1$ be the center. Let $j \in \mathbb N$, $\alpha = (\alpha_0, \alpha_1, \dotsc, \alpha_{j - 1}, 1)$ be an admissible sequence and $n^- = n_\alpha \in N_1^-$. Let $u \in \tilde{U}_1$ and $n^+ \in N_1^+$ such that $F(u) = F(z_1)n^+$. Let $s \in S_1$ such that $F(s) = F(z_1)n^-$. To calculate $\Xi(n^+, n^-)$, we first consider $(z_1, z_2, z_3, z_4, z_1) = (z_1, s, [u, s], u, z_1) \in (\tilde{R}_1)^5$ and then compute the associated sequence of frames $(g_1, g_2, g_3, g_4, g_5) \in G^5$. Firstly, $g_1 = \mathsf{F}(\tilde{z}_1)$. Then using definitions, we have
\begin{align*}
g_2 = g_1n^- = \mathsf{F}(\tilde{z}_1)n^- = \mathsf{F}(\mathcal{P}^j(\sigma^{-\alpha}(\tilde{z}_1))) = \mathsf{F}(\sigma^{-\alpha}(\tilde{z}_1))\Phi^\alpha(\sigma^{-\alpha}(z_1)).
\end{align*}
Now $g_3 = g_2n_2$ for some $n_2 \in N^+$. So
\begin{align*}
g_3 = \mathsf{F}(\sigma^{-\alpha}(\tilde{z}_1))\Phi^\alpha(\sigma^{-\alpha}(z_1))n_2 = \mathsf{F}(\sigma^{-\alpha}(\tilde{z}_1)) n_2' \Phi^\alpha(\sigma^{-\alpha}(z_1))
\end{align*}
where $n_2' = \Phi^\alpha(\sigma^{-\alpha}(z_1)) n_2 \Phi^\alpha(\sigma^{-\alpha}(z_1))^{-1} \in N^+$. But the frame $\mathsf{F}(\sigma^{-\alpha}(\tilde{z}_1)) n_2'$ must be based at $\tilde{v} \in \gamma\tilde{\mathsf{U}}_{\alpha_0}$ for some $\gamma \in \Gamma$ such that $\tilde{v}a_{\tau_\alpha(\sigma^{-\alpha}(\tilde{z}_1)) + t} = [\tilde{u}, \tilde{s}]$ for some $t \in (-\underline{\tau}, \underline{\tau})$. So $\tilde{v} = \sigma^{-\alpha}(\tilde{u})$ and $t = \tau_\alpha(\sigma^{-\alpha}(\tilde{u})) - \tau_\alpha(\sigma^{-\alpha}(\tilde{z}_1))$. Moreover, $\mathsf{F}(\sigma^{-\alpha}(\tilde{z}_1)) n_2' = \mathsf{F}(\sigma^{-\alpha}(\tilde{u}))$ and hence $g_3 = \mathsf{F}(\sigma^{-\alpha}(\tilde{u}))\Phi^\alpha(\sigma^{-\alpha}(z_1))$. From definitions, $\tilde{v}a_{\tau_\alpha(\sigma^{-\alpha}(\tilde{u}))} = [\tilde{u}, \tilde{s}]$ implies $\mathsf{F}(\sigma^{-\alpha}(\tilde{u})) = \mathsf{F}([\tilde{u}, \tilde{s}])\Phi^\alpha(\sigma^{-\alpha}(u))^{-1}$. Thus
\begin{align*}
g_3 = \mathsf{F}([\tilde{u}, \tilde{s}])\Phi^\alpha(\sigma^{-\alpha}(u))^{-1}\Phi^\alpha(\sigma^{-\alpha}(z_1)).
\end{align*}
Now $g_4 = g_3n_3$ for some $n_3 \in N^+$. So
\begin{align*}
g_4 &= \mathsf{F}([\tilde{u}, \tilde{s}])\Phi^\alpha(\sigma^{-\alpha}(u))^{-1}\Phi^\alpha(\sigma^{-\alpha}(z_1))n_3 \\
&= \mathsf{F}([\tilde{u}, \tilde{s}])n_3'\Phi^\alpha(\sigma^{-\alpha}(u))^{-1}\Phi^\alpha(\sigma^{-\alpha}(z_1))
\end{align*}
where $n_3' = \Phi^\alpha(\sigma^{-\alpha}(u))^{-1}\Phi^\alpha(\sigma^{-\alpha}(z_1)) n_3 \left(\Phi^\alpha(\sigma^{-\alpha}(u))^{-1}\Phi^\alpha(\sigma^{-\alpha}(z_1))\right)^{-1} \in N^-$. By similar arguments, the frame $\mathsf{F}([\tilde{u}, \tilde{s}])n_3'$ must be based at $\tilde{u} \in \tilde{\mathsf{U}}_1$. Thus
\begin{align*}
g_4 = \mathsf{F}(\tilde{u})\Phi^\alpha(\sigma^{-\alpha}(u))^{-1}\Phi^\alpha(\sigma^{-\alpha}(z_1)).
\end{align*}
Finally, $g_5 = g_4n_4$ for some $n_4 \in N^+$. So
\begin{align*}
g_5 = \mathsf{F}(\tilde{u})\Phi^\alpha(\sigma^{-\alpha}(u))^{-1}\Phi^\alpha(\sigma^{-\alpha}(z_1))n_4 = \mathsf{F}(\tilde{u})n_4'\Phi^\alpha(\sigma^{-\alpha}(u))^{-1}\Phi^\alpha(\sigma^{-\alpha}(z_1))
\end{align*}
where $n_4' = \Phi^\alpha(\sigma^{-\alpha}(u))^{-1}\Phi^\alpha(\sigma^{-\alpha}(z_1)) n_4 \left(\Phi^\alpha(\sigma^{-\alpha}(u))^{-1}\Phi^\alpha(\sigma^{-\alpha}(z_1))\right)^{-1} \in N^-$. Again by similar arguments, the frame $\mathsf{F}(\tilde{u})n_4'$ must be based at $\tilde{z}_1 \in \tilde{\mathsf{U}}_1$. Thus
\begin{align*}
g_5 = \mathsf{F}(\tilde{z}_1)\Phi^\alpha(\sigma^{-\alpha}(u))^{-1}\Phi^\alpha(\sigma^{-\alpha}(z_1)).
\end{align*}
Then by definition, we have the calculation
\begin{align*}
\Xi(n^+, n^-) = g_5^{-1}g_1 = \Phi^\alpha(\sigma^{-\alpha}(z_1))^{-1}\Phi^\alpha(\sigma^{-\alpha}(u)).
\end{align*}
\end{proof}

Recall from definitions that $e \in N_1^+$ where $N_1^+ \subset N^+$ is an open subset and hence $\T_e(N_1^+) = \T_e(N^+) = \mathfrak{n}^+$. Note that $AMN^+N^- \subset G$ is an open dense subset and hence we have the vector space decomposition $\mathfrak{g} = \mathfrak{a} \oplus \mathfrak{m} \oplus \mathfrak{n}^+ \oplus \mathfrak{n}^-$. Denote the projection onto $\mathfrak{a} \oplus \mathfrak{m}$ with respect to this decomposition by $\pi: \mathfrak{g} \to \mathfrak{a} \oplus \mathfrak{m}$. We then have the following \cref{lem:BrinPesinDerivativeImageIsAdjointProjection} where $\epsilon_0$ is as in \cref{subsec:MarkovSections} and $N_{1, \epsilon_0}^- = \left\{n^- \in N^-: F(z_1)n^- \in F\left(W_{\epsilon_0}^{\mathrm{ss}}(z_1)\right)\right\}$ where $z_1 \in \tilde{R}_1$ is the center.

\begin{lemma}
\label{lem:BrinPesinDerivativeImageIsAdjointProjection}
For all $n^- \in N_1^-$, we have
\begin{align*}
(d\Xi_{n^-})_e(\omega) = \pi(\Ad_{n^-}((dh_{n^-})_e(\omega)))
\end{align*}
for all $\omega \in \mathfrak{n}^+$ where $h_{n^-}: N_1^+ \to N^+$ is a diffeomorphism onto its image which is also smooth in $n^- \in N_{1, \epsilon_0}^-$ and satisfies $h_e = \Id_{N_1^+}$. Moreover, the image $(d\Xi_{n^-})_e(\mathfrak{n}^+) \subset \mathfrak{a} \oplus \mathfrak{m}$ is the projection $(d\Xi_{n^-})_e(\mathfrak{n}^+) = \pi(\Ad_{n^-}(\mathfrak{n}^+))$.
\end{lemma}

\begin{proof}
Let $z_1 \in \tilde{R}_1$ be the center. Let $n^- \in N_1^-$. For all $n^+ \in N_1^+$, consider $(z_1, z_2, z_3(n^+), z_4(n^+), z_1) \in (\tilde{R}_1)^5$ such that $F(z_2) = F(z_1)n^-$ and $F(z_4(n^+)) = F(z_1)n^+$ and let $(g_1, g_2, g_3(n^+), g_4(n^+), g_5(n^+))$ be the associated sequence of frames. Then for all $n^+ \in N_1^+$, the frame $g_5(n^+)a_t$ is based at $\tilde{z}_1$ for some $t \in (-\underline{\tau}, \underline{\tau})$. By transversality of the smooth foliations, the implicit function theorem on a coordinate chart gives smooth functions $f_{n^-, 2}: N_1^+ \to N^+, f_{n^-, 3}: N_1^+ \to N^-$, and $f_{n^-, 4}: N_1^+ \to N^+$ which are also smooth in $n^- \in N_{1, \epsilon_0}^-$ such that
\begin{align*}
g_5(n^+) &= g_1n^-f_{n^-, 2}(n^+)f_{n^-, 3}(n^+)f_{n^-, 4}(n^+).
\end{align*}
Note that $f_{n^-, 2}$ is a diffeomorphism onto its image. Let $h_{n^-} = h_{n^-, 2}: N_1^+ \to N^+, h_{n^-, 3}: N_1^+ \to N^-$, and $h_{n^-, 4}: N_1^+ \to N^+$ be smooth functions which are also smooth in $n^- \in N_{1, \epsilon_0}^-$ defined by the group inverses $h_{n^-, j}(n^+) = f_{n^-, j}(n^+)^{-1}$ for all $n^+ \in N_1^+$ and $2 \leq j \leq 4$. Then, $h_{n^-, 2}$ is a diffeomorphism onto its image and
\begin{align*}
\Xi_{n^-}(n^+) = g_5(n^+)^{-1}g_1 = h_{n^-, 4}(n^+)h_{n^-, 3}(n^+)h_{n^-, 2}(n^+)(n^-)^{-1}.
\end{align*}
Note that $h_{n^-, 2}(e) = h_{n^-, 4}(e) = e$ and $h_{n^-, 3}(e) = n^-$. Let $m_g^{\mathrm{R}}: G \to G$ be the right multiplication map and $C_g: G \to G$ be the conjugation map by $g \in G$. Let $\omega \in \mathfrak{n}^+$. Then using the product rule, we have
\begin{align*}
(d\Xi_{n^-})_e(\omega) ={}& \left(\left(dm_{(n^-)^{-1}}^{\mathrm{R}}\right)_{n^-} \circ \left(dm_{n^-}^{\mathrm{R}}\right)_e\right)((dh_{n^-, 4})_e(\omega)) \\
&{}+ \left(dm_{(n^-)^{-1}}^{\mathrm{R}}\right)_{n^-}((dh_{n^-, 3})_e(\omega)) + (dC_{n^-})_e((dh_{n^-, 2})_e(\omega)) \\
={}& (dh_{n^-, 4})_e(\omega) + \left(dm_{(n^-)^{-1}}^{\mathrm{R}}\right)_{n^-}((dh_{n^-, 3})_e(\omega)) \\
&{}+ (dC_{n^-})_e((dh_{n^-, 2})_e(\omega)).
\end{align*}
Since $(d\Xi_{n^-})_e(\omega) \in \mathfrak{a} \oplus \mathfrak{m}, (dh_{n^-, 4})_e(\omega) \in \mathfrak{n}^+, \left(dm_{(n^-)^{-1}}^{\mathrm{R}}\right)_{n^-}((dh_{n^-, 3})_e(\omega)) \in \mathfrak{n}^-$, and $(dC_{n^-})_e = \Ad_{n^-}$, we have $(d\Xi_{n^-})_e(\omega) = \pi(\Ad_{n^-}((dh_{n^-, 2})_e(\omega)))$. Noting that $(dh_{n^-, 2})_e$ surjects to $\mathfrak{n}^+$, we also have $(d\Xi_{n^-})_e(\mathfrak{n}^+) = \pi(\Ad_{n^-}(\mathfrak{n}^+))$.
\end{proof}

Throughout the paper it is often convenient to use the upper half space model $\mathbb H^n \cong \{(x_1, x_2, \dotsc, x_n) \in \mathbb R^n: x_n > 0\}$ with boundary at infinity $\partial_\infty(\mathbb H^n) \cong \{(x_1, x_2, \dotsc, x_n) \in \mathbb R^n: x_n = 0\} \cup \{\infty\} \cong \mathbb R^{n - 1} \cup \{\infty\}$. We also use the isometry $\T(\mathbb H^n) \cong \mathbb H^n \times \mathbb R^n$. Let $(e_1, e_2, \dotsc, e_n)$ be the standard basis on $\mathbb R^n$. We assume without loss of generality that the identifications are made such that the reference vector is $v_o = (e_n, e_n)$ and the reference frame is $F_o = ((e_n, e_1), (e_n, e_2), \dotsc, (e_n, e_n))$ where the first entries of the tangent vectors are their basepoints. Let $d_{\mathrm{E}}$ denote the Euclidean distance. Let $B^{\mathrm{E}}_\epsilon(x) \subset \mathbb R^{n - 1}$ denote the open Euclidean ball of radius $\epsilon > 0$ centered at $x \in \mathbb R^{n - 1}$.

\begin{lemma}
\label{lem:am_ProjectionOfAdjointImage}
There exist $n_1^-, n_2^-, \dotsc, n_{j_{\mathrm{m}}}^- \in N_1^-$ for some $j_{\mathrm{m}} \in \mathbb N$ and $\delta > 0$ such that if $\eta_1^-, \eta_2^-, \dotsc, \eta_{j_{\mathrm{m}}}^- \in N_1^-$ with $d_{N^-}(\eta_j^-, n_j^-) \leq \delta$ for all $1 \leq j \leq j_{\mathrm{m}}$, then
\begin{align*}
\sum_{j = 1}^{j_{\mathrm{m}}} \pi\left(\Ad_{\eta_j^-}(\mathfrak{n}^+)\right) = \mathfrak{a} \oplus \mathfrak{m}.
\end{align*}
\end{lemma}

\begin{proof}
First we show that $\sum_{n^- \in N^-} \pi(\Ad_{n^-}(\mathfrak{n}^+)) = \mathfrak{a} \oplus \mathfrak{m}$, or more explicitly, we show that there exist $n_1^-, n_2^-, \dotsc, n_{j_{\mathrm{m}}}^- \in N^-$ for some $j_{\mathrm{m}} \in \mathbb N$ such that $\sum_{j = 1}^{j_{\mathrm{m}}} \pi\left(\Ad_{n_j^-}(\mathfrak{n}^+)\right) = \mathfrak{a} \oplus \mathfrak{m}$. We use the formula
\begin{align*}
\Ad_{e^{n^-}}(n^+) &= e^{\ad_{n^-}}(n^+) = \sum_{j = 0}^\infty \frac{1}{j!}(\ad_{n^-})^j(n^+) \\
&= n^+ + [n^-, n^+] + \frac{1}{2!}[n^-, [n^-, n^+]] + \frac{1}{3!}[n^-, [n^-, [n^-, n^+]]] + \dotsb
\end{align*}
for all $n^+ \in \mathfrak{n}^+$ and $n^- \in \mathfrak{n}^-$. Note that $\exp: \mathfrak{n}^- \to N^-$ is surjective since $N^- \cong \mathbb R^{n - 1}$. Examining the formula above term by term, our first objective follows if we show that there exist $n_1^+, n_2^+, \dotsc, n_{j_{\mathrm{m}}}^+ \in \mathfrak{n}^+$ and $n_1^-, n_2^-, \dotsc, n_{j_{\mathrm{m}}}^- \in \mathfrak{n}^-$ for some $j_{\mathrm{m}} \in \mathbb N$ such that $\Span\{[n_j^-, n_j^+] \in \mathfrak{g}: j \in \{1, 2, \dotsc, j_{\mathrm{m}}\}\} = \mathfrak{a} \oplus \mathfrak{m}$ and $[n_j^-, [n_j^-, n_j^+]] \in \mathfrak{n}^-$ for all $1 \leq j \leq j_{\mathrm{m}}$. 

We use the upper half space model. Recall $M = \Stab_G(v_o) \cong \SO(n - 1)$ whose elements act on $\mathbb H^n$ by rotations in $\mathbb R^n$ which keep the $n$\textsuperscript{th} coordinate fixed. It is a fact that for any chosen basis of $\mathbb R^n$, any rotation can be expressed as a composition of planar rotations where the planes are generated by any two distinct basis vectors. Using our standard basis $(e_1, e_2, \dotsc, e_n)$, this means that $M$ is generated by the subgroups $M_{j, k} = \{m \in M: m(e_l) = e_l \text{ for all } l \in \{1, 2, \dotsc, n\} \setminus \{j, k\}\} \cong \SO(2)$ for all $1 \leq j, k \leq n - 1$ with $j \neq k$. Then, we have the corresponding sum of vector spaces $\mathfrak{m} = \sum_{\substack{1 \leq j, k \leq n - 1\\ j \neq k}} \mathfrak{m}_{j, k}$ where $\mathfrak{m}_{j, k} \cong \mathfrak{so}(2) \cong \mathbb R$ is the Lie algebra of $M_{j, k}$ for all $1 \leq j, k \leq n - 1$ with $j \neq k$.

Now, let $1 \leq j, k \leq n - 1$ be with $j \neq k$ and consider the totally geodesic submanifold $P_{j, k} = \Span(e_j, e_k, e_n) \cap \mathbb H^n \subset \mathbb H^n$. Let $H_{j, k} = \{g \in G: g(P_{j, k}) = P_{j, k}\} < G$ be the subgroup of isometries of $P_{j, k}$. Then, $P_{j, k} \cong \mathbb H^3$ is the upper half space of $\Span(e_j, e_k, e_n) \cong \mathbb R^3$ which induces the canonical identifications of Lie groups and their Lie algebras, $H_{j, k} \cong \PSL_2(\mathbb C)$ and $\mathfrak{h}_{j, k} \cong \mathfrak{sl}_2(\mathbb C)$. Let $N_{j, k}^+ = \{n^+ \in N^+: n^+ K \in P_{j, k}\} = N^+ \cap H_{j, k}$ and $N_{j, k}^- = \{n^- \in N^-: n^- K \in P_{j, k}\} = N^- \cap H_{j, k}$ be the expanding and contracting horospherical subgroups of $H_{j, k}$, with corresponding Lie algebras $\mathfrak{n}_{j, k}^+ \subset \mathfrak{n}^+$ and $\mathfrak{n}_{j, k}^- \subset \mathfrak{n}^-$ respectively. Note that $M_{j, k}, N_{j, k}^+, N_{j, k}^- < H_{j, k}$ and hence also $\mathfrak{m}_{j, k}, \mathfrak{n}_{j, k}^+, \mathfrak{n}_{j, k}^- \subset \mathfrak{h}_{j, k}$. Now, it suffices to show that $[\mathfrak{n}_{j, k}^-, \mathfrak{n}_{j, k}^+] = \mathfrak{a} \oplus \mathfrak{m}_{j, k}$ and $[\mathfrak{n}_{j, k}^-, \mathfrak{a} \oplus \mathfrak{m}_{j, k}] = \mathfrak{n}_{j, k}^-$. This is simply a matter of calculations. Using the canonical identifications, we can explicitly write the Lie algebras as
\begin{align*}
\mathfrak{a} &\cong
\left\{
\begin{pmatrix}
t & 0 \\
0 & -t
\end{pmatrix}
: t \in \mathbb R\right\};
&
\mathfrak{m}_{j, k} &\cong
\left\{
\begin{pmatrix}
i\theta & 0 \\
0 & -i\theta
\end{pmatrix}
: \theta \in \mathbb R\right\}; \\
\mathfrak{n}_{j, k}^+ &\cong
\left\{
\begin{pmatrix}
0 & 0 \\
n^+ & 0
\end{pmatrix}
: n^+ \in \mathbb C\right\};
&
\mathfrak{n}_{j, k}^- &\cong
\left\{
\begin{pmatrix}
0 & n^- \\
0 & 0
\end{pmatrix}
: n^- \in \mathbb C\right\}.
\end{align*}
Now for all $n^+ \in \mathfrak{n}_{j, k}^+$ and $n^- \in \mathfrak{n}_{j, k}^-$, using the corresponding matrices
$
\left(
\begin{smallmatrix}
0 & n^- \\
0 & 0
\end{smallmatrix}
\right)
$
and
$
\left(
\begin{smallmatrix}
0 & 0 \\
n^+ & 0
\end{smallmatrix}
\right)
$,
we have the calculation
\begin{align*}
\begin{pmatrix}
0 & n^- \\
0 & 0
\end{pmatrix}
\begin{pmatrix}
0 & 0 \\
n^+ & 0
\end{pmatrix}
-
\begin{pmatrix}
0 & 0 \\
n^+ & 0
\end{pmatrix}
\begin{pmatrix}
0 & n^- \\
0 & 0
\end{pmatrix}
=
\begin{pmatrix}
n^- n^+ & 0 \\
0 & -n^- n^+
\end{pmatrix}
\end{align*}
for the matrix corresponding to $[n^-, n^+]$. Thus, $[\mathfrak{n}_{j, k}^-, \mathfrak{n}_{j, k}^+] = \mathfrak{a} \oplus \mathfrak{m}_{j, k}$. Similarly, for all $n^- \in \mathfrak{n}_{j, k}^-$ and $a + m \in \mathfrak{a} \oplus \mathfrak{m}_{j, k}$, using the corresponding matrices
$
\left(
\begin{smallmatrix}
0 & n^- \\
0 & 0
\end{smallmatrix}
\right)
$
and
$
\left(
\begin{smallmatrix}
z & 0 \\
0 & -z
\end{smallmatrix}
\right)
$,
we have the calculation
\begin{align*}
\begin{pmatrix}
0 & n^- \\
0 & 0
\end{pmatrix}
\begin{pmatrix}
z & 0 \\
0 & -z
\end{pmatrix}
-
\begin{pmatrix}
z & 0 \\
0 & -z
\end{pmatrix}
\begin{pmatrix}
0 & n^- \\
0 & 0
\end{pmatrix}
=
\begin{pmatrix}
0 & -2zn^- \\
0 & 0
\end{pmatrix}
\end{align*}
for the matrix corresponding to $[n^-, a + m]$. Thus, $[\mathfrak{n}_{j, k}^-, \mathfrak{a} \oplus \mathfrak{m}_{j, k}] = \mathfrak{n}_{j, k}^-$.

Now, we show that we can choose $n_1^-, n_2^-, \dotsc, n_{j_{\mathrm{m}}}^-$ to be in $N_1^-$. By way of contradiction, suppose this is false. Then, $V = \sum_{n^- \in N_1^-} \pi(\Ad_{n^-}(\mathfrak{n}^+)) \subset \mathfrak{a} \oplus \mathfrak{m}$ is a proper subspace. Hence, there is a functional $L: \mathfrak{a} \oplus \mathfrak{m} \to \mathbb R$ with $\ker(L) = V$. But we have already proved that $\sum_{n^- \in N^-} \pi(\Ad_{n^-}(\mathfrak{n}^+)) = \mathfrak{a} \oplus \mathfrak{m}$ and so we can choose $\hat{n}^+ \in \mathfrak{n}^+$ and $\hat{n}^- \in N^-$ such that $\pi(\Ad_{\hat{n}^-}(\hat{n}^+)) \in \mathfrak{a} \oplus \mathfrak{m} \setminus V$. Let $z_1 \in \tilde{R}_1$ be the center. Without loss of generality, we can assume $\mathsf{F}(\tilde{z}_1) = e$. Consider the map $N^- \to \mathbb R^{n - 1}$ defined by $n^- \mapsto (n^-)^-$ which is just mapping the frame $n^-$ to its backward limit point $(n^-)^- \in \mathbb R^{n - 1} \subset \partial_\infty(\mathbb H^n)$. The inverse of the described map is a Lie group isomorphism $\mathbb R^{n - 1} \to N^-$. Since $\exp: \mathbb R^{n - 1} \to \mathbb R^{n - 1}$ is simply the identity map, the previous Lie group isomorphism induces the Lie algebra isomorphism $\mathbb R^{n - 1} \to \mathfrak{n}^-$ where we can still view the domain as the boundary at infinity $\mathbb R^{n - 1} \subset \partial_\infty(\mathbb H^n)$. Denote the image of $x \in \mathbb R^{n - 1}$ under this map by $n_x^- \in \mathfrak{n}^-$. Now consider the function $P: \mathbb R^{n - 1} \to \mathbb R$ defined by $P(x) = L\left(\pi\left(\Ad_{e^{n_x^-}}(\hat{n}^+)\right)\right)$ for all $x \in \mathbb R^{n - 1}$. Since using the basis $(e_1, e_2, \dotsc, e_n)$ above was arbitrary, we can see from the calculations that in fact $\pi\big(\Ad_{e^{n^-}}(n^+)\big) = [n^-, n^+]$ for all $n^+ \in \mathfrak{n}^+$ and $n^- \in \mathfrak{n}^-$, and so in particular $\pi\left(\Ad_{e^{n_x^-}}(\hat{n}^+)\right) = [n_x^-, \hat{n}^+] = -\ad_{\hat{n}^+}(n_x^-)$ for all $x \in \mathbb R^{n - 1}$. Then $P(x) = -L(\ad_{\hat{n}^+}(n_x^-))$ for all $x \in \mathbb R^{n - 1}$ which is a composition of linear maps. Now, $\Lambda(\Gamma) \cap B_\epsilon^{\mathrm{E}}((\tilde{z}_1)^-) \subset \ker(P)$ for some $\epsilon > 0$, but $P$ is nontrivial because $P(x_0) \neq 0$ where $x_0 \in \mathbb R^{n - 1}$ such that $\hat{n}^- = e^{n_{x_0}^-}$. But this is a contradiction by \cite[Proposition 3.12]{Win15} since $\Gamma < G$ is Zariski dense. Finally, it is clear that $n_1, n_2, \dotsc, n_{j_{\mathrm{m}}}$ satisfying the result is an open set and so the lemma follows.
\end{proof}

We fix $j_{\mathrm{m}} \in \mathbb N$ as in \cref{lem:am_ProjectionOfAdjointImage} for the rest of the paper.

The following proposition is the required LNIC in our setting.

\begin{proposition}[LNIC]
\label{pro:FrameFlowLNIC}
There exist $\epsilon \in (0, 1)$, $m_0 \in \mathbb N$, $j_{\mathrm{m}} \in \mathbb N$, and an open neighborhood $\mathcal{U} \subset \tilde{U}_1$ of the center $z_1 \in \tilde{R}_1$ with $\mathcal{U} \cap \Omega \subset U_1$ such that for all $m \geq m_0$, there exist sections $v_j = \sigma^{-\alpha_j}: \tilde{U}_1 \to \tilde{U}_{\alpha_{j, 0}}$ for some admissible sequences $\alpha_j = (\alpha_{j, 0}, \alpha_{j, 1}, \dotsc, \alpha_{j, m - 1}, 1)$ for all integers $0 \leq j \leq j_{\mathrm{m}}$ such that for all $u \in \mathcal{U}$ and $\omega \in \mathfrak{a} \oplus \mathfrak{m}$ with $\|\omega\| = 1$, there exist $1 \leq j \leq j_{\mathrm{m}}$ and $Z \in \T_u(\tilde{U}_1)$ with $\|Z\| = 1$ such that
\begin{align*}
|\langle (d\BP_{j, u})_u(Z), \omega \rangle| \geq \epsilon
\end{align*}
where we define $\BP_j: \tilde{U}_1 \times \tilde{U}_1 \to AM$ by
\begin{align*}
\BP_j(u, u') = \Phi^{\alpha_0}(v_0(u))^{-1} \Phi^{\alpha_0}(v_0(u')) \Phi^{\alpha_j}(v_j(u'))^{-1} \Phi^{\alpha_j}(v_j(u))
\end{align*}
and we denote $\BP_{j, u} = \BP_j(u, \cdot)$ for all $u, u' \in \tilde{U}_1$ and $1 \leq j \leq j_{\mathrm{m}}$. Moreover, $v_0(\mathcal{U}), v_1(\mathcal{U}), \dotsc, v_{j_{\mathrm{m}}}(\mathcal{U})$ are mutually disjoint.
\end{proposition}

\begin{proof}
By \cref{lem:am_ProjectionOfAdjointImage}, there exist distinct $n_1^-, n_2^-, \dotsc, n_{j_{\mathrm{m}}}^- \in N_1^-$ for some $j_{\mathrm{m}} \in \mathbb N$ and $\delta_1 > 0$ such that for all $\eta_1^-, \eta_2^-, \dotsc, \eta_{j_{\mathrm{m}}}^- \in N_1^-$ with $d_{N^-}(\eta_j^-, n_j^-) \leq \delta_1$ for all $1 \leq j \leq j_{\mathrm{m}}$, we have
\begin{align*}
\sum_{j = 1}^{j_{\mathrm{m}}} \pi\left(\Ad_{\eta_j^-}(\mathfrak{n}^+)\right) = \mathfrak{a} \oplus \mathfrak{m}.
\end{align*}
This property allows us to define a positive constant
\begin{align}
\label{eqn:Epsilon1UsingLargeProjectionProperty}
\epsilon_1 ={}&\inf_{\substack{\eta_1^-, \eta_2^-, \dotsc, \eta_{j_{\mathrm{m}}}^- \in N_1^- \text{ such that}\\ d_{N^-}(\eta_j^-, n_j^-) \leq \delta_1 \, \forall j \in \{1, 2, \dotsc, j_{\mathrm{m}}\}}} \inf_{\substack{\omega \in \mathfrak{a} \oplus \mathfrak{m}\\ \text{such that}\\ \|\omega\| = 1}} \sup_{\substack{n^+ \in \mathfrak{n}^+\\ \text{such that } \|n^+\| = 1,\\ j \in \{1, 2, \dotsc, j_{\mathrm{m}}\}}} \left|\left\langle \pi\left(\Ad_{\eta_j^-}(n^+)\right), \omega \right\rangle\right|.
\end{align}
Let $z_1 \in \tilde{R}_1$ be the center. Define the diffeomorphism $\psi: \tilde{U}_1 \to N_1^+$ by $\psi(u) = n^+$ such that $F(u) = F(z_1)n^+$ for all $u \in \tilde{U}_1$. Fix $C_1 = \|(d\psi)_{z_1}\|_{\mathrm{op}} > 0$. Recall $h_{n^-}: N_1^+ \to N^+$ from \cref{lem:BrinPesinDerivativeImageIsAdjointProjection} which is a diffeomorphism onto its image which is also smooth in $n^- \in N_{1, \epsilon_0}^-$ and satisfies $h_e = \Id_{N_1^+}$. Since $N_1^-$ is compact, there exists $C_2 > 1$ such that $\frac{1}{C_2} \leq \|(dh_{n^-})_e\|_{\mathrm{op}} \leq C_2$ for all $n^- \in N_1^-$. Fix $\epsilon \in \left(0, \min\left(\frac{C_1\epsilon_1}{4C_2}, 1\right)\right)$ and $\epsilon_2 \in \left(0, \frac{\epsilon}{C_1C_2}\right)$. Observe that $(\pi \circ \Ad_{n^-})|_{\mathfrak{n}^+}: \mathfrak{n}^+ \to \mathfrak{a} \oplus \mathfrak{m}$ is linear and also smooth in $n^- \in N_{1, \epsilon_0}^-$. Since $(\pi \circ \Ad_e)|_{\mathfrak{n}^+} = 0$, it follows that there exists $\delta_2 > 0$ such that $\|(\pi \circ \Ad_{n^-})|_{\mathfrak{n}^+}\|_{\mathrm{op}} \leq \epsilon_2$ for all $n^- \in N^-$ with $d_{N^-}(n^-, e) \leq \delta_2$. Now, using the Markov property and the topological mixing property of $T$, we can fix $m_0 \in \mathbb N$ such that given any $m \geq m_0$, there exist distinct $\eta_0^-, \eta_1^-, \dotsc, \eta_{j_{\mathrm{m}}}^- \in N_1^-$ with $d_{N^-}(\eta_0^-, e) < \delta_2$ and $d_{N^-}(\eta_j^-, n_j^-) < \delta_1$ for all $1 \leq j \leq j_{\mathrm{m}}$, such that $\mathcal{P}^m(u_j) = s_j \in S_1$ with $F(s_j) = F(z_1)\eta_j^-$, for some $u_j \in \hat{U}_{\alpha_{j, 0}}$, for some $\alpha_{j, 0} \in \mathcal{A}$. Then for all $0 \leq j \leq j_{\mathrm{m}}$, the associated trajectory of the geodesic flow of $u_j$ through the Markov section gives a corresponding admissible sequence $\alpha_j = (\alpha_{j, 0}, \alpha_{j, 1}, \dotsc, \alpha_{j, m - 1}, 1)$. By increasing $m_0$ if neccessary, we can assume that the admissible sequences $\alpha_0, \alpha_1, \dotsc, \alpha_{j_{\mathrm{m}}}$ are distinct. We define the sections $v_j = \sigma^{-\alpha_j}: \tilde{U}_1 \to \tilde{U}_{\alpha_{j, 0}}$ for all $0 \leq j \leq j_{\mathrm{m}}$. We will show that $v_0, v_1, \dotsc, v_{j_{\mathrm{m}}}$ are the required sections. We define $\BP_j: \tilde{U}_1 \times \tilde{U}_1 \to AM$ for all $1 \leq j \leq j_{\mathrm{m}}$ as in the proposition. The equation
\begin{align*}
F(\mathcal{P}^m(\sigma^{-\alpha_j}(z_1))) = F(\mathcal{P}^m(v_j(z_1))) = F(\mathcal{P}^m(u_j)) = F(s_j) = F(z_1)\eta_j^-
\end{align*}
implies that $\eta_j^- = n_{\alpha_j}$ for all $0 \leq j \leq j_{\mathrm{m}}$. Hence by \cref{lem:BrinPesinInTermsOfHolonomy}, we have
\begin{align*}
\BP_j(u, u') = \Xi_{\eta_0^-}(\psi(u))^{-1}\Xi_{\eta_0^-}(\psi(u'))\Xi_{\eta_j^-}(\psi(u'))^{-1}\Xi_{\eta_j^-}(\psi(u))
\end{align*}
for all $u, u' \in \tilde{U}_1$ and $1 \leq j \leq j_{\mathrm{m}}$. Starting with the case $u = z_1 \in \tilde{U}_1$, we have
\begin{align*}
\BP_{j, z_1}(u') = \Xi_{\eta_0^-}(\psi(u'))\Xi_{\eta_j^-}(\psi(u'))^{-1} \qquad \text{for all $u' \in \tilde{U}_1$ and $1 \leq j \leq j_{\mathrm{m}}$}.
\end{align*}
By \cref{lem:BrinPesinDerivativeImageIsAdjointProjection}, the differential $(d\BP_{j, z_1})_{z_1}: \T_{z_1}(\tilde{U}_1) \to \mathfrak{a} \oplus \mathfrak{m}$ is given by
\begin{align*}
&(d\BP_{j, z_1})_{z_1}(Z) \\
={}&\left(d\Xi_{\eta_0^-}\right)_e((d\psi)_{z_1}(Z)) - \left(d\Xi_{\eta_j^-}\right)_e((d\psi)_{z_1}(Z)) \\
={}&\pi\left(\Ad_{\eta_0^-}\left(\left(dh_{\eta_0^-}\right)_e((d\psi)_{z_1}(Z))\right)\right) - \pi\left(\Ad_{\eta_j^-}\left(\left(dh_{\eta_j^-}\right)_e((d\psi)_{z_1}(Z))\right)\right)
\end{align*}
for all $Z \in \T_{z_1}(\tilde{U}_1)$ and $1 \leq j \leq j_{\mathrm{m}}$. Define $S^1_{\mathfrak{a} \oplus \mathfrak{m}} = \{\omega \in \mathfrak{a} \oplus \mathfrak{m}: \|\omega\| = 1\}$ and let $\omega \in S^1_{\mathfrak{a} \oplus \mathfrak{m}}$. Using the above formula, we have
\begin{align}
\begin{aligned}
\label{eqn:PreliminaryNLIC_For_u=z_1}
|\langle (d\BP_{j, z_1})_{z_1}(Z), \omega \rangle| \geq{}&\left|\left\langle \pi\left(\Ad_{\eta_j^-}\left(\left(dh_{\eta_j^-}\right)_e((d\psi)_{z_1}(Z))\right)\right), \omega \right\rangle\right| \\
&{}- \left|\left\langle \pi\left(\Ad_{\eta_0^-}\left(\left(dh_{\eta_0^-}\right)_e((d\psi)_{z_1}(Z))\right)\right), \omega \right\rangle\right|
\end{aligned}
\end{align}
for all $Z \in \T_{z_1}(\tilde{U}_1)$ and $1 \leq j \leq j_{\mathrm{m}}$. We will show that there is a choice of $1 \leq j_\omega \leq j_{\mathrm{m}}$ and $Z_\omega \in \T_{z_1}(\tilde{U}_1)$ with $\|Z_\omega\| = 1$ such that $|\langle (d\BP_{j_\omega, z_1})_{z_1}(Z_\omega), \omega \rangle| \geq 3\epsilon$. We first deal with the first term in \cref{eqn:PreliminaryNLIC_For_u=z_1}. Since $d_{N^-}(\eta_j^-, n_j^-) < \delta_1$ for all $1 \leq j \leq j_{\mathrm{m}}$, using \cref{eqn:Epsilon1UsingLargeProjectionProperty}, there exists $1 \leq j_\omega \leq j_{\mathrm{m}}$ and $n_\omega^+ \in \mathfrak{n}^+$ with $\|n_\omega^+\| = 1$ such that $\left|\left\langle \pi\left(\Ad_{\eta_{j_\omega}^-}(n_\omega^+)\right), \omega \right\rangle\right| \geq \epsilon_1$. Since $\left(dh_{\eta_{j_\omega}^-}\right)_e$ and $(d\psi)_{z_1}$ are invertible linear maps, there exists $Z_\omega \in \T_{z_1}(\tilde{U}_1)$ with $\|Z_\omega\| = 1$ such that $\left(dh_{\eta_{j_\omega}^-}\right)_e((d\psi)_{z_1}(Z_\omega))$ is a scalar multiple of $n_\omega^+$. The operator norm bounds on the linear maps give
\begin{align*}
\left|\left\langle \pi\left(\Ad_{\eta_{j_\omega}^-}\left(\left(dh_{\eta_{j_\omega}^-}\right)_e((d\psi)_{z_1}(Z_\omega))\right)\right), \omega \right\rangle\right| \geq C_1C_2^{-1}\epsilon_1 \geq 4\epsilon.
\end{align*}
Now we turn to the second term in \cref{eqn:PreliminaryNLIC_For_u=z_1}. Since $d_{N^-}(\eta_0^-, e) < \delta_2$, we again use the operator norm bounds to get
\begin{align*}
&\left|\left\langle \pi\left(\Ad_{\eta_0^-}\left(\left(dh_{\eta_0^-}\right)_e((d\psi)_{z_1}(Z_\omega))\right)\right), \omega \right\rangle\right| \\
\leq{}&\left\|\pi\left(\Ad_{\eta_0^-}\left(\left(dh_{\eta_0^-}\right)_e((d\psi)_{z_1}(Z_\omega))\right)\right)\right\| \leq C_1C_2\epsilon_2 \leq \epsilon.
\end{align*}
With these bounds, we conclude that for all $\omega \in S^1_{\mathfrak{a} \oplus \mathfrak{m}}$, there exist $1 \leq j_\omega \leq j_{\mathrm{m}}$ and $Z_\omega \in \T_{z_1}(\tilde{U}_1)$ with $\|Z_\omega\| = 1$ such that $|\langle (d\BP_{j_\omega, z_1})_{z_1}(Z_\omega), \omega \rangle| \geq 3\epsilon$ as desired. Since the map $S^1_{\mathfrak{a} \oplus \mathfrak{m}} \to \mathbb R$ defined by $\omega \mapsto |\langle \omega', \omega \rangle|$ is continuous for all $\omega' \in \mathfrak{a} \oplus \mathfrak{m}$ and $S^1_{\mathfrak{a} \oplus \mathfrak{m}}$ is compact, there exist $\omega_1, \omega_2, \dotsc, \omega_{k_0} \in S^1_{\mathfrak{a} \oplus \mathfrak{m}}$ for some $k_0 \in \mathbb N$ contained in corresponding open sets $V_1, V_2, \dotsc, V_{k_0} \subset S^1_{\mathfrak{a} \oplus \mathfrak{m}}$ which cover $S^1_{\mathfrak{a} \oplus \mathfrak{m}}$ such that $|\langle (d\BP_{j_{\omega_k}, z_1})_{z_1}(Z_{\omega_k}), \omega \rangle| \geq 2\epsilon$ for all $\omega \in \overline{V_k}$ and $1 \leq k \leq k_0$. Now, for all $1 \leq k \leq k_0$, extend $Z_{\omega_k} \in \T_{z_1}(\tilde{U}_1)$ to any smooth unit vector field $Z_k: \mathcal{U}_k \to \T(\tilde{U}_1)$ for some open set $\mathcal{U}_k \subset \tilde{U}_1$ containing $z_1$, i.e., $Z_k$ satisfies $Z_k(z_1) = Z_{\omega_k}$ and $\|Z_k(u)\| = 1$ for all $u \in \mathcal{U}_k$. Since $\BP_{j_{\omega_k}}$ is smooth, the map $\mathcal{U}_k \times \overline{V_k} \to \mathbb R$ defined by $(u, \omega) \mapsto |\langle (d\BP_{j_{\omega_k}, u})_u(Z_k(u)), \omega \rangle|$ is continuous for all $1 \leq k \leq k_0$. Hence for all $1 \leq k \leq k_0$, by compactness of $\overline{V_k}$, there exists an open subset $\mathcal{U}_k' \subset \mathcal{U}_k$ containing $z_1$ such that we can extend the inequality to
\begin{align*}
|\langle (d\BP_{j_{\omega_k}, u})_u(Z_k(u)), \omega \rangle| \geq \epsilon \qquad \text{for all $u \in \mathcal{U}_k'$ and $\omega \in \overline{V_k}$}.
\end{align*}
Take $\mathcal{U}' = \bigcap_{k = 1}^{k_0} \mathcal{U}_k'$ which almost satisfies all the requirements. We make one last adjustment so that the last property holds. Since $v_0(z_1), v_1(z_1), \dotsc, v_{j_{\mathrm{m}}}(z_1)$ are distinct, there are open sets $v_j(z_1) \in W_j \subset v_j(\tilde{U}_1)$ for all $0 \leq j \leq j_{\mathrm{m}}$ which are mutually disjoint. The proposition follows by taking $\mathcal{U} = \mathcal{U}' \cap \big(\bigcap_{j = 0}^{j_{\mathrm{m}}} v_j^{-1}(W_j)\big)$ and shrinking it so that $\mathcal{U} \cap \Omega \subset U_1$.
\end{proof}

Fix $\varepsilon_2 \in (0, 1)$, $m_0 \in \mathbb N$, $j_{\mathrm{m}} \in \mathbb N$, and the open subset $\mathcal{U} \subset \tilde{U}_1$ containing the center $z_1 \in \tilde{R}_1$ to be the $\epsilon$, $m_0$, $j_{\mathrm{m}}$, and $\mathcal{U}$ provided by \cref{pro:FrameFlowLNIC}.

\subsection{Non-concentration property}
\label{subsec:Non-ConcentrationProperty}
In the upper half space model, applying an appropriate isometry, we assume that the vectors in $\tilde{\mathsf{U}}_1$ have direction $\pi_2(\tilde{\mathsf{U}}_1) = -e_n$ and their basepoints lie on the hyperplane $\langle \pi_1(\tilde{\mathsf{U}}_1), e_n\rangle = 1$. We will often view the limit set as $\Lambda(\Gamma) \subset \mathbb R^{n - 1} \cup \{\infty\}$ in the rest of the paper. The following \cref{pro:NonConcentrationProperty} is the required NCP.

\begin{proposition}[NCP]
\label{pro:NonConcentrationProperty}
There exists $\delta \in (0, 1)$ such that for all $\epsilon \in (0, 1)$, $w \in \mathbb R^{n - 1}$ with $\|w\| = 1$, and $x \in \Lambda(\Gamma) \cap \mathbb R^{n - 1}$, there exists $y \in \Lambda(\Gamma) \cap B_\epsilon^{\mathrm{E}}(x)$ such that $|\langle y - x, w \rangle| \geq \epsilon \delta$.
\end{proposition}

\begin{proof}
By way of contradiction, suppose the proposition is false. Then for all $j \in \mathbb N$, taking $\delta_j = \frac{1}{j}$, there exist $\epsilon_j \in (0, 1)$, $w_j \in \mathbb R^{n - 1}$ with $\|w_j\| = 1$, and $x_j \in \Lambda(\Gamma) \cap \mathbb R^{n - 1}$ such that $|\langle y - x_j, w_j \rangle| \leq \epsilon_j \delta_j = \frac{\epsilon_j}{j}$ for all $y \in \Lambda(\Gamma) \cap B_{\epsilon_j}^{\mathrm{E}}(x_j)$. Hence, we can rewrite this as
\begin{align}
\label{eqn:IfLemmaIsFalse}
\Lambda(\Gamma) \cap B_{\epsilon_j}^{\mathrm{E}}(x_j) \subset \left\{y \in \mathbb R^{n - 1}: |\langle y - x_j, w_j \rangle| \leq \frac{\epsilon_j}{j}\right\} \qquad \text{for all $j \in \mathbb N$}.
\end{align}

We want to use the self-similarity property of the fractal set $\Lambda(\Gamma)$. Note that $A < G$ is such that $a_t$ acts on $\overline{\mathbb H^n}$ by dilation by $e^t$ for all $t \in \mathbb R$, and elements of $N^-$ act on $\overline{\mathbb H^n}$ by translation. For all $x \in \mathbb R^{n - 1}$, denote by $n_x^- \in N^-$ the element which acts on $\overline{\mathbb H^n}$ by translation by $x$. For all $j \in \mathbb N$, we have $(n_{x_j}^-)^+ = \infty \in \Lambda(\Gamma)$ and $(n_{x_j}^-)^- = x_j \in \Lambda(\Gamma)$, and hence $\Gamma n_{x_j}^- \in \Omega_{\F} \subset \Gamma \backslash G$ which we recall is compact. Also recalling that $\Omega_{\F}$ is invariant under the frame flow, we have $n_{x_j}^- a_t \in \Gamma \Omega_0$ for all $t \in \mathbb R$ and $j \in \mathbb N$, where $\Omega_0 \subset G$ is some compact subset. Then for all $j \in \mathbb N$, setting $t_j = \log(\epsilon_j)$, there exist $\gamma_j \in \Gamma$ and $g_j \in \Omega_0$ such that $n_{x_j}^- a_{t_j} = \gamma_j g_j$. Now for all $j \in \mathbb N$, we have $g_j a_{-t_j} n_{-x_j}^- = \gamma_j^{-1}$ whose action on $\partial_\infty(\mathbb H^n)$ preserves $\Lambda(\Gamma)$. This captures the self similarity property of the fractal set $\Lambda(\Gamma)$.

Now, applying $g_j a_{-t_j} n_{-x_j}^-$ in \cref{eqn:IfLemmaIsFalse} gives
\begin{align*}
\Lambda(\Gamma) \cap g_j(B_1^{\mathrm{E}}(0)) \subset g_j\left(\left\{y \in \mathbb R^{n - 1}: |\langle y, w_j \rangle| \leq \frac{1}{j}\right\}\right) \qquad \text{for all $j \in \mathbb N$}.
\end{align*}
By compactness, we can pass to subsequences such that without loss of generality we can assume that $\lim_{j \to \infty} w_j = w \in \mathbb R^{n - 1}$ with $\|w\| = 1$ and $\lim_{j \to \infty} g_j = g \in \Omega_0$. Then in the limit $j \to \infty$, we have $\Lambda(\Gamma) \cap g(B_1^{\mathrm{E}}(0)) \subset g\left(\left\{y \in \mathbb R^{n - 1}: |\langle y, w \rangle| \leq 0\right\}\right)$ which contradicts \cite[Proposition 3.12]{Win15} since $\Gamma < G$ is Zariski dense.
\end{proof}

Fix $\varepsilon_3 \in (0, 1)$ to be the $\delta$ provided by \cref{pro:NonConcentrationProperty} henceforth.

\section{Preliminary lemmas and constants}
\label{PreliminaryLemmasAndConstants}
In this section, we cover some more lemmas and then fix many constants which are needed to construct the Dolgopyat operators and prove \cref{thm:FrameFlowDolgopyat}.

Let $\Psi_1: \tilde{\mathsf{U}}_1 \to \mathbb R^{n - 1}$ be the diffeomorphism defined by $\Psi_1(u) = u^+$ for all $u \in \tilde{\mathsf{U}}_1$. Let $\Psi_2: \tilde{\mathsf{U}}_1 \to \tilde{U}_1$ be the isometry obtained from the covering map. Define the diffeomorphism $\Psi: \Psi_1(\tilde{\mathsf{U}}_1) \to \tilde{U}_1$ by $\Psi(x) = \Psi_2(\Psi_1^{-1}(x))$ for all $x \in \Psi_1(\tilde{\mathsf{U}}_1)$. Then $(d\Psi)_x^*$ is invertible for all $x \in \Psi_1(\tilde{\mathsf{U}}_1)$ and hence by continuity, we can fix $\delta_\Psi > 0$ such that $\inf_{x \in \Psi_1(\tilde{\mathsf{U}}_1)} \inf_{\|w\| = 1} \|(d\Psi)_x^*(w)\| \geq \delta_\Psi$. We also fix $C_\Psi > 1$ such that $\frac{1}{C_\Psi} d_{\mathrm{E}}(x, y) \leq d(\Psi(x), \Psi(y)) \leq C_\Psi d_{\mathrm{E}}(x, y)$ for all $x, y \in \Psi_1(\tilde{\mathsf{U}}_1)$.

We now introduce a technical lemma. Denote $\check{x} = \Psi^{-1}(x)$ for all $x \in \tilde{U}_1$. Let $x, y \in \tilde{U}_1$, $z = (\check{x}, \check{y} - \check{x}) \in \T_{\check{x}}(\mathbb R^{n - 1})$ such that $\{\check{x} + tz \in \mathbb R^{n - 1}: t \in [0, 1]\} \subset \Psi^{-1}(\tilde{U}_1)$, and $1 \leq j \leq j_{\mathrm{m}}$. Define the curve $\varphi^{\mathrm{BP}}_{j, x, z}: [0, 1] \to AM$ by $\varphi^{\mathrm{BP}}_{j, x, z}(t) = \BP_{j, x}(\Psi(\check{x} + tz))$ for all $t \in [0, 1]$. Note that the curve has endpoints $\varphi^{\mathrm{BP}}_{j, x, z}(0) = e$ and $\varphi^{\mathrm{BP}}_{j, x, z}(1) = \BP_{j, x}(y) = \BP_j(x, y)$. There exists $\delta_0 > 0$ such that any pair of points in $B_{\delta_0}^{AM}(e) \subset AM$ has a unique geodesic through them. Fix $C_{\BP, \Psi} = \sup_{x, y \in \tilde{U}_1, j \in \{1, 2, \dotsc, j_{\mathrm{m}}\}} \|d(\BP_{j, x} \circ \Psi)_{\check{y}}\|_{\mathrm{op}}$.

\begin{lemma}
\label{lem:ComparingExpWithBP}
There exists $C > 0$ such that for all $1 \leq j \leq j_{\mathrm{m}}$ and $x, y \in \tilde{U}_1$ with $d(x, y) < \frac{\delta_0}{C_\Psi C_{\BP, \Psi}}$ such that $\{\check{x} + tz \in \mathbb R^{n - 1}: t \in [0, 1]\} \subset \Psi^{-1}(\tilde{U}_1)$, we have
\begin{align*}
d_{AM}\left(\exp(Z), \varphi_{j, x, z}^{\mathrm{BP}}(1)\right) \leq C d(x, y)^2
\end{align*}
where $z = (\check{x}, \check{y} - \check{x}) \in \T_{\check{x}}(\mathbb R^{n - 1})$, and $Z = d(\BP_{j, x} \circ \Psi)_{\check{x}}(z)$.
\end{lemma}

\begin{proof}
Fix $\delta = \frac{\delta_0}{C_{\BP, \Psi}}$. Let $1 \leq j \leq j_{\mathrm{m}}$ and $x, y \in \tilde{U}_1$ with $d(x, y) < \frac{\delta}{C_\Psi}$ such that $\{\check{x} + tz \in \mathbb R^{n - 1}: t \in [0, 1]\} \subset \Psi^{-1}(\tilde{U}_1)$ where $z = (\check{x}, \check{y} - \check{x}) \in \T_{\check{x}}(\mathbb R^{n - 1})$. Note that the bound implies $\|z\| = d_{\mathrm{E}}(\check{x}, \check{y}) \leq \delta$. Let $\hat{z} = \frac{z}{\|z\|}$, $Z = d(\BP_{j, x} \circ \Psi)_{\check{x}}(z)$, and $\hat{Z} = \frac{Z}{\|z\|} = d(\BP_{j, x} \circ \Psi)_{\check{x}}(\hat{z})$. Define $L: [0, \|z\|] \to \mathbb R$ by $L(t) = d_{AM}\left(\exp(t\hat{Z}), \varphi_{j, x, \hat{z}}^{\mathrm{BP}}(t)\right)$ for all $ t \in [0, \|z\|]$. We can fix $C_0 > 0$ such that $\frac{1}{2}\sup_{t \in [0, \|z\|]}|L''(t)| \leq C_0$ holds independently of the choice of $j$, $x$, and $y$ because the geodesics $\varphi_{j, x, \hat{z}}^{\mathrm{BP}}$ depend smoothly in $x$ and $\hat{z}$. Fix $C = C_0 C_\Psi^2$. Define $\tilde{L}: [0, \|z\|] \to \mathbb R$ by $\tilde{L}(t) = \frac{1}{2}L(t)^2 = \frac{1}{2}d_{AM}\left(\exp(t\hat{Z}), \varphi_{j, x, \hat{z}}^{\mathrm{BP}}(t)\right)^2$ for all $t \in [0, \|z\|]$. For all $t \in [0, \|z\|]$, define $\gamma_t: [0, 1] \to AM$ to be the unique constant speed geodesic with endpoints $\gamma_t(0) = \exp(t\hat{Z})$ and $\gamma_t(1) = \varphi_{j, x, \hat{z}}^{\mathrm{BP}}(t)$. Then we have
\begin{align*}
L(t) &= \int_0^1 \|\gamma_t'(s)\| \, ds; & \tilde{L}(t) &= \int_0^1 \frac{1}{2}\|\gamma_t'(s)\|^2 \, ds
\end{align*}
for all $t \in [0, \|z\|]$. Note that $L(0) = 0$ and $\tilde{L}(t) = \frac{1}{2}L(t)^2$ for all $t \in [0, \|z\|]$. Since $\tilde{L}$ is the energy along the variation of geodesics $\gamma: [0, 1] \times (0, \|z\|) \to AM$, we can use the first variation formula to get the derivative
\begin{align*}
\tilde{L}'(t_0) &= \left\langle \left.\frac{d}{dt}\right|_{t = t_0} \gamma_t(1), \gamma_{t_0}'(1) \right\rangle - \left\langle \left.\frac{d}{dt}\right|_{t = t_0} \gamma_t(0), \gamma_{t_0}'(0) \right\rangle \\
&= \bigl\langle \bigl(\varphi_{j, x, \hat{z}}^{\mathrm{BP}}\bigr)'(t_0), \gamma_{t_0}'(1) \bigr\rangle - \left\langle \left.\frac{d}{dt}\right|_{t = t_0} \exp(t\hat{Z}), \gamma_{t_0}'(0) \right\rangle
\end{align*}
for all $t_0 \in (0, \|z\|)$. Hence, we calculate that
\begin{align}
\label{eqn:DerivativeOfDistanceFunctionL}
L'(t_0) = \frac{\tilde{L}'(t_0)}{L(t_0)} = \left\langle \bigl(\varphi_{j, x, \hat{z}}^{\mathrm{BP}}\bigr)'(t_0), \frac{\gamma_{t_0}'(1)}{\|\gamma_{t_0}'(1)\|} \right\rangle - \left\langle \left.\frac{d}{dt}\right|_{t = t_0} \exp(t\hat{Z}), \frac{\gamma_{t_0}'(0)}{\|\gamma_{t_0}'(0)\|} \right\rangle
\end{align}
for all $t_0 \in (0, \|z\|)$. It is easy to see that $\bigl(\varphi_{j, x, \hat{z}}^{\mathrm{BP}}\bigr)'(0) = \left.\frac{d}{dt}\right|_{t = 0} \exp(t\hat{Z}) = \hat{Z}$. Using the distance function $d_{\T(AM)}$ on $\T(AM)$ induced by the Sasaki metric on $\T(AM)$, we see that
\begin{align*}
 \lim_{t_0 \searrow 0} d_{\T(AM)}\left(\frac{\gamma_{t_0}'(0)}{\|\gamma_{t_0}'(0)\|}, \frac{\gamma_{t_0}'(1)}{\|\gamma_{t_0}'(1)\|}\right) = \lim_{t_0 \searrow 0} L(t_0) = 0.
\end{align*}
It follows from \cref{eqn:DerivativeOfDistanceFunctionL} that $\lim_{t_0 \searrow 0} L'(t_0) = 0$. Taylor's theorem immediately gives
\begin{align*}
L(t) \leq \frac{1}{2}\left(\sup_{t_0 \in [0, \|z\|]}|L''(t_0)|\right)t^2 \leq C_0t^2
\end{align*}
for all $t \in [0, \|z\|]$. The lemma follows by taking $t = \|z\|$.
\end{proof}

Fix $C_{\exp, \BP} > 0$ to be the $C$ provided by \cref{lem:ComparingExpWithBP}.

\begin{remark}
Shrinking $\mathcal{U} \subset \tilde{U}_1$ if necessary, we can assume that $\Psi^{-1}(\mathcal{U}) \subset \mathbb R^{n - 1}$ is a convex open subset so that \cref{lem:ComparingExpWithBP} applies for our purposes.
\end{remark}

The following \cref{lem:SigmaHyperbolicity} is derived from the hyperbolicity of the geodesic flow.

\begin{lemma}
\label{lem:SigmaHyperbolicity}
There exist $c_0 \in (0, 1)$ and $\kappa_1 > \kappa_2 > 1$ such that for all $j \in \mathbb N$ and admissible sequences $\alpha = (\alpha_0, \alpha_1, \dotsc, \alpha_j)$, we have
\begin{align*}
\frac{c_0}{\kappa_1^j} \leq \|(d\sigma^{-\alpha})_u\|_{\mathrm{op}} \leq \frac{1}{c_0\kappa_2^j} \qquad \text{for all $u \in \tilde{U}_{\alpha_j}$}.
\end{align*}
\end{lemma}

We fix constants $c_0 \in (0, 1)$ and $\kappa_1 > \kappa_2 > 1$ as in \cref{lem:SigmaHyperbolicity} for the rest of the paper and use these inequalities without further comments.

Recall that Dolgopyat's method can be successfully carried out when the derivative of $\rho_b$ is large, which motivated the definition of $\widehat{M}_0(b_0)$ for all $b_0 > 0$. This criteria is ultimately manifested in \cref{lem:FrameFlowPreliminaryLogLipschitz} which is a Lasota--Yorke type inequality \cite{LY73}.

\begin{lemma}
\label{lem:FrameFlowPreliminaryLogLipschitz}
There exists $A_0 > 0$ such that for all $\xi \in \mathbb C$ with $|a| < a_0'$, if $(b, \rho) \in \widehat{M}_0(1)$, then for all $k \in \mathbb N$, we have
\begin{enumerate}
\item\label{itm:FrameFlowPreliminaryLogLipschitzProperty1}	if $h \in K_B(\tilde{U})$ for some $B > 0$, then we have $\tilde{\mathcal{L}}_a^k(h) \in K_{B'}(\tilde{U})$ where $B' = A_0\left(\frac{B}{\kappa_2^k} + 1\right)$;
\item\label{itm:FrameFlowPreliminaryLogLipschitzProperty2}	if $H \in \mathcal{V}_\rho(\tilde{U})$ and $h \in B(\tilde{U}, \mathbb R)$ satisfy $\|(dH)_u\|_{\mathrm{op}} \leq Bh(u)$ for all $u \in \tilde{U}$, for some $B > 0$, then we have
\begin{align*}
\left\|\left(d\tilde{\mathcal{M}}_{\xi, \rho}^k(H)\right)_u\right\|_{\mathrm{op}} \leq A_0\left(\frac{B}{\kappa_2^k}\tilde{\mathcal{L}}_a^k(h)(u) + \|\rho_b\|\tilde{\mathcal{L}}_a^k\|H\|(u)\right)
\end{align*}
for all $u \in \tilde{U}$.
\end{enumerate}
\end{lemma}

\begin{proof}
Fix $A_0 > \max\left(\frac{4T_0}{c_0(\kappa_2 - 1)}, \frac{2T_0}{\delta_{1, \varrho}c_0(\kappa_2 - 1)}, \frac{1}{c_0}\right)$. Let $\xi \in \mathbb C$ with $|a| < a_0'$, $(b, \rho) \in \widehat{M}_0(1)$, and $k \in \mathbb N$. To prove \cref{itm:FrameFlowPreliminaryLogLipschitzProperty1}, let $h \in K_B(\tilde{U})$ for some $B > 0$. Let $u \in \tilde{U}_{\alpha_k}$ for some $\alpha_k \in \mathcal{A}$. Let $z \in \T_u(\tilde{U})$ with $\|z\| = 1$. Taking the differential and using the product rule, we have
\begin{align*}
\left(d\tilde{\mathcal{L}}_a^k(h)\right)_u(z) ={}&\sum_{\alpha: \len(\alpha) = k} e^{f_\alpha^{(a)}(\sigma^{-\alpha}(u))} d\bigl(f_\alpha^{(a)} \circ \sigma^{-\alpha}\bigr)_u(z) \cdot h(\sigma^{-\alpha}(u)) \\
{}&+ \sum_{\alpha: \len(\alpha) = k} e^{f_\alpha^{(a)}(\sigma^{-\alpha}(u))} \cdot d(h \circ \sigma^{-\alpha})_u(z).
\end{align*}
We need to estimate $\big|d\bigl(f_\alpha^{(a)} \circ \sigma^{-\alpha}\bigr)_u(z)\big|$. From definitions, we have
\begin{align*}
f_\alpha^{(a)}(\sigma^{-\alpha}(u)) = \sum_{j = 0}^{k - 1} f_{(\alpha_j, \alpha_{j + 1})}^{(a)}(\sigma^{-(\alpha_j, \alpha_{j + 1}, \dotsc, \alpha_k)}(u))
\end{align*}
for all admissible sequences $\alpha$ with $\len(\alpha) = k$. Thus, we have the bound
\begin{align}
\begin{aligned}
\label{eqn:f_alpha^(a)DerivativeBound}
\big|d\bigl(f_\alpha^{(a)} \circ \sigma^{-\alpha}\bigr)_u(z)\big| &\leq \sum_{j = 0}^{k - 1} \left|f_{(\alpha_j, \alpha_{j + 1})}^{(a)}\right|_{C^1} \big|\sigma^{-(\alpha_j, \alpha_{j + 1}, \dotsc, \alpha_k)}\big|_{C^1} \\
&\leq \sum_{j = 0}^{k - 1} \frac{T_0}{c_0 \kappa_2^{k - j}} \leq \frac{T_0}{c_0(\kappa_2 - 1)} \leq \frac{A_0}{4}
\end{aligned}
\end{align}
for all admissible sequences $\alpha$ with $\len(\alpha) = k$. Using the bound, we get
\begin{align*}
\left|\left(d\tilde{\mathcal{L}}_a^k(h)\right)_u(z)\right| \leq{}&\sum_{\alpha: \len(\alpha) = k} e^{f_\alpha^{(a)}(\sigma^{-\alpha}(u))} h(\sigma^{-\alpha}(u)) \big|d\bigl(f_\alpha^{(a)} \circ \sigma^{-\alpha}\bigr)_u(z)\big| \\
{}&+ \sum_{\alpha: \len(\alpha) = k} e^{f_\alpha^{(a)}(\sigma^{-\alpha}(u))} \|(dh)_{\sigma^{-\alpha}(u)}\|_{\mathrm{op}} \|(d\sigma^{-\alpha})_u\|_{\mathrm{op}} \\
\leq{}&A_0\sum_{\substack{\alpha: \len(\alpha) = k\\ u' = \sigma^{-\alpha}(u)}} e^{f_\alpha^{(a)}(u')} h(u') + \frac{B}{c_0\kappa_2^k}\sum_{\substack{\alpha: \len(\alpha) = k\\ u' = \sigma^{-\alpha}(u)}} e^{f_\alpha^{(a)}(u')} h(u') \\
\leq{}&A_0\left(\frac{B}{\kappa_2^k} + 1\right)\tilde{\mathcal{L}}_a^k(h)(u).
\end{align*}

Now to prove \cref{itm:FrameFlowPreliminaryLogLipschitzProperty2}, suppose $H \in \mathcal{V}_\rho(\tilde{U})$ and $h \in B(\tilde{U}, \mathbb R)$ satisfy $\|(dH)_u\|_{\mathrm{op}} \leq Bh(u)$ for all $u \in \tilde{U}$, for some $B > 0$. Let $u \in \tilde{U}_{\alpha_k}$ for some $\alpha_k \in \mathcal{A}$. Let $z \in \T_u(\tilde{U})$ with $\|z\| = 1$. Recall \cref{eqn:k^thIterationOfCongruenceTransferOperatorOfType_rho}. Taking the differential and using the product rule, we have
\begin{align}
\begin{aligned}
\label{eqn:DifferentialOfCongruenceTransferOperatorOfType_rho}
&\left(d\tilde{\mathcal{M}}_{\xi, \rho}^k(H)\right)_u(z) \\
={}&\sum_{\alpha: \len(\alpha) = k} e^{f_\alpha^{(a)}(\sigma^{-\alpha}(u))} d\bigl(f_\alpha^{(a)} \circ \sigma^{-\alpha}\bigr)_u(z) \cdot \rho_b(\Phi^\alpha(\sigma^{-\alpha}(u))^{-1}) H(\sigma^{-\alpha}(u)) \\
{}&- \sum_{\alpha: \len(\alpha) = k} e^{f_\alpha^{(a)}(\sigma^{-\alpha}(u))} \cdot d(\rho_b \circ \Phi^\alpha \circ \sigma^{-\alpha})_u(z) H(\sigma^{-\alpha}(u)) \\
{}&+ \sum_{\alpha: \len(\alpha) = k} e^{f_\alpha^{(a)}(\sigma^{-\alpha}(u))} \rho_b(\Phi^\alpha(\sigma^{-\alpha}(u))^{-1}) d(H \circ \sigma^{-\alpha})_u(z) \\
={}&K_1 - K_2 + K_3.
\end{aligned}
\end{align}
Then $\left\|\left(d\tilde{\mathcal{M}}_{\xi, \rho}^k(H)\right)_u(z)\right\|_2 \leq \|K_1\|_2 + \|K_2\|_2 + \|K_3\|_2$ and we can bound each of these terms in a similar fashion as before. Using a previous bound and recalling that $\rho_b$ is a unitary representation, we estimate the first term $\|K_1\|_2$ as
\begin{align*}
\|K_1\|_2 \leq{}&\sum_{\alpha: \len(\alpha) = k} e^{f_\alpha^{(a)}(\sigma^{-\alpha}(u))} \big|d\bigl(f_\alpha^{(a)} \circ \sigma^{-\alpha}\bigr)_u(z)\big| \cdot \left\|H(\sigma^{-\alpha}(u))\right\|_2 \\
\leq{}&\frac{A_0}{2} \delta_{1, \varrho} \sum_{\substack{\alpha: \len(\alpha) = k\\ u' = \sigma^{-\alpha}(u)}} e^{f_\alpha^{(a)}(u')} \|H(u')\|_2 \leq \frac{A_0}{2} \delta_{1, \varrho} \tilde{\mathcal{L}}_a^k\|H\|(u).
\end{align*}
To estimate the second term $\|K_2\|_2$, we first obtain bounds for $\|d(\Phi^\alpha \circ \sigma^{-\alpha})_u(z)\|$. From definitions, we have
\begin{align*}
\Phi^\alpha(\sigma^{-\alpha}(u)) &= a_{\tau_{\alpha}(\sigma^{-\alpha}(u))} \prod_{j = 0}^{k - 1} \vartheta^{(\alpha_j, \alpha_{j + 1})}(\sigma^{-(\alpha_j, \alpha_{j + 1}, \dotsc, \alpha_k)}(u))
\end{align*}
for all admissible sequences $\alpha$ with $\len(\alpha) = k$. Denote by $a_{\boldsymbol{\cdot}}: \mathbb R \to A$ the map defined by $t \mapsto a_t$. Let $m^{\mathrm{L}}_g, m^{\mathrm{R}}_g: G \to G$ be the left and right multiplication maps respectively, by $g \in G$. For convenience, we also introduce the notations
\begin{align*}
L^{\alpha, j} &= a_{\tau_{\alpha}(\sigma^{-\alpha}(u))} \prod_{l = 0}^{j - 1} \vartheta^{(\alpha_l, \alpha_{l + 1})}(\sigma^{-(\alpha_l, \alpha_{l + 1}, \dotsc, \alpha_k)}(u)) \\
C^{\alpha, j} &= \vartheta^{(\alpha_j, \alpha_{j + 1})}(\sigma^{-(\alpha_j, \alpha_{j + 1}, \dotsc, \alpha_k)}(u)) \\
R^{\alpha, j} &= \prod_{l = j + 1}^{k - 1} \vartheta^{(\alpha_l, \alpha_{l + 1})}(\sigma^{-(\alpha_l, \alpha_{l + 1}, \dotsc, \alpha_k)}(u))
\end{align*}
for all $0 \leq j \leq k - 1$ and admissible sequences $\alpha$ with $\len(\alpha) = k$ with the convention $L^{\alpha, 0} = R^{\alpha, k - 1} = e$. Taking the differential and using the product rule, we calculate that
\begin{align*}
&d(\Phi^\alpha \circ \sigma^{-\alpha})_u(z) \\
={}& \left(\left(dm_{\vartheta^{\alpha}(\sigma^{-\alpha}(u))}^{\mathrm{R}}\right)_{a_{\tau_\alpha(\sigma^{-\alpha}(u))}} \circ (da_{\boldsymbol{\cdot}})_{\tau_\alpha(\sigma^{-\alpha}(u))} \circ d(\tau_\alpha \circ \sigma^{-\alpha})_u\right)(z) \\
{}&+ \sum_{j = 0}^{k - 1} \left(d\left(m^{\mathrm{R}}_{R^{\alpha, j}} \circ m^{\mathrm{L}}_{L^{\alpha, j}}\right)_{C^{\alpha, j}} \circ d(\vartheta^{(\alpha_j, \alpha_{j + 1})} \circ \sigma^{-(\alpha_j, \alpha_{j + 1}, \dotsc, \alpha_k)})_u\right)(z).
\end{align*}
Hence by left $G$-invariance and right $K$-invariance of the Riemannian metric on $G$, and the fact that $\|(da_{\boldsymbol{\cdot}})_t(1)\| = 1$ for all $t \in \mathbb R$ since $a_{\boldsymbol{\cdot}}$ is the geodesic flow, we have
\begin{align*}
\|d(\Phi^\alpha \circ \sigma^{-\alpha})_u(z)\|
\leq \left|d(\tau_\alpha \circ \sigma^{-\alpha})_u(z)\right| + \sum_{j = 0}^{k - 1} \big|\vartheta^{(\alpha_j, \alpha_{j + 1})}\big|_{C^1} \big|\sigma^{-(\alpha_j, \alpha_{j + 1}, \dotsc, \alpha_k)}\big|_{C^1}.
\end{align*}
By similar calculations as in \cref{eqn:f_alpha^(a)DerivativeBound}, we get $\|d(\Phi^\alpha \circ \sigma^{-\alpha})_u(z)\| \leq \frac{2T_0}{c_0(\kappa_2 - 1)} \leq \frac{A_0}{2}$. Thus, we have the bound
\begin{align*}
\|K_2\|_2 &\leq \sum_{\alpha: \len(\alpha) = k} e^{f_\alpha^{(a)}(\sigma^{-\alpha}(u))} \left\|d(\rho_b \circ \Phi^\alpha \circ \sigma^{-\alpha})_u(z)\right\|_{\mathrm{op}} \|H(\sigma^{-\alpha}(u))\|_2 \\
&\leq \sum_{\alpha: \len(\alpha) = k} e^{f_\alpha^{(a)}(\sigma^{-\alpha}(u))} \|\rho_b\| \cdot \|d(\Phi^\alpha \circ \sigma^{-\alpha})_u(z)\| \cdot \|H(\sigma^{-\alpha}(u))\|_2 \\
&\leq \frac{A_0}{2}\|\rho_b\|\sum_{\substack{\alpha: \len(\alpha) = k\\ u' = \sigma^{-\alpha}(u)}} e^{f_\alpha^{(a)}(u')} \|H(u')\|_2 \leq \frac{A_0}{2}\|\rho_b\|\tilde{\mathcal{L}}_a^k\|H\|(u).
\end{align*}
Finally, recalling that $\rho_b$ is a unitary representation, we estimate the third and last term $\|K_3\|_2$ as
\begin{align*}
\|K_3\|_2 &\leq \sum_{\alpha: \len(\alpha) = k} e^{f_\alpha^{(a)}(\sigma^{-\alpha}(u))} \|(dH)_{\sigma^{-\alpha}(u)}\|_{\mathrm{op}} \|(d\sigma^{-\alpha})_u\|_{\mathrm{op}} \\
&\leq \frac{B}{c_0\kappa_2^k}\sum_{\substack{\alpha: \len(\alpha) = k\\ u' = \sigma^{-\alpha}(u)}} e^{f_\alpha^{(a)}(u')} h(u') \leq \frac{A_0B}{\kappa_2^k}\tilde{\mathcal{L}}_a^k(h)(u).
\end{align*}
Combining all three bounds, and recalling that they hold for all $z \in \T_u(\tilde{U})$ with $\|z\| = 1$, we have
\begin{align*}
\left\|\left(d\tilde{\mathcal{M}}_{\xi, \rho}^k(H)\right)_u\right\|_{\mathrm{op}} &\leq \frac{A_0}{2}\delta_{1, \varrho} \tilde{\mathcal{L}}_a^k\|H\|(u) + \frac{A_0}{2}\|\rho_b\| \tilde{\mathcal{L}}_a^k\|H\|(u) + \frac{A_0B}{\kappa_2^k}\tilde{\mathcal{L}}_a^k(h)(u) \\
&\leq A_0\left(\frac{B}{\kappa_2^k}\tilde{\mathcal{L}}_a^k(h)(u) + \|\rho_b\| \tilde{\mathcal{L}}_a^k\|H\|(u)\right).
\end{align*}
\end{proof}

Fix $A_0 > 0$ provided by \cref{lem:FrameFlowPreliminaryLogLipschitz}. Fix $m_1 \in \mathbb N$ sufficiently large and for all $k \in \mathcal{A}$, fix a cylinder $\mathtt{C}_k \subset U_1$ with $\len(\mathtt{C}_k) = m_1$ such that $\overline{\mathtt{C}_k} \subset \mathcal{U}$ and $\sigma^{m_1}(\mathtt{C}_k) = \interior(U_k)$. Fix $C_{\mathrm{Vit}} = \min_{k \in \mathcal{A}}d(\overline{\mathtt{C}_k}, \partial(\mathcal{U}))$. Let the corresponding sections be $\mathtt{v}_k: \tilde{U}_k \to \tilde{U}_1$ for all $k \in \mathcal{A}$. Now, fix positive constants
\begin{align}
\label{eqn:Constant_b_0}
b_0 &= 1; \\
\label{eqn:ConstantE}
E &> \frac{2A_0}{\delta_{1, \varrho}}; \\
\delta_1 &< \frac{\varepsilon_1\varepsilon_2\varepsilon_3\delta_\Psi }{14C_\Psi}; \\
\label{eqn:Constantepsilon1}
\epsilon_1 &< \min\left(C_{\mathrm{Vit}}, \frac{2\delta_0 \delta_{1, \varrho}}{C_\Psi C_{\BP, \Psi}}, \frac{4\delta_1}{(C_\Psi C_{\BP, \Psi})^2}, \frac{4\delta_1 \delta_{1, \varrho}}{C_{\exp, \BP}}, \frac{1}{\delta_1}, \frac{c_0C_{\mathrm{Ano}}C_\phi\hat{\delta}e^{\hat{\delta}}}{5\kappa_1^{m_1}C_\Psi^2\delta_{1, \varrho}}\right); \\
\label{eqn:Constantepsilon2}
\epsilon_2 &< \min\left(\frac{\varepsilon_3\epsilon_1}{4NC_\Psi^2}, \frac{\delta_1\epsilon_1}{4N(A_0 + \delta_1)}\right); \\
\label{eqn:Constantepsilon3}
\epsilon_3 &= \frac{c_0\kappa_2^{m_1}\epsilon_2}{2}; \\
\label{eqn:Constantepsilon4}
\epsilon_4 &= 10c_0^{-1}\kappa_1^{m_1}C_\Psi^2\epsilon_1; \\
\label{eqn:Constantm2}
m_2 &> m_0 \text{ such that } \kappa_2^{m_2} > \max\left(8A_0, \frac{4EN\epsilon_2}{c_0\log(2)}, \frac{32EN\epsilon_2}{c_0}, \frac{4E}{c_0\delta_1}\right); \\
\label{eqn:Constantmu}
\mu &< \min\left(\frac{E\epsilon_2}{2N}, \frac{1}{4N}, \frac{1}{16 \cdot 16e^{2m_2 T_0} \cdot N}\arccos\left(1 - \frac{(\delta_1 \epsilon_1)^2}{2}\right)^2\right).
\end{align}
We defer the definition of $C_\phi$, which only depends on the Markov section, until \cref{subsec:ProofOfFrameFlowDolgopyatProperty2InFrameFlowDolgopyat} where it is needed. Fix $m = m_1 + m_2$. Fix admissible sequences $\alpha_j = (\alpha_{j, 0}, \alpha_{j, 1}, \dotsc, \alpha_{j, m_2 - 1}, 1)$ and corresponding sections $v_j = \sigma^{-\alpha_j}: \tilde{U}_1 \to \tilde{U}_{\alpha_{j, 0}}$ provided by \cref{pro:FrameFlowLNIC} for all $0 \leq j \leq j_{\mathrm{m}}$. Fix corresponding maps $\BP_j: \tilde{U}_1 \times \tilde{U}_1 \to AM$ provided by \cref{pro:FrameFlowLNIC} for all $1 \leq j \leq j_{\mathrm{m}}$.

\section{Construction of Dolgopyat operators}
\label{sec:ConstructionOfDolgopyatOperators}
Now we have the tools to construct the Dolgopyat operators and prove \cref{thm:FrameFlowDolgopyat}.

Let $(b, \rho) \in \widehat{M}_0(b_0)$ and $k \in \mathcal{A}$. We can use the map $\Psi$ and the Vitali covering lemma on $\mathbb R^{n - 1}$ to choose a finite subset $\big\{x_{k, r, 1}^{(b, \rho)} \in \mathtt{C}_k: r \in \big\{1, 2, \dotsc, r_k^{(b, \rho)}\big\}\big\} \subset \mathtt{C}_k$ for some $r_k^{(b, \rho)} \in \mathbb N$ and corresponding open balls $C_{k, r}^{(b, \rho)} = W_{\epsilon_1/\|\rho_b\|}^{\mathrm{su}}\big(x_{k, r, 1}^{(b, \rho)}\big) \subset \mathcal{U}$ and $\hat{C}_{k, r}^{(b, \rho)} = W_{5C_\Psi^2\epsilon_1/\|\rho_b\|}^{\mathrm{su}}\big(x_{k, r, 1}^{(b, \rho)}\big)$ for all $1 \leq r \leq r_k^{(b, \rho)}$ such that $C_{k, r}^{(b, \rho)} \cap C_{k, r'}^{(b, \rho)} = \varnothing$ for all $1 \leq r, r' \leq r_k^{(b, \rho)}$ with $r \neq r'$ and $\mathtt{C}_k \subset \bigcup_{r = 1}^{r_k^{(b, \rho)}} \hat{C}_{k, r}^{(b, \rho)}$. Recall the notation $\check{x}_{k, r, 1}^{(b, \rho)} = \Psi^{-1}\big(x_{k, r, 1}^{(b, \rho)}\big)$ for all $1 \leq r \leq r_k^{(b, \rho)}$.

\begin{lemma}
\label{lem:PartnerPointInZariskiDenseLimitSetForBPBound}
For all $(b, \rho) \in \widehat{M}_0(b_0)$, $\omega \in V_\rho^{\oplus \dim(\rho)}$ with $\|\omega\|_2 = 1$, $k \in \mathcal{A}$, and $1 \leq r \leq r_k^{(b, \rho)}$, there exist $1 \leq j \leq j_{\mathrm{m}}$ and $\check{x}_2 \in \Lambda(\Gamma) \cap \left(B_{s_1}^{\mathrm{E}}\big(\check{x}_{k, r, 1}^{(b, \rho)}\big) \setminus B_{s_2}^{\mathrm{E}}\big(\check{x}_{k, r, 1}^{(b, \rho)}\big)\right)$ such that
\begin{align*}
\left\|d\rho_b\left(d(\BP_{j, x_{k, r, 1}^{(b, \rho)}} \circ \Psi)_{\check{x}_{k, r, 1}^{(b, \rho)}}(z)\right)(\omega)\right\|_2 \geq 7\delta_1\epsilon_1
\end{align*}
where $s_1 = \frac{\epsilon_1}{2C_\Psi \|\rho_b\|}$, $s_2 = \frac{\varepsilon_3\epsilon_1}{2C_\Psi \|\rho_b\|}$, and $z = \big(\check{x}_{k, r, 1}^{(b, \rho)}, \check{x}_2 - \check{x}_{k, r, 1}^{(b, \rho)}\big) \in \T_{\check{x}_{k, r, 1}^{(b, \rho)}}(\mathbb R^{n - 1})$.
\end{lemma}

\begin{proof}
Let $(b, \rho) \in \widehat{M}_0(b_0)$, $\omega \in V_\rho^{\oplus \dim(\rho)}$ with $\|\omega\|_2 = 1$, $k \in \mathcal{A}$, and $1 \leq r \leq r_k^{(b, \rho)}$. Fix $s_1 = \frac{\epsilon_1}{2C_\Psi \|\rho_b\|}$, $s_2 = \frac{\varepsilon_3\epsilon_1}{2C_\Psi \|\rho_b\|}$. Denote $x_{k, r, 1}^{(b, \rho)}$ by $x_1$ and $\check{x}_{k, r, 1}^{(b, \rho)}$ by $\check{x}_1$. Define the linear maps $L_1 = (d\Psi)_{\check{x}_1}$, $L_{2, j}: \T_{x_1}(\tilde{U}_1) \to \mathfrak{a} \oplus \mathfrak{m}$ by $L_{2, j}(w) = (d\BP_{j, x_1})_{x_1}(w)$ for all $w \in \T_{x_1}(\tilde{U}_1)$, and $L_3: \mathfrak{a} \oplus \mathfrak{m} \to V_\rho^{\oplus \dim(\rho)}$ by $L_3(w) = d\rho_b(w)(\omega)$ for all $w \in \mathfrak{a} \oplus \mathfrak{m}$. It suffices to find $1 \leq j \leq j_{\mathrm{m}}$ and $\check{x}_2 \in \Lambda(\Gamma) \cap \left(B_{s_1}^{\mathrm{E}}(\check{x}_1) \setminus B_{s_2}^{\mathrm{E}}(\check{x}_1)\right)$ such that
\begin{align*}
|\langle (L_3 \circ L_{2, j} \circ L_1)(z), \omega_0 \rangle| = |\langle z, (L_1^* \circ L_{2, j}^* \circ L_3^*)(\omega_0) \rangle| \geq 7\delta_1\epsilon_1
\end{align*}
for some $\omega_0 \in V_\rho^{\oplus \dim(\rho)}$ with $\|\omega_0\|_2 = 1$, where $z = (\check{x}_1, \check{x}_2 - \check{x}_1) \in \T_{\check{x}_1}(\mathbb R^{n - 1})$. By \cref{lem:maActionLowerBound}, $\|L_3\|_{\mathrm{op}} \geq \varepsilon_1\|\rho_b\|$ which implies $\|L_3^*\|_{\mathrm{op}} \geq \varepsilon_1\|\rho_b\|$. Hence there exists $\omega_0 \in V_\rho^{\oplus \dim(\rho)}$ with $\|\omega_0\|_2 = 1$ such that $\|L_3^*(\omega_0)\| \geq \varepsilon_1\|\rho_b\|$. Now, \cref{pro:FrameFlowLNIC} implies that there exists $1 \leq j \leq j_{\mathrm{m}}$ such that $\|L_{2, j}^*(L_3^*(\omega_0))\| \geq \varepsilon_2 \varepsilon_1\|\rho_b\|$. Using a previous bound for $\Psi$, we get $\|L_1^*(L_{2, j}^*(L_3^*(\omega_0)))\| \geq \delta_\Psi \varepsilon_1\varepsilon_2 \|\rho_b\|$. Finally, by \cref{pro:NonConcentrationProperty}, there exists $\check{x}_2 \in \Lambda(\Gamma) \cap \left(B_{s_1}^{\mathrm{E}}(\check{x}_1) \setminus B_{s_2}^{\mathrm{E}}(\check{x}_1)\right)$ such that
\begin{align*}
|\langle z, (L_1^* \circ L_{2, j}^* \circ L_3^*)(\omega_0) \rangle| &\geq \frac{\epsilon_1}{2C_\Psi \|\rho_b\|} \cdot \varepsilon_3 \cdot \|L_1^*(L_{2, j}^*(L_3^*(\omega_0)))\| \\
&\geq \frac{\epsilon_1}{2C_\Psi \|\rho_b\|} \cdot \varepsilon_1\varepsilon_2\varepsilon_3\delta_\Psi \|\rho_b\| \geq 7\delta_1\epsilon_1
\end{align*}
where $z = (\check{x}_1, \check{x}_2 - \check{x}_1) \in \T_{\check{x}_1}(\mathbb R^{n - 1})$.
\end{proof}

Let $(b, \rho) \in \widehat{M}_0(b_0)$, $H \in \mathcal{V}_\rho(\tilde{U})$, $k \in \mathcal{A}$, and $1 \leq r \leq r_k^{(b, \rho)}$. Corresponding to
\begin{align*}
\omega = \frac{\rho_b\big(\Phi^{\alpha_0}\big(v_0\big(x_{k, r, 1}^{(b, \rho)}\big)\big)^{-1}\big)H\big(v_0\big(x_{k, r, 1}^{(b, \rho)}\big)\big)}{\big\|H\big(v_0\big(x_{k, r, 1}^{(b, \rho)}\big)\big)\big\|_2} \in V_\rho^{\oplus \dim(\rho)}
\end{align*}
denote $j_{k, r}^{(b, \rho), H}$ and $x_{k, r, 2}^{(b, \rho), H}$ to be the $j$ and $\Psi(\check{x}_2) \in U_1 \cap \left(W_{s_1}^{\mathrm{su}}\big(x_{k, r, 1}^{(b, \rho)}\big) \setminus W_{s_2}^{\mathrm{su}}\big(x_{k, r, 1}^{(b, \rho)}\big)\right)$ provided by \cref{lem:PartnerPointInZariskiDenseLimitSetForBPBound}, where $s_1 = \frac{\epsilon_1}{2\|\rho_b\|}$ and $s_2 = \frac{\varepsilon_3\epsilon_1}{2C_\Psi^2 \|\rho_b\|}$. Define
\begin{align*}
D_{k, r, 1}^{(b, \rho)} &= W_{\epsilon_2/\|\rho_b\|}^{\mathrm{su}}\big(x_{k, r, 1}^{(b, \rho)}\big) \subset C_{k, r}^{(b, \rho)}; & D_{k, r, 2}^{(b, \rho), H} &= W_{\epsilon_2/\|\rho_b\|}^{\mathrm{su}}\big(x_{k, r, 2}^{(b, \rho), H}\big) \subset C_{k, r}^{(b, \rho)}; \\
\slashed{D}_{k, r, 1}^{(b, \rho)} &= W_{\frac{\epsilon_2}{2\|\rho_b\|}}^{\mathrm{su}}\big(x_{k, r, 1}^{(b, \rho)}\big) \subset C_{k, r}^{(b, \rho)}; & \slashed{D}_{k, r, 2}^{(b, \rho), H} &= W_{\frac{\epsilon_2}{2\|\rho_b\|}}^{\mathrm{su}}\big(x_{k, r, 2}^{(b, \rho), H}\big) \subset C_{k, r}^{(b, \rho)}; \\
\hat{D}_{k, r, 1}^{(b, \rho)} &= W_{2N\epsilon_2/\|\rho_b\|}^{\mathrm{su}}\big(x_{k, r, 1}^{(b, \rho)}\big) \subset C_{k, r}^{(b, \rho)}; & \hat{D}_{k, r, 2}^{(b, \rho), H} &= W_{2N\epsilon_2/\|\rho_b\|}^{\mathrm{su}}\big(x_{k, r, 2}^{(b, \rho), H}\big) \subset C_{k, r}^{(b, \rho)}.
\end{align*}
Denote $\psi_{k, r, 1}^{(b, \rho)}, \psi_{k, r, 2}^{(b, \rho), H} \in C^\infty(\tilde{U}, \mathbb R)$ to be bump functions with $\supp\big(\psi_{k, r, 1}^{(b, \rho)}\big) = \overline{D_{k, r, 1}^{(b, \rho)}}$ and $\supp\big(\psi_{k, r, 2}^{(b, \rho), H}\big) = \overline{D_{k, r, 2}^{(b, \rho), H}}$ such that they attain the maximum values $\left.\psi_{k, r, 1}^{(b, \rho)}\right|_{\overline{\slashed{D}_{k, r, 1}^{(b, \rho)}}} = \left.\psi_{k, r, 2}^{(b, \rho), H}\right|_{\overline{\slashed{D}_{k, r, 2}^{(b, \rho), H}}} = 1$, and the minimum values $\left.\psi_{k, r, 1}^{(b, \rho)}\right|_{\tilde{U} \setminus D_{k, r, 1}^{(b, \rho)}} = \left.\psi_{k, r, 2}^{(b, \rho), H}\right|_{\tilde{U} \setminus D_{k, r, 2}^{(b, \rho), H}} = 0$, and we can further assume that $\left|\psi_{k, r, 1}^{(b, \rho)}\right|_{C^1}, \left|\psi_{k, r, 2}^{(b, \rho), H}\right|_{C^1} \leq \frac{4\|\rho_b\|}{\epsilon_2}$. It can be checked that $D_{k, r_1, p_1}^{(b, \rho)} \cap D_{k, r_2, p_2}^{(b, \rho), H} = \varnothing$ for all $(r_1, p_1), (r_2, p_2) \in \{1, 2, \dotsc, r_k^{(b, \rho)}\} \times \{1, 2\}$ with $(r_1, p_1) \neq (r_2, p_2)$ and $k \in \mathcal{A}$. Define $\Xi_1(b, \rho) = \big\{(k, r) \in \mathbb Z^2: k \in \mathcal{A}, r \in \big\{1, 2, \dotsc, r_k^{(b, \rho)}\big\}\big\}$ and $\Xi_2 = \{1, 2\} \times \{1, 2\}$. Define $\Xi(b, \rho) = \Xi_1(b, \rho) \times \Xi_2$. For all $(k, r, p, l) \in \Xi(b, \rho)$, denoting $j_{k, r}^{(b, \rho), H}$ by $j$ for convenience, we define the function $\tilde{\psi}_{(k, r, p, l)}^{(b, \rho), H} \in C^\infty(\tilde{U}, \mathbb R)$ by
\begin{align*}
\tilde{\psi}_{(k, r, p, l)}^{(b, \rho), H} =
\begin{cases}
\chi_{\tilde{\mathtt{C}}[\alpha_0]} \cdot \left(\psi_{k, r, 1}^{(b, \rho)} \circ \sigma^{\alpha_0}\right), & p = 1, l = 1 \\
\chi_{\tilde{\mathtt{C}}[\alpha_j]} \cdot \left(\psi_{k, r, 1}^{(b, \rho)} \circ \sigma^{\alpha_j}\right), & p = 1, l = 2 \\
\chi_{\tilde{\mathtt{C}}[\alpha_0]} \cdot \left(\psi_{k, r, 2}^{(b, \rho), H} \circ \sigma^{\alpha_0}\right), & p = 2, l = 1 \\
\chi_{\tilde{\mathtt{C}}[\alpha_j]} \cdot \left(\psi_{k, r, 2}^{(b, \rho), H} \circ \sigma^{\alpha_j}\right), & p = 2, l = 2
\end{cases}
\end{align*}
where using $\sigma^{\alpha_0}$ and $\sigma^{\alpha_j}$ are indeed justified because of the indicator functions $\chi_{\tilde{\mathtt{C}}[\alpha_0]} = \chi_{v_0(\tilde{U}_1)}$ and $\chi_{\tilde{\mathtt{C}}[\alpha_j]} = \chi_{v_j(\tilde{U}_1)}$. For all subsets $J \subset \Xi(b, \rho)$, we define
\begin{align*}
\beta_J^H = \chi_{\tilde{U}} - \mu\sum_{(k, r, p, l) \in J} \tilde{\psi}_{(k, r, p, l)}^{(b, \rho), H} \in C^\infty(\tilde{U}, \mathbb R).
\end{align*}

\newsavebox{\FirstFigure}
\newsavebox{\SecondFigure}
\newlength{\FirstFigureWidth}
\newlength{\SecondFigureWidth}
\savebox{\FirstFigure}{
\begin{tikzpicture}
\definecolor{highlight}{RGB}{98,145,166}
%Gray blob on the right hand side.
\draw[gray, fill = gray, fill opacity=0.7] (-1.5, 1.5) to[out=10,in=180] (0, 1.05) to[out=0,in=150] (1.8, 2.025) to[out=-30,in=90] (1.5, 0) to[out=-90,in=20] (1.95, -1.8) to[out=200,in=20] (-0.75, -0.75) to[out=200,in=240] (-1.8, -0.15) to[out=60,in=190] (-1.5, 1.5) -- cycle;

%Gray blob on the left hand side.
\draw[gray, fill = gray, fill opacity=0.7] (-1.95, 0.3) to[out=-80,in=-10] (-2.25, -0.3) to[out=170,in=10] (-2.85, -0.3) to[out=190,in=200] (-2.25, 0.6) to[out=20,in=100] (-1.95, 0.3) -- cycle;

%Centers of the circles.
\draw[fill = black] (-1.125, 1.275) circle  [radius = 0.008];
\draw[fill = black] (0.345, -0.15) circle  [radius = 0.008];
\draw[fill = black] (1.38, 1.35) circle  [radius = 0.008];
\draw[fill = black] (1.56, -1.335) circle  [radius = 0.008];
\draw[fill = black] (-2.025, 0) circle  [radius = 0.008];
\draw[fill = black] (-1.68, 0) circle  [radius = 0.008];

%Pair of gray circles at the top left.
\draw[gray, thin] (-1.125, 1.275) circle  [radius = 0.18, fill = white];
\draw[gray, thin, fill = gray, fill opacity=0.3] (-1.125, 1.275) circle  [radius = 0.9, fill = white];

%Pair of gray circles in the middle.
\draw[gray, thin] (0.345, -0.15) circle  [radius = 0.21, fill = white];
\draw[gray, thin, fill = gray, fill opacity=0.3] (0.345, -0.15) circle  [radius = 1.5, fill = white];

%Pair of gray circles at the top right.
\draw[gray, thin] (1.38, 1.35) circle  [radius = 0.18, fill = white];
\draw[gray, thin, fill = gray, fill opacity=0.3] (1.38, 1.35) circle  [radius = 0.9, fill = white];

%Pair of gray circles at the bottom right.
\draw[gray, thin] (1.56, -1.335) circle  [radius = 0.15, fill = white];
\draw[gray, thin, fill = gray, fill opacity=0.3] (1.56, -1.335) circle  [radius = 0.75, fill = white];

%Pair of gray circles on the left left.
\draw[highlight, thin] (-2.025, 0) circle  [radius = 0.225, fill = white];
\draw[gray, thin, fill = gray, fill opacity=0.3] (-2.025, 0) circle  [radius = 1.125, fill = white];

%Pair of gray circles on the left right.
\draw[highlight, thin] (-1.68, 0) circle  [radius = 0.225, fill = white];
\draw[gray, thin, fill = gray, fill opacity=0.3] (-1.68, 0) circle  [radius = 1.125, fill = white];

%redrawing just to make sure the blue circle is on top and has good color.
\draw[highlight, thin] (-2.025, 0) circle  [radius = 0.225, fill = white];
\draw[highlight, thin] (-1.68, 0) circle  [radius = 0.225, fill = white];
\end{tikzpicture}
}
\savebox{\SecondFigure}{
\begin{tikzpicture}
\definecolor{highlight}{RGB}{98,145,166}
\draw[gray] (0, 0) circle  [radius = 1, fill = white];
\draw[gray, fill = gray, fill opacity=0.3] (0, 0) circle  [radius = 2, fill = white];

\draw[highlight, fill = highlight, fill opacity=0.7] (0.125, 0.05) circle  [radius = 0.1, fill = white];
\draw[highlight, fill = highlight, fill opacity=0.7] (0.25, 0.06) circle  [radius = 0.1, fill = white];
\draw[highlight, dotted] (0.4, 0.08) to[out=20,in=200] (1.7, 0);
\draw[highlight, fill = highlight, fill opacity=0.7] (1.8, 0) circle  [radius = 0.1, fill = white];

\draw[gray, fill = gray, fill opacity=0.7] (0, 0) circle  [radius = 0.1, fill = white];
\draw[gray, fill = gray, fill opacity=0.7] (0, 0.75) circle  [radius = 0.1, fill = white];

\node[below, font=\tiny] at (-1.5, 0.55) {$\hat{D}_{1, 1, 1}^{(b, \rho)}$};
\node[above, font=\tiny] at (-1.5, -0.55) {$C_{1, 1}^{(b, \rho)}$};
\draw[->, very thin] (-1.5, 0.55) to[out=90,in=-35] (-1.55, 1.2);
\draw[->, very thin] (-1.5, -0.55) to[out=-90,in=215] (-0.77, -0.7);

\node[above, font=\tiny] at (0, 1.15) {$D_{1, 1, 2}^{(b, \rho), H}$};
\node[below, font=\tiny] at (0, -0.4) {$D_{1, 1, 1}^{(b, \rho)}$};
\draw[->, very thin] (0, 1.15) -- (0, 0.89);
\draw[->, very thin] (0, -0.4) -- (0, -0.13);

\node[below, highlight, font=\tiny] at (1.35, -0.35) {$D_{k, 1, 1}^{(b, \rho)}$};
\draw[->, highlight, very thin] (1.35, -0.35) to[out=90,in=-90] (0.9, 0.05);
\end{tikzpicture}
}
\settowidth{\FirstFigureWidth}{\usebox{\FirstFigure}}
\settowidth{\SecondFigureWidth}{\usebox{\SecondFigure}}
\begin{figure}
\centering
\hfill%
\begin{subfigure}[t]{\FirstFigureWidth}
\centering
\usebox{\FirstFigure}%
\caption{This example illustrates the Vitali covering of two cylinders which are depicted by the two dark gray regions. Although the smaller balls for a fixed cylinder are mutually disjoint, they may intersect the smaller balls of other cylinders as shown in blue.}
\label{subfig:VitaliCovering}
\end{subfigure}%
\hfill%
\begin{subfigure}[t]{\SecondFigureWidth}
\centering
\usebox{\SecondFigure}%
\caption{This example illustrates that the chain of intersecting $D_{k, 1, 1}^{(b, \rho)}$ is contained in $\hat{D}_{1, 1, 1}^{(b, \rho)}$.}
\label{subfig:test}
\end{subfigure}%
\hfill%
\null
\caption{}
\label{fig:VitaliCovering}
\end{figure}

\begin{remark}
We will often include the superscript $H$  even when there is no dependence on it for a more uniform notation to simplify exposition.
\end{remark}

\begin{lemma}
\label{lem:NumberOfIntersectingBallsLessThanOrEqualToN}
Let $(b, \rho) \in \widehat{M}_0(b_0)$, $H \in \mathcal{V}_\rho(\tilde{U})$, and $J \subset \Xi(b, \rho)$. Then, any connected component of
\begin{align*}
\bigcup \left\{D_{k, r, p}^{(b, \rho), H}: (k, r, p, l) \in J \text{ for some } l \in \{1, 2\}\right\}
\end{align*}
is a union of at most $N$ number of the terms and hence contained in $\hat{D}_{k, r, p}^{(b, \rho), H}$ for any $(k, r, p, l) \in J$ corresponding to one of those terms.
\end{lemma}

\begin{proof}
Let $(b, \rho) \in \widehat{M}_0(b_0)$, $H \in \mathcal{V}_\rho(\tilde{U})$, and $J \subset \Xi(b, \rho)$. We drop superscripts $(b, \rho)$ and $H$ to simply notation. Define
\begin{align*}
D_\cup = \{D_{k, r, p}: (k, r, p, l) \in J \text{ for some } l \in \{1, 2\}\}
\end{align*}
and $D_\cup^{\mathrm{conn}} = D_\cup/{\sim}$ where $\sim$ is the equivalence relation defined by $D_{k_1, r_1, p_1} \sim D_{k_2, r_2, p_2}$ if $D_{k_1, r_1, p_1} \cap D_{k_2, r_2, p_2} \neq \varnothing$, for all distinct $D_{k_1, r_1, p_1}, D_{k_2, r_2, p_2} \in D_\cup$ and then extended by transitivity. Let $[D_{k, r, p}] \in D_\cup^{\mathrm{conn}}$ be any equivalence class. It suffices to show that show that it has cardinality $\#[D_{k, r, p}] \leq N$. The cardinality must be finite and by way of contradiction, suppose $\#[D_{k, r, p}] > N$. Consider the connected graph $G_{[D_{k, r, p}]} = \bigl([D_{k, r, p}], E_{[D_{k, r, p}]}\bigr)$ where
\begin{align*}
E_{\left[D_{k, r, p}\right]} ={}&\left\{\left(D_{k_1, r_1, p_1}, D_{k_2, r_2, p_2}\right): D_{k_1, r_1, p_1}, D_{k_2, r_2, p_2} \in \left[D_{k, r, p}\right] \right.\\
&\left.\text{are distinct and } D_{k_1, r_1, p_1} \cap D_{k_2, r_2, p_2} \neq \varnothing \right\}.
\end{align*}
Let $T_{[D_{k, r, p}]}$ be any spanning tree of the graph $G_{[D_{k, r, p}]}$. If the number of vertices of $T_{[D_{k, r, p}]}$ is greater than $N + 1$, then we can repeatedly delete leaves (vertices with only one emanating edge) and their corresponding edges until we obtain a subtree $T'_{[D_{k, r, p}]}$ with $N + 1$ vertices. This is possible since all trees have at least one leaf and deleting one results in a subtree. Since $\#\mathcal{A} = N$, by the pigeonhole principle there must be some $D_{k_0, r_1, p_1} \subset C_{k_0, r_1}$ and $D_{k_0, r_2, p_2} \subset C_{k_0, r_2}$ which are vertices of $T'_{[D_{k, r, p}]}$. Hence, there is a path between them of length at most $N$. But this represents a sequence of consecutive pairs of balls with nonempty intersections. This implies
\begin{align*}
d(x_{k_0, r_1, p_1}, x_{k_0, r_2, p_2}) < N \cdot \frac{2\epsilon_2}{\|\rho_b\|} = \frac{2N\epsilon_2}{\|\rho_b\|}.
\end{align*}
This is a contradiction using \cref{eqn:Constantepsilon2} by construction of $x_{k_0, r_1, p_1}$ and $x_{k_0, r_2, p_2}$.
\end{proof}

\begin{corollary}
\label{cor:SupAndC1SeminormBoundsForBeta_J}
Let $(b, \rho) \in \widehat{M}_0(b_0)$, $H \in \mathcal{V}_\rho(\tilde{U})$, and $J \subset \Xi(b, \rho)$. Then, we have $1 - N\mu \leq \beta_J^H \leq 1$ and $\left|\beta_J^H\right|_{C^1} \leq \frac{4N\mu\|\rho_b\|}{\epsilon_2}$.
\end{corollary}

\begin{proof}
Let $(b, \rho) \in \widehat{M}_0(b_0)$, $H \in \mathcal{V}_\rho(\tilde{U})$, and $J \subset \Xi(b, \rho)$. We drop superscripts $(b, \rho)$ and $H$ to simply notation. Let $u \in \tilde{U}$. If $u \notin v_j(D_{k, r, p})$ for all $0 \leq j \leq j_{\mathrm{m}}$ and $(k, r, p, l) \in J$, then from definitions we have $\beta_J(u) = 1$ and $\|(d\beta_J)_u\|_{\mathrm{op}} = 0$. Otherwise, suppose $u \in v_j(D_{k, r, p})$ for some $0 \leq j \leq j_{\mathrm{m}}$ and $(k, r, p, l) \in J$. By \cref{lem:NumberOfIntersectingBallsLessThanOrEqualToN}, using the same notations, we have $\#[D_{k, r, p}] \leq N$. By the last criterion in \cref{pro:FrameFlowLNIC}, we also have $u \notin v_{j'}(D_{k', r', p'})$ if $j' \neq j$ or $D_{k', r', p'} \notin [D_{k, r, p}]$. The corollary now follows from definitions.
\end{proof}

\begin{definition}[Dolgopyat operator]
For all $\xi \in \mathbb C$ with $|a| < a_0'$, if $(b, \rho) \in \widehat{M}_0(b_0)$, then for all $J \subset \Xi(b, \rho)$ and $H \in \mathcal{V}_\rho(\tilde{U})$, we define the \emph{Dolgopyat operator} $\mathcal{N}_{a, J}^H: C^1(\tilde{U}, \mathbb R) \to C^1(\tilde{U}, \mathbb R)$ by
\begin{align*}
\mathcal{N}_{a, J}^H(h) = \tilde{\mathcal{L}}_{a}^m\big(\beta_J^Hh\big) \qquad \text{for all $h \in C^1(\tilde{U}, \mathbb R)$}.
\end{align*}
\end{definition}

\begin{definition}[Dense]
For all $(b, \rho) \in \widehat{M}_0(b_0)$, a subset $J \subset \Xi(b, \rho)$ is said to be \emph{dense} if for all $(k, r) \in \Xi_1(b, \rho)$, there exists $(p, l) \in \Xi_2$ such that $(k, r, p, l) \in J$.
\end{definition}

For all $(b, \rho) \in \widehat{M}_0(b_0)$, define $\mathcal J(b, \rho) = \{J \subset \Xi(b, \rho): J \text{ is dense}\}$.

\section{Proof of \texorpdfstring{\cref{thm:FrameFlowDolgopyat}}{\autoref{thm:FrameFlowDolgopyat}}}
\label{sec:ProofOfFrameFlowDolgopyat}
We dedicate this section for the proof of \cref{thm:FrameFlowDolgopyat}. We do this by proving all the properties in the theorem in the following subsections.

For this section, recall the positive constant $a_0'$ from the end of \cref{subsec:TransferOperators} and that we already fixed $b_0 = 1$.

\subsection{Proof of \texorpdfstring{\cref{itm:FrameFlowDolgopyatProperty1,itm:FrameFlowLogLipschitzDolgopyat}}{Properties \ref{itm:FrameFlowDolgopyatProperty1} and \ref{itm:FrameFlowLogLipschitzDolgopyat}} in \texorpdfstring{\cref{thm:FrameFlowDolgopyat}}{\autoref{thm:FrameFlowDolgopyat}}}
\begin{lemma}
\label{lem:FrameFlowDolgopyatProperty1}
For all $\xi \in \mathbb C$ with $|a| < a_0'$, if $(b, \rho) \in \widehat{M}_0(b_0)$, then for all $J \in \mathcal J(b, \rho)$ and $H \in \mathcal{V}_\rho(\tilde{U})$, we have $\mathcal{N}_{a, J}^H(K_{E\|\rho_b\|}(\tilde{U})) \subset K_{E\|\rho_b\|}(\tilde{U})$.
\end{lemma}

\begin{proof}
Let $\xi \in \mathbb C$ with $|a| < a_0'$ and suppose $(b, \rho) \in \widehat{M}_0(b_0)$. Let $J \in \mathcal J(b, \rho)$ and $H \in \mathcal{V}_\rho(\tilde{U})$. Let $h \in K_{E\|\rho_b\|}(\tilde{U})$ and $u \in \tilde{U}$. \Cref{cor:SupAndC1SeminormBoundsForBeta_J} and \cref{eqn:Constantmu} give
\begin{align*}
\left\|d\left(\beta_J^H h\right)_u\right\|_{\mathrm{op}} &= \left\|\left(d\beta_J^H\right)_u\right\|_{\mathrm{op}} \cdot h(u) + \beta_J^H(u) \cdot \|(dh)_u\|_{\mathrm{op}} \\
&\leq \frac{4N\mu \|\rho_b\|}{\epsilon_2} h(u) + E\|\rho_b\|h(u) \\
&\leq (2E + E) \|\rho_b\|h(u) \cdot \frac{\beta_J^H(u)}{1 - N\mu} \\
&\leq 4E\|\rho_b\|\big(\beta_J^H h\big)(u).
\end{align*}
So $\beta_J^H h \in K_{4E\|\rho_b\|}(\tilde{U})$. Now applying \cref{lem:FrameFlowPreliminaryLogLipschitz,eqn:ConstantE,eqn:Constantm2}, we have
\begin{align*}
\left\|\Big(d\mathcal{N}_{a, J}^H(h)\Big)_u\right\|_{\mathrm{op}} &= \left\|\left(d\tilde{\mathcal{L}}_a^m\big(\beta_J^H h\big)\right)_u\right\|_{\mathrm{op}} \\
&\leq A_0\left(\frac{4E\|\rho_b\|}{\kappa_2^m} + 1\right) \tilde{\mathcal{L}}_a^m\big(\beta_J^H h\big)(u) \\
&\leq A_0\left(\frac{4E\|\rho_b\|}{8A_0} + \frac{E\|\rho_b\|}{2A_0}\right) \tilde{\mathcal{L}}_a^m\big(\beta_J^H h\big)(u) \\
&= E\|\rho_b\|\mathcal{N}_{a, J}^H(h)(u).
\end{align*}
\end{proof}

\begin{lemma}
\label{lem:FrameFlowLogLipschitzDolgopyat}
For all $\xi \in \mathbb C$ with $|a| < a_0'$, if $(b, \rho) \in \widehat{M}_0(b_0)$, and if $H \in \mathcal{V}_\rho(\tilde{U})$ and $h \in B(\tilde{U}, \mathbb R)$ satisfy \cref{itm:FrameFlowDominatedByh,itm:FrameFlowLogLipschitzh} in \cref{thm:FrameFlowDolgopyat}, then for all $J \in \mathcal{J}(b, \rho)$ we have
\begin{align*}
\left\|\left(d\tilde{\mathcal{M}}_{\xi, \rho}^m(H)\right)_u\right\|_{\mathrm{op}} \leq E\|\rho_b\|\mathcal{N}_{a, J}^H(h)(u) \qquad \text{for all $u \in \tilde{U}$}.
\end{align*}
\end{lemma}

\begin{proof}
Let $\xi \in \mathbb C$ with $|a| < a_0'$ and suppose $(b, \rho) \in \widehat{M}_0(b_0)$. Suppose $H \in \mathcal{V}_\rho(\tilde{U})$ and $h \in B(\tilde{U}, \mathbb R)$ satisfy \cref{itm:FrameFlowDominatedByh,itm:FrameFlowLogLipschitzh} in \cref{thm:FrameFlowDolgopyat}. Let $J \in \mathcal{J}(b, \rho)$ and $u \in \tilde{U}$. Applying \cref{lem:FrameFlowPreliminaryLogLipschitz,eqn:ConstantE,eqn:Constantm2,eqn:Constantmu}, we have
\begin{align*}
\left\|\left(d\tilde{\mathcal{M}}_{\xi, \rho}^m(H)\right)_u\right\|_{\mathrm{op}} &\leq A_0\left(\frac{E\|\rho_b\|}{\kappa_2^m}\tilde{\mathcal{L}}_a^m(h)(u) + \|\rho_b\|\tilde{\mathcal{L}}_a^m\|H\|(u)\right) \\
&\leq A_0\left(\frac{E}{8A_0} + \frac{E}{2A_0}\right) \|\rho_b\|\tilde{\mathcal{L}}_a^m(h)(u) \\
&\leq \left(\frac{E}{8(1 - N\mu)} + \frac{E}{2(1 - N\mu)}\right) \|\rho_b\|\tilde{\mathcal{L}}_a^m\big(\beta_J^Hh\big)(u) \\
&\leq \left(\frac{E}{6} + \frac{2E}{3}\right) \|\rho_b\|\mathcal{N}_{a, J}^H(h)(u) \\
&\leq E\|\rho_b\|\mathcal{N}_{a, J}^H(h)(u).
\end{align*}
\end{proof}

\subsection{Proof of \texorpdfstring{\cref{itm:FrameFlowDolgopyatProperty2}}{Property \ref{itm:FrameFlowDolgopyatProperty2}} in \texorpdfstring{\cref{thm:FrameFlowDolgopyat}}{\autoref{thm:FrameFlowDolgopyat}}}
\label{subsec:ProofOfFrameFlowDolgopyatProperty2InFrameFlowDolgopyat}
Recall the constants from \cref{eqn:Constantepsilon3,eqn:Constantepsilon4} and note that $\epsilon_4 > 80\epsilon_3$. Let $(b, \rho) \in \widehat{M}_0(b_0)$ and $H \in \mathcal{V}_\rho(\tilde{U})$. For all $k \in \mathcal{A}$ and $1 \leq r \leq r_k^{(b, \rho)}$, define the open sets
\begin{align*}
Z_{k, r, 1}^{(b, \rho)} &= W_{\epsilon_3/\|\rho_b\|}^{\mathrm{su}}\big(\sigma^{m_1}\big(x_{k, r, 1}^{(b, \rho)}\big)\big) \cap \tilde{U}_k; \\
Z_{k, r, 2}^{(b, \rho), H} &= W_{\epsilon_3/\|\rho_b\|}^{\mathrm{su}}\big(\sigma^{m_1}\big(x_{k, r, 2}^{(b, \rho), H}\big)\big) \cap \tilde{U}_k
\end{align*}
which then satisfy $\mathtt{v}_k\big(Z_{k, r, 1}^{(b, \rho)}\big) \subset \slashed{D}_{k, r, 1}^{(b, \rho)}$ and $\mathtt{v}_k\big(Z_{k, r, 2}^{(b, \rho), H}\big) \subset \slashed{D}_{k, r, 2}^{(b, \rho), H}$. We need to first prove the crucial \cref{cor:FedererPropertyCorollary}.

We begin with definitions for this subsection. For all $w \in \T^1(X)$, the Patterson--Sullivan density induces the measure $\mu^{\mathrm{PS}}_{W^{\mathrm{su}}(w)}$ on the leaf $W^{\mathrm{su}}(w)$ defined by
\begin{align*}
d\mu^{\mathrm{PS}}_{W^{\mathrm{su}}(w)}(u) = e^{\delta_\Gamma \beta_{(\tilde{u})^+}(o, \tilde{u})} \, d\mu^{\mathrm{PS}}_{o}((\tilde{u})^+).
\end{align*}
Let $k \in \mathcal{A}$ and $w_k \in R_k$ be the centers. Then, we have
\begin{align*}
\frac{d\left(\nu_{\tilde{U}}|_{\tilde{U}_k}\right)}{d\big(\mu^{\mathrm{PS}}_{W^{\mathrm{su}}(w_k)}\big|_{\tilde{U}_k}\big)}(u) = C \int_{[u, S_k]} e^{\delta_\Gamma \beta_{[\tilde{u}, \tilde{s}]^-}(o, [\tilde{u}, \tilde{s}])} \, d\mu^{\mathrm{PS}}_o([\tilde{u}, \tilde{s}]^-)
\end{align*}
for all $u \in \tilde{U}_k$, for some $C > 0$. In particular, by positivity and continuity of the integrand, there exists $C_{\mathrm{PS}}^\nu > 0$ such that $\frac{1}{C_{\mathrm{PS}}^\nu} \leq \frac{d\left(\nu_{\tilde{U}}|_{\tilde{U}_k}\right)}{d\big(\mu^{\mathrm{PS}}_{W^{\mathrm{su}}(w_k)}\big|_{\tilde{U}_k}\big)} \leq C_{\mathrm{PS}}^\nu$. Recall the trajectory isomorphism $\psi$ from \cite[Definition 1.1]{Rat73}. For all $w \in [W_{\epsilon_0}^{\mathrm{su}}(w_k), W_{\epsilon_0}^{\mathrm{ss}}(w_k)]$, we define another map $\phi_w: U_k \to W^{\mathrm{su}}(w)$ by $\phi_w(u) = \psi_w^{-1}([u, w])$ for all $u \in U_k$. The maps $\phi_w$ are Lipschitz and smooth in $w \in [W_{\epsilon_0}^{\mathrm{su}}(w_k), W_{\epsilon_0}^{\mathrm{ss}}(w_k)]$, and hence there exists $C_\phi = \max_{k \in \mathcal{A}}\sup_{w \in R_k} \Lip_d(\phi_w)$.

\begin{lemma}
\label{lem:Compare_nu_on_Utilde_and_muPS_on_whole_leaf}
For all $j \in \mathcal{A}$, let $w_j \in R_j$ be the centers. There exists $C > 0$ such that for all $j \in \mathcal{A}$, $u \in U_j$, and $\epsilon \in \bigl(0, 2C_{\mathrm{Ano}}C_\phi\hat{\delta}e^{\hat{\delta}}\bigr)$, we have
\begin{align*}
\nu_{\tilde{U}}(W^{\mathrm{su}}_\epsilon(u) \cap \tilde{U}_j) \geq C\mu^{\mathrm{PS}}_{W^{\mathrm{su}}(w_j)}(W^{\mathrm{su}}_\epsilon(u)).
\end{align*}
\end{lemma}

\begin{proof}
By continuity, we can fix $C_1 = \min_{k \in \mathcal{A}} \inf_{u \in U_k, u' \in R_k} e^{\delta_\Gamma \beta_{(\tilde{u})^+}(\tilde{u}, \widetilde{\phi_{u'}(u)})}$ so that we have the bound
\begin{align*}
\mu^{\mathrm{PS}}_{W^{\mathrm{su}}(u')}(\phi_{u'}(U_k)) &= \int_{\phi_{u'}(U_k)} e^{\delta_\Gamma \beta_{(\tilde{u})^+}(o, \tilde{u})} \, d\mu^{\mathrm{PS}}_o((\tilde{u})^+) \\
&= \int_{U_k} e^{\delta_\Gamma \beta_{(\tilde{u})^+}(o, \widetilde{\phi_{u'}(u)})} \, d\mu^{\mathrm{PS}}_o((\tilde{u})^+) \\
&= \int_{U_k}  e^{\delta_\Gamma \beta_{(\tilde{u})^+}(\tilde{u}, \widetilde{\phi_{u'}(u)})} e^{\delta_\Gamma \beta_{(\tilde{u})^+}(o, \tilde{u})} \, d\mu^{\mathrm{PS}}_o((\tilde{u})^+) \\
&\geq C_1 \int_{U_k} e^{\delta_\Gamma \beta_{(\tilde{u})^+}(o, \tilde{u})} \, d\mu^{\mathrm{PS}}_o((\tilde{u})^+) = C_1 \mu^{\mathrm{PS}}_{W^{\mathrm{su}}(w_k)}(U_k)
\end{align*}
for all $u' \in R_k$ and $k \in \mathcal{A}$. Fix $C_2 = \min_{k \in \mathcal{A}} \mu^{\mathrm{PS}}_{W^{\mathrm{su}}(w_k)}(U_k)$. By continuity of the Busemann function, finiteness of the Patterson--Sullivan density, and compactness of $R$, we can fix $C_3 = \sup_{u' \in R} \mu^{\mathrm{PS}}_{W^{\mathrm{su}}(u')}\left(W^{\mathrm{su}}_{2C_{\mathrm{Ano}}^2C_\phi\hat{\delta}e^{\hat{\delta}}}(u')\right)$. Fix $C = \frac{C_1 C_2}{C_{\mathrm{PS}}^\nu C_3}$. Let $j \in \mathcal{A}$, $u \in U_j$, and $\epsilon \in \bigl(0, 2C_{\mathrm{Ano}}C_\phi\hat{\delta}e^{\hat{\delta}}\bigr)$. Let $\alpha = (\alpha_0, \alpha_1, \dotsc, \alpha_l)$ for some $l \in \mathbb N$ be any admissible sequence with $\alpha_0 = j$ such that $u \in \overline{\mathtt{C}[\alpha]}$ and $2C_{\mathrm{Ano}}C_\phi\hat{\delta} \leq e^t \epsilon < 2C_{\mathrm{Ano}}C_\phi\hat{\delta}e^{\hat{\delta}}$ where $t = \tau_\alpha(u)$. Let $k = \alpha_l$ and $u' = ua_t \in R_k$. Note that $\overline{\mathtt{C}[\alpha]} = \sigma^{-\alpha}(U_k) = \phi_{u'}(U_k)a_{-t}$. In fact, since $\diam_{d_{\mathrm{su}}}(\phi_{u'}(U_k)) \leq C_\phi \hat{\delta}$, we have $\overline{\mathtt{C}[\alpha]} \subset W^{\mathrm{su}}_{C_{\mathrm{Ano}}e^{-t}C_\phi \hat{\delta}}(u) \cap U_j \subset W^{\mathrm{su}}_\epsilon(u) \cap U_j$. It is helpful to also note that $W^{\mathrm{su}}(w_j) = W^{\mathrm{su}}(u)$. Thus, we calculate that
\begin{align*}
\nu_{\tilde{U}}(W^{\mathrm{su}}_\epsilon(u) \cap \tilde{U}_j) &\geq \nu_{\tilde{U}}(\overline{\mathtt{C}[\alpha]}) \geq \frac{1}{C_{\mathrm{PS}}^\nu} \mu^{\mathrm{PS}}_{W^{\mathrm{su}}(u)}(\overline{\mathtt{C}[\alpha]}) = \frac{1}{C_{\mathrm{PS}}^\nu}e^{-\delta_\Gamma t} \mu^{\mathrm{PS}}_{W^{\mathrm{su}}(u')}(\overline{\mathtt{C}[\alpha]}a_t) \\
&\geq \frac{C_1}{C_{\mathrm{PS}}^\nu}e^{-\delta_\Gamma t} \mu^{\mathrm{PS}}_{W^{\mathrm{su}}(w_k)}(U_k) \geq \frac{C_1C_2}{C_{\mathrm{PS}}^\nu}e^{-\delta_\Gamma t}.
\end{align*}
On the other hand
\begin{align*}
\mu^{\mathrm{PS}}_{W^{\mathrm{su}}(u)}(W^{\mathrm{su}}_\epsilon(u)) &= e^{-\delta_\Gamma t} \mu^{\mathrm{PS}}_{W^{\mathrm{su}}(u')}(W^{\mathrm{su}}_\epsilon(u) a_t) \\
&\leq e^{-\delta_\Gamma t} \mu^{\mathrm{PS}}_{W^{\mathrm{su}}(u')}\left(W^{\mathrm{su}}_{2C_{\mathrm{Ano}}^2C_\phi\hat{\delta}e^{\hat{\delta}}}(u')\right) \leq C_3e^{-\delta_\Gamma t}.
\end{align*}
Combining the two inequalities above, the lemma follows.
\end{proof}

\begin{corollary}
\label{cor:FedererPropertyOfNu}
The measure $\nu_{\tilde{U}}$ satisfies the \emph{doubling/Federer property}, i.e., there exists $C > 0$ such that for all $k \in \mathcal{A}$, $u \in U_k$, and $\epsilon \in \bigl(0, 2C_{\mathrm{Ano}}C_\phi\hat{\delta}e^{\hat{\delta}}\bigr)$, we have
\begin{align*}
\nu_{\tilde{U}}(W^{\mathrm{su}}_{2\epsilon}(u) \cap \tilde{U}_k) \leq C\nu_{\tilde{U}}(W^{\mathrm{su}}_{\epsilon}(u) \cap \tilde{U}_k).
\end{align*}
\end{corollary}

\begin{proof}
By \cite[Proposition 3.12]{PPS15}, we know that $\mu^{\mathrm{PS}}_{W^{\mathrm{su}}(w_k)}$ satisfies the doubling property for all $k \in \mathcal{A}$. Fix $C_1 > 0$ to be an upper bound for the corresponding doubling constants for all $k \in \mathcal{A}$. Fix $C_2 > 0$ to be the constant from \cref{lem:Compare_nu_on_Utilde_and_muPS_on_whole_leaf}. Fix $C = \frac{C_1 C_{\mathrm{PS}}^\nu}{C_2}$. Let $k \in \mathcal{A}$, $u \in U_k$, and $\epsilon \in \bigl(0, 2C_{\mathrm{Ano}}C_\phi\hat{\delta}e^{\hat{\delta}}\bigr)$. We have
\begin{align*}
\nu_{\tilde{U}}(W^{\mathrm{su}}_{2\epsilon}(u) \cap \tilde{U}_k)
&\leq C_{\mathrm{PS}}^\nu \mu^{\mathrm{PS}}_{W^{\mathrm{su}}(w_k)}(W^{\mathrm{su}}_{2\epsilon}(u)) \leq C_1 C_{\mathrm{PS}}^\nu \mu^{\mathrm{PS}}_{W^{\mathrm{su}}(w_k)}(W^{\mathrm{su}}_{\epsilon}(u)) \\
&\leq \frac{C_1 C_{\mathrm{PS}}^\nu}{C_2}\nu_{\tilde{U}}(W^{\mathrm{su}}_\epsilon(u) \cap \tilde{U}_k) = C\nu_{\tilde{U}}(W^{\mathrm{su}}_\epsilon(u) \cap \tilde{U}_k).
\end{align*}
\end{proof}

\begin{corollary}
\label{cor:FedererPropertyCorollary}
There exists $C > 1$ such that for all $(b, \rho) \in \widehat{M}_0(b_0)$, $k \in \mathcal{A}$, and $u \in U_k$, we have
\begin{align*}
\nu_{\tilde{U}}\left(W_{\epsilon_4/\|\rho_b\|}^{\mathrm{su}}(u) \cap \tilde{U}_k\right) \leq C\nu_{\tilde{U}}\left(W_{\epsilon_3/\|\rho_b\|}^{\mathrm{su}}(u) \cap \tilde{U}_k\right).
\end{align*}
\end{corollary}

For all $(b, \rho) \in \widehat{M}_0(b_0)$, $H \in \mathcal{V}_\rho(\tilde{U})$, and $J \in \mathcal J(b, \rho)$, define the set $Z_J^H = \bigcup_{(k, r, p, l) \in J} Z_{k, r, p}^{(b, \rho), H}$.

\begin{lemma}
\label{lem:WDenseInequality}
There exists $\eta \in (0, 1)$ such that for all $(b, \rho) \in \widehat{M}_0(b_0)$, $J \in \mathcal J(b, \rho)$, $H \in \mathcal{V}_\rho(\tilde{U})$, and $h \in K_{2E\|\rho_b\|}(\tilde{U})$, we have
\begin{align*}
\int_{Z_J^H} h \, d\nu_{\tilde{U}} \geq \eta \int_{\tilde{U}} h \, d\nu_{\tilde{U}}.
\end{align*}
\end{lemma}

\begin{proof}
Fix $C$ to be the one provided by \cref{cor:FedererPropertyCorollary} and $\eta = (Ce^{4E\epsilon_4})^{-1} \in (0, 1)$. Let $(b, \rho) \in \widehat{M}_0(b_0)$, $J \in \mathcal J(b, \rho)$, $H \in \mathcal{V}_\rho(\tilde{U})$, and $h \in K_{E\|\rho_b\|}(\tilde{U})$. Denote $\epsilon_j' = \frac{\epsilon_j}{\|\rho_b\|}$ and $W_{j, k}(u) = W_{\epsilon_j'}^{\mathrm{su}}(u) \cap \tilde{U}_k$ for all $u \in \tilde{U}_k$, $k \in \mathcal{A}$, and $j \in \{3, 4\}$. Define
\begin{align*}
P_k = \big\{\sigma^{m_1}\big(x_{k, r, p}^{(b, \rho), H}\big) \in U: (k, r, p, l) \in J \text{ for some } l \in \{1, 2\}\big\}.
\end{align*}
Since $\big\{\hat{C}_{k, r}^{(b, \rho)} \subset W^{\mathrm{su}}(w_1): 1 \leq r \leq r_k^{(b, \rho)}\big\}$, where $w_1 \in R_1$ is the center, covers $\mathtt{C}_k$ for all $k \in \mathcal{A}$ and $J \subset \Xi(b, \rho)$ is dense, so $\big\{W_{\epsilon_4'}^{\mathrm{su}}(x) \subset \tilde{U}_k: x \in P_k\big\}$ covers $\interior(U_k)$ for all $k \in \mathcal{A}$. Let $l_x = \inf_{u \in W_{4, k}(x)} h(u)$ and $L_x = \sup_{u \in W_{4, k}(x)} h(u)$ for all $x \in P_k$ and $k \in \mathcal{A}$. Using $|\log \circ h|_{C^1} \leq 2E\|\rho_b\|$, we can derive that $L_x \leq l_x e^{2E\|\rho_b\| \diam_d\big(W_{\epsilon_4'}^{\mathrm{su}}(x)\big)} = l_x e^{4E\epsilon_4}$. Hence, by \cref{cor:FedererPropertyCorollary}, we have
\begin{align*}
\int_{\tilde{U}} h(u) \, d\nu_{\tilde{U}}(u) &\leq \sum_{k \in \mathcal{A}} \sum_{x \in P_k} \int_{W_{4, k}(x)} h(u) \, d\nu_{\tilde{U}}(u) \\
&\leq \sum_{k \in \mathcal{A}} \sum_{x \in P_k} L_x \cdot \nu_{\tilde{U}}(W_{4, k}(x)) \\
&\leq Ce^{4E\epsilon_4} \sum_{k \in \mathcal{A}} \sum_{x \in P_k} l_x \cdot \nu_{\tilde{U}}(W_{3, k}(x)) \\
&\leq Ce^{4E\epsilon_4} \sum_{k \in \mathcal{A}} \sum_{x \in P_k} \int_{W_{3, k}(x)} h(u) \, d\nu_{\tilde{U}}(u) \\
&\leq \frac{1}{\eta} \int_{Z_J^H} h(u) \, d\nu_{\tilde{U}}(u).
\end{align*}
\end{proof}

\begin{lemma}
\label{lem:FrameFlowDolgopyatProperty2}
There exist $a_0 > 0$ and $\eta \in (0, 1)$ such that for all $\xi \in \mathbb C$ with $|a| < a_0$, if $(b, \rho) \in \widehat{M}_0(b_0)$, then for all $J \in \mathcal J(b, \rho)$, $H \in \mathcal{V}_\rho(\tilde{U})$, and $h \in K_{E\|\rho_b\|}(\tilde{U})$, we have $\left\|\mathcal{N}_{a, J}^H(h)\right\|_2 \leq \eta \|h\|_2$.
\end{lemma}

\begin{proof}
Fix $\eta' \in (0, 1)$ to be the $\eta$ provided by \cref{lem:WDenseInequality}. Fix a positive constant
\begin{align*}
a_0 < \min\left(a_0', \frac{1}{mA_f}\log\left(\frac{1}{1 - \eta'\mu e^{-mT_0}}\right)\right)
\end{align*}
so that we can also fix $\eta = \sqrt{e^{mA_fa_0}(1 - \eta'\mu e^{-mT_0})} \in (0, 1)$. Let $\xi \in \mathbb C$ with $|a| < a_0$. Suppose $(b, \rho) \in \widehat{M}_0(b_0)$. Let $J \in \mathcal J(b, \rho)$, $H \in \mathcal{V}_\rho(\tilde{U})$, and $h \in K_{E\|\rho_b\|}(\tilde{U})$. We have the estimate $\mathcal{N}_{a, J}^H(h)^2 \leq e^{mA_fa_0}\mathcal{N}_{0, J}^H(h)^2$ since $\left|f_{(j, k)}^{(a)} - f_{(j, k)}^{(0)}\right| \leq A_f|a|$ for all admissible pairs $(j, k)$ and by the Cauchy--Schwarz inequality, we have
\begin{align*}
\mathcal{N}_{0, J}^H(h)^2 = \tilde{\mathcal{L}}_0^m\big(\beta_J^H h\big)^2 \leq \tilde{\mathcal{L}}_0^m\Big(\big(\beta_J^H\big)^2\Big) \tilde{\mathcal{L}}_0^m(h^2).
\end{align*}
It is easy to see that $h^2 \in K_{2E\|\rho_b\|}(\tilde{U})$. Then \cref{lem:FrameFlowPreliminaryLogLipschitz} gives $\tilde{\mathcal{L}}_0^m(h^2) \in K_{B'}(\tilde{U})$ where $B' = A_0\left(\frac{2E|b|}{\kappa_2^m} + 1\right) \leq A_0\left(\frac{2E|b|}{8A_0} + \frac{E|b|}{2A_0}\right) \leq 2E|b|$. So $\tilde{\mathcal{L}}_0^m(h^2) \in K_{2E\|\rho_b\|}(\tilde{U})$. Now, \cref{lem:WDenseInequality} gives $\int_{Z_J^H} \tilde{\mathcal{L}}_0^m(h^2) \, d\nu_{\tilde{U}} \geq \eta' \int_{\tilde{U}} \tilde{\mathcal{L}}_0^m(h^2) \, d\nu_{\tilde{U}}$. Note that
\begin{align*}
\tilde{\mathcal{L}}_0^m\Big(\big(\beta_J^H\big)^2\Big)(u) \leq \tilde{\mathcal{L}}_0^m\Big(\chi_{\tilde{U}} - \mu \tilde{\psi}_{(k, r, p, l)}^{(b, \rho), H}\Big)(u) \leq 1 - \mu e^{-mT_0}
\end{align*}
for all $u \in Z_J^H$ by choosing any $(k, r, p, l) \in J$. So putting everything together and using $\tilde{\mathcal{L}}_0^*(\nu_{\tilde{U}}) = \nu_{\tilde{U}}$ (which is easily derived from $\mathcal{L}_0^*(\nu_U) = \nu_U$), we have
\begin{align*}
&\int_{\tilde{U}} \mathcal{N}_{a, J}^H(h)^2 \, d\nu_{\tilde{U}} \leq \int_{\tilde{U}} e^{mA_fa_0}\mathcal{N}_{0, J}^H(h)^2 \, d\nu_{\tilde{U}} \\
\leq{}&e^{mA_fa_0}\left(\int_{Z_J^H} \tilde{\mathcal{L}}_0^m\Big(\big(\beta_J^H\big)^2\Big) \tilde{\mathcal{L}}_0^m(h^2) \, d\nu_{\tilde{U}} + \int_{\tilde{U} \setminus Z_J^H} \tilde{\mathcal{L}}_0^m\Big(\big(\beta_J^H\big)^2\Big) \tilde{\mathcal{L}}_0^m(h^2) \, d\nu_{\tilde{U}}\right) \\
\leq{}&e^{mA_fa_0}\left(\bigl(1 - \mu e^{-mT_0}\bigr)\int_{Z_J^H} \tilde{\mathcal{L}}_0^m(h^2) \, d\nu_{\tilde{U}} + \int_{\tilde{U} \setminus Z_J^H} \tilde{\mathcal{L}}_0^m(h^2) \, d\nu_{\tilde{U}}\right) \\
={}&e^{mA_fa_0}\left(\int_{\tilde{U}} \tilde{\mathcal{L}}_0^m(h^2) \, d\nu_{\tilde{U}} - \mu e^{-mT_0}\int_{Z_J^H} \tilde{\mathcal{L}}_0^m(h^2) \, d\nu_{\tilde{U}}\right) \\
\leq{}&e^{mA_fa_0}\bigl(1 - \eta'\mu e^{-mT_0}\bigr)\int_{\tilde{U}} \tilde{\mathcal{L}}_0^m(h^2) \, d\nu_{\tilde{U}} \\
={}&\eta^2 \int_{\tilde{U}} h^2 \, d\nu_{\tilde{U}}.
\end{align*}
\end{proof}

\subsection{Proof of \texorpdfstring{\cref{itm:FrameFlowDominatedByDolgopyat}}{Property \ref{itm:FrameFlowDominatedByDolgopyat}} in \texorpdfstring{\cref{thm:FrameFlowDolgopyat}}{\autoref{thm:FrameFlowDolgopyat}}}
Now, for all $\xi \in \mathbb C$ with $|a| < a_0'$, if $(b, \rho) \in \widehat{M}_0(b_0)$, then for all $H \in \mathcal{V}_\rho(\tilde{U})$, $h \in K_{E\|\rho_b\|}(\tilde{U})$, and $1 \leq j \leq j_{\mathrm{m}}$, we define the functions $\chi_1^j[\xi, \rho, H, h], \chi_2^j[\xi, \rho, H, h]: \tilde{U}_1 \to \mathbb C$ by
\begin{align*}
&\chi_1^j[\xi, \rho, H, h](u) \\
={}&\frac{\left\|e^{f_{\alpha_0}^{(a)}(v_0(u))} \rho_b(\Phi^{\alpha_0}(v_0(u))^{-1})H(v_0(u)) + e^{f_{\alpha_j}^{(a)}(v_j(u))} \rho_b(\Phi^{\alpha_j}(v_j(u))^{-1})H(v_j(u))\right\|_2}{(1 - N\mu)e^{f_{\alpha_0}^{(a)}(v_0(u))}h(v_0(u)) + e^{f_{\alpha_j}^{(a)}(v_j(u))}h(v_j(u))}
\end{align*}
and
\begin{align*}
&\chi_2^j[\xi, \rho, H, h](u) \\
={}&\frac{\left\|e^{f_{\alpha_0}^{(a)}(v_0(u))} \rho_b(\Phi^{\alpha_0}(v_0(u))^{-1})H(v_0(u)) + e^{f_{\alpha_j}^{(a)}(v_j(u))} \rho_b(\Phi^{\alpha_j}(v_j(u))^{-1})H(v_j(u))\right\|_2}{e^{f_{\alpha_0}^{(a)}(v_0(u))}h(v_0(u)) + (1 - N\mu)e^{f_{\alpha_j}^{(a)}(v_j(u))}h(v_j(u))}
\end{align*}
for all $u \in \tilde{U}_1$.

\begin{lemma}
\label{lem:FrameFlowHTrappedByh}
Let $(b, \rho) \in \widehat{M}_0(b_0)$. Suppose that $H \in \mathcal{V}_\rho(\tilde{U})$ and $h \in K_{E\|\rho_b\|}(\tilde{U})$ satisfy \cref{itm:FrameFlowDominatedByh,itm:FrameFlowLogLipschitzh} in \cref{thm:FrameFlowDolgopyat}. Then for all $(k, r, p, l) \in \Xi(b, \rho)$, denoting $0$ by $j$ if $l = 1$ and $j_{k, r}^{(b, \rho), H}$ by $j$ if $l = 2$, we have
\begin{align*}
\frac{1}{2} \leq \frac{h(v_j(u))}{h(v_j(u'))} \leq 2 \qquad \text{for all $u, u' \in \hat{D}_{k, r, p}^{(b, \rho), H}$}
\end{align*}
and also either of the alternatives
\begin{alternative}
\item\label{alt:FrameFlowHLessThan3/4h}	$\|H(v_j(u))\|_2 \leq \frac{3}{4}h(v_j(u))$ for all $u \in \hat{D}_{k, r, p}^{(b, \rho), H}$;
\item\label{alt:FrameFlowHGreaterThan1/4h}	$\|H(v_j(u))\|_2 \geq \frac{1}{4}h(v_j(u))$ for all $u \in \hat{D}_{k, r, p}^{(b, \rho), H}$.
\end{alternative}
\end{lemma}

\begin{proof}
Let $(b, \rho) \in \widehat{M}_0(b_0)$. Suppose that $H \in \mathcal{V}_\rho(\tilde{U})$ and $h \in K_{E\|\rho_b\|}(\tilde{U})$ satisfy \cref{itm:FrameFlowDominatedByh,itm:FrameFlowLogLipschitzh} in \cref{thm:FrameFlowDolgopyat}. Let $(k, r, p, l) \in \Xi(b, \rho)$. We show the first inequality. Let $u, u' \in \hat{D}_{k, r, p}^{(b, \rho), H}$. Since $|\log \circ h|_{C^1} \leq E\|\rho_b\|$, using \cref{eqn:Constantm2}, we have
\begin{align*}
|\log(h(v_j(u))) - \log(h(v_j(u')))| &\leq |\log \circ h|_{C^1} \cdot |v_j|_{C^1} \cdot d(u, u') \\
&\leq E\|\rho_b\| \cdot \frac{1}{c_0\kappa_2^{m_2}} \cdot \diam_d\left(\hat{D}_{k, r, p}^{(b, \rho), H}\right) \leq \frac{4EN\epsilon_2}{c_0\kappa_2^{m_2}} \\
&\leq \log(2).
\end{align*}
Hence $\left|\log\left(\frac{h(v_j(u))}{h(v_j(u'))}\right)\right| \leq \log(2)$ which implies the first inequality.

Now we show the alternatives. If $\|H(v_j(u))\|_2 \geq \frac{1}{4}h(v_j(u))$ for all $u \in \hat{D}_{k, r, p}^{(b, \rho), H}$, then we are done. Otherwise, there exists $u_0 \in \hat{D}_{k, r, p}^{(b, \rho), H}$ such that $\|H(v_j(u_0))\|_2 \leq \frac{1}{4}h(v_j(u_0))$. Let $u \in \hat{D}_{k, r, p}^{(b, \rho), H}$, $D = d(u_0, u) \leq \diam_d\big(\hat{D}_{k, r, p}^{(b, \rho), H}\big) = \frac{4N\epsilon_2}{\|\rho_b\|}$, and $\gamma: [0, D] \to \tilde{U}_1$ be a unit speed geodesic from $\gamma(0) = u_0$ to $\gamma(D) = u$. Note that $H(v_j(u)) = H(v_j(u_0)) + \int_0^D (H \circ v_j \circ \gamma)'(t) \, dt$. Then using the first proven inequality and \cref{eqn:Constantm2}, we have
\begin{align*}
\|H(v_j(u))\|_2 &\leq \|H(v_j(u_0))\|_2 + \int_0^D \|(dH)_{v_j(\gamma(t))}\|_{\mathrm{op}} |v_j|_{C^1} \, dt \\
&\leq \frac{1}{4}h(v_j(u_0)) + \int_0^D E\|\rho_b\|h(v_j(\gamma(t))) \cdot \frac{1}{c_0\kappa_2^{m_2}} \, dt \\
&\leq \frac{1}{2}h(v_j(u)) + \frac{E\|\rho_b\|}{c_0\kappa_2^{m_2}}\int_0^D 2h(v_j(\gamma(D))) \, dt \\
&\leq \left(\frac{1}{2} + \frac{8EN\epsilon_2}{c_0\kappa_2^{m_2}}\right) h(v_j(u)) \\
&\leq \frac{3}{4}h(v_j(u)).
\end{align*}
\end{proof}

For any $k \geq 2$, let $\Theta: (\mathbb R^k \setminus \{0\}) \times (\mathbb R^k \setminus \{0\}) \to [0, \pi]$ be the map which gives the angle defined by $\Theta(w_1, w_2) = \arccos\left(\frac{\langle w_1, w_2\rangle}{\|w_1\| \cdot \|w_2\|}\right)$ for all $w_1, w_2 \in \mathbb R^k \setminus \{0\}$, where we use the standard inner product and norm. The following lemma can be proven by elementary trigonometry.

\begin{lemma}
\label{lem:StrongTriangleInequality}
Let $k \geq 2$. If $w_1, w_2 \in \mathbb R^k \setminus \{0\}$ such that $\Theta(w_1, w_2) \geq \alpha$ and $\frac{\|w_1\|}{\|w_2\|} \leq L$ for some $\alpha \in [0, \pi]$ and $L \geq 1$, then we have
\begin{align*}
\|w_1 + w_2\| \leq \left(1 - \frac{\alpha^2}{16L}\right)\|w_1\| + \|w_2\|.
\end{align*}
\end{lemma}

\begin{lemma}
\label{lem:chiLessThan1}
Let $\xi \in \mathbb C$ with $|a| < a_0'$ and $(b, \rho) \in \widehat{M}_0(b_0)$. Suppose $H \in \mathcal{V}_\rho(\tilde{U})$ and $h \in K_{E\|\rho_b\|}(\tilde{U})$ satisfy \cref{itm:FrameFlowDominatedByh,itm:FrameFlowLogLipschitzh} in \cref{thm:FrameFlowDolgopyat}. For all $(k, r) \in \Xi_1(b, \rho)$, denoting $j_{k, r}^{(b, \rho), H}$ by $j$, there exists $(p, l) \in \Xi_2$ such that $\chi_l^j[\xi, \rho, H, h](u) \leq 1$ for all $u \in \hat{D}_{k, r, p}^{(b, \rho), H}$.
\end{lemma}

\begin{proof}
Let $\xi \in \mathbb C$ with $|a| < a_0'$ and $(b, \rho) \in \widehat{M}_0(b_0)$. Suppose $H \in \mathcal{V}_\rho(\tilde{U})$ and $h \in K_{E\|\rho_b\|}(\tilde{U})$ satisfy \cref{itm:FrameFlowDominatedByh,itm:FrameFlowLogLipschitzh} in \cref{thm:FrameFlowDolgopyat}. Let $(k, r) \in \Xi_1(b, \rho)$. Denote $j_{k, r}^{(b, \rho), H}$ by $j$, $x_{k, r, 1}^{(b, \rho)}$ by $x_1$, $x_{k, r, 2}^{(b, \rho), H}$ by $x_2$, and $\hat{D}_{k, r, p}^{(b, \rho), H}$ by $\hat{D}_p$. Now, suppose \cref{alt:FrameFlowHLessThan3/4h} in \cref{lem:FrameFlowHTrappedByh} holds for $(k, r, p, l) \in \Xi(b, \rho)$ for some $(p, l) \in \Xi_2$. Then it is easy to check that $\chi_l^j[\xi, \rho, H, h](u) \leq 1$ for all $u \in \hat{D}_p$. Otherwise, \cref{alt:FrameFlowHGreaterThan1/4h} in \cref{lem:FrameFlowHTrappedByh} holds for $(k, r, 1, 1), (k, r, 1, 2), (k, r, 2, 1), (k, r, 2, 2) \in \Xi(b, \rho)$. We would like to use \cref{lem:StrongTriangleInequality} but first we need to establish bounds on relative angle and relative size. We start with the former. Define $\omega_\ell(u) = \frac{H(v_\ell(u))}{\|H(v_\ell(u))\|_2}$ and $\phi_\ell(u) = \Phi^{\alpha_\ell}(v_\ell(u))$ for all $u \in \tilde{U}_1$ and $\ell \in \{0, j\}$. Let $D = 2\dim(\rho)^2$ and define the map $\varphi: \mathbb R^D \setminus \{0\} \to \mathbb R^D$ by $\varphi(w) = \frac{w}{\|w\|}$ for all $w \in \mathbb R^D \setminus \{0\}$, where we use the standard inner product and norm on $\mathbb R^D$. Then we note that $\|(d\varphi)_w\|_{\mathrm{op}} = \frac{1}{\|w\|}$ for all $w \in \mathbb R^D$. We can write $\omega_\ell = \varphi \circ H \circ v_\ell$ using the isomorphism $V_\rho^{\oplus \dim(\rho)} \cong \mathbb R^D$ of real vector spaces. Then using \cref{lem:SigmaHyperbolicity,eqn:Constantm2}, we calculate that
\begin{align*}
\|(d\omega_\ell)_u\|_{\mathrm{op}} &\leq \|(d\varphi)_{H(v_\ell(u))}\|_{\mathrm{op}} \|(dH)_{v_\ell(u)}\|_{\mathrm{op}} \|(dv_\ell)_u\|_{\mathrm{op}} \\
&\leq \frac{1}{\|H(v_\ell(u))\|_2} \cdot E\|\rho_b\|h(v_\ell(u)) \cdot \frac{1}{c_0\kappa_2^{m_2}} \\
&\leq \frac{4E\|\rho_b\|}{c_0\kappa_2^{m_2}} \leq \delta_1\|\rho_b\|
\end{align*}
for all $u \in \hat{D}_p$, $\ell \in \{0, j\}$, and $p \in \{1, 2\}$. In other words, $\omega_0$ and $\omega_j$ are Lipschitz on $\hat{D}_p$ with Lipschitz constant $\delta_1\|\rho_b\|$ for all $p \in \{1, 2\}$. Define
\begin{align*}
\begin{split}
V_\ell(u) &= e^{f_{\alpha_\ell}^{(a)}(v_\ell(u))} \rho_b(\phi_\ell(u)^{-1})H(v_\ell(u)); \\
\hat{V}_\ell(u) &= \frac{V_\ell(u)}{\|V_\ell(u)\|_2} = \rho_b(\phi_\ell(u)^{-1})\omega_\ell(u);
\end{split}
\qquad
\text{for all $u \in \tilde{U}_1$ and $\ell \in \{0, j\}$}.
\end{align*}
Since $\omega_0$ and $\omega_j$ are Lipschitz and $d(x_1, x_2) \leq \frac{\epsilon_1}{2\|\rho_b\|}$, we have
\begin{align*}
&\big\|\hat{V}_0(x_2) - \hat{V}_j(x_2)\big\|_2 \\
={}&\|\rho_b(\phi_0(x_2)^{-1})\omega_0(x_2) - \rho_b(\phi_j(x_2)^{-1})\omega_j(x_2)\|_2 \\
={}&\|\rho_b(\phi_j(x_2)\phi_0(x_2)^{-1})\omega_0(x_2) - \omega_j(x_2)\|_2 \\
\geq{}&\|\rho_b(\phi_j(x_2)\phi_0(x_2)^{-1})\omega_0(x_1) - \omega_j(x_1)\|_2 \\
{}&- \|\rho_b(\phi_j(x_2)\phi_0(x_2)^{-1})\omega_0(x_2) - \rho_b(\phi_j(x_2)\phi_0(x_2)^{-1})\omega_0(x_1)\|_2 \\
{}&- \|\omega_j(x_2) - \omega_j(x_1)\|_2 \\
={}&\|\rho_b(\phi_j(x_2)\phi_0(x_2)^{-1})\omega_0(x_1) - \omega_j(x_1)\|_2 - \|\omega_0(x_2) - \omega_0(x_1)\|_2 \\
{}&- \|\omega_j(x_2) - \omega_j(x_1)\|_2 \\
\geq{}&\|\rho_b(\phi_j(x_2)\phi_0(x_2)^{-1})\omega_0(x_1) - \rho_b(\phi_j(x_1)\phi_0(x_1)^{-1})\omega_0(x_1)\|_2 \\
{}&- \|\rho_b(\phi_j(x_1)\phi_0(x_1)^{-1})\omega_0(x_1) - \omega_j(x_1)\|_2 - \delta_1\epsilon_1 \\
={}&\|\rho_b(\phi_0(x_1)^{-1})\omega_0(x_1) - \rho_b(\phi_0(x_1)^{-1}\phi_0(x_2)\phi_j(x_2)^{-1}\phi_j(x_1)\phi_0(x_1)^{-1})\omega_0(x_1)\|_2 \\
{}&- \|\rho_b(\phi_0(x_1)^{-1})\omega_0(x_1) - \rho_b(\phi_j(x_1)^{-1})\omega_j(x_1)\|_2 - \delta_1\epsilon_1 \\
\geq{}&\|\rho_b(\phi_0(x_1)^{-1})\omega_0(x_1) - \rho_b(\BP_j(x_1, x_2))\rho_b(\phi_0(x_1)^{-1})\omega_0(x_1)\|_2 \\
&{}- \big\|\hat{V}_0(x_1) - \hat{V}_j(x_1)\big\|_2 - \delta_1\epsilon_1.
\end{align*}
Denote $\omega = \rho_b(\phi_0(x_1)^{-1})\omega_0(x_1)$ and $Z = d(\BP_{j, x_1} \circ \Psi)_{\check{x}_1}(z)$ where $z = (\check{x}_1, \check{x}_2 - \check{x}_1) \in \T_{\check{x}_1}(\mathbb R^{n - 1})$. Recall the curve $\varphi^{\mathrm{BP}}_{j, x_1, z}: [0, 1] \to AM$ defined by $\varphi^{\mathrm{BP}}_{j, x_1, z}(t) = \BP_{j, x_1}(\Psi(\check{x}_1 + tz))$ for all $t \in [0, 1]$. Recall that $\bigl(\varphi^{\mathrm{BP}}_{j, x_1, z}\bigr)'(0) = Z$ and $\varphi^{\mathrm{BP}}_{j, x_1, z}(0) = \BP_{j, x_1}(x_1) = e$ and $\varphi^{\mathrm{BP}}_{j, x_1, z}(1) = \BP_{j, x_1}(x_2) = \BP_j(x_1, x_2)$. Continuing to bound the first term above, we apply \cref{lem:PartnerPointInZariskiDenseLimitSetForBPBound,lem:ComparingExpWithBP,eqn:Constantepsilon1} to get
\begin{align*}
&\|\omega - \rho_b(\BP_j(x_1, x_2))(\omega)\|_2 \\
\geq{}&\|\omega - \rho_b(\exp(Z))(\omega)\|_2 - \left\|\rho_b(\exp(Z))(\omega) - \rho_b\left(\varphi^{\mathrm{BP}}_{j, x_1, z}(1)\right)(\omega)\right\|_2 \\
\geq{}&\|\omega - \exp(d\rho_b(Z))(\omega)\|_2 - \|\rho_b\| \cdot d_{AM}\left(\exp(Z), \varphi^{\mathrm{BP}}_{j, x_1, z}(1)\right) \\
\geq{}&\|d\rho_b(Z)(\omega)\|_2 - \|\rho_b\|^2 \|Z\|^2 - \|\rho_b\| \cdot d_{AM}\left(\exp(Z), \varphi^{\mathrm{BP}}_{j, x_1, z}(1)\right) \\
\geq{}&\|d\rho_b(Z)(\omega)\|_2 - \|\rho_b\|^2 (C_{\BP, \Psi} C_{\Psi})^2 d(x_1, x_2)^2 - C_{\exp, \BP} \cdot \|\rho_b\| \cdot d(x_1, x_2)^2 \\
\geq{}&7\delta_1\epsilon_1 - \delta_1\epsilon_1 - \delta_1\epsilon_1 \geq 5\delta_1\epsilon_1.
\end{align*}
Hence, we have
\begin{align*}
\big\|\hat{V}_0(x_1) - \hat{V}_j(x_1)\big\|_2 + \big\|\hat{V}_0(x_2) - \hat{V}_j(x_2)\big\|_2 \geq 4\delta_1\epsilon_1.
\end{align*}
Then, $\big\|\hat{V}_0(x_p) - \hat{V}_j(x_p)\big\|_2 \geq 2\delta_1\epsilon_1$ for some $p \in \{1, 2\}$. Recalling estimates from \cref{lem:FrameFlowPreliminaryLogLipschitz}, \cref{eqn:Constantepsilon2}, and that $\omega_\ell$ is Lipschitz, we have
\begin{align*}
&\big\|\hat{V}_\ell(x_p) - \hat{V}_\ell(u)\big\|_2 \\
={}&\|(\rho_b(\phi_\ell(x_p)^{-1}) - \rho_b(\phi_\ell(u)^{-1}))\omega_\ell(x_p) + \rho_b(\phi_\ell(u)^{-1})(\omega_\ell(x_p) - \omega_\ell(u))\|_2 \\
\leq{}&\|(\rho_b(\phi_\ell(x_p)^{-1}) - \rho_b(\phi_\ell(u)^{-1}))\omega_\ell(x_p)\|_2 + \|\omega_\ell(x_p) - \omega_\ell(u)\|_2 \\
\leq{}&A_0 \|\rho_b\| d(x_p, u) + \delta_1\|\rho_b\| d(x_p, u) \\
\leq{}&(A_0 + \delta_1)\|\rho_b\| \cdot \frac{2N\epsilon_2}{\|\rho_b\|} = 2N\epsilon_2(A_0 + \delta_1) \leq \frac{\delta_1\epsilon_1}{2}
\end{align*}
for all $u \in \hat{D}_p$ and $\ell \in \{0, j\}$. Hence $\big\|\hat{V}_0(u) - \hat{V}_j(u)\big\|_2 \geq \delta_1\epsilon_1 \in (0, 1)$ for all $u \in \hat{D}_p$. Then using the cosine law, the required bound for relative angle is
\begin{align*}
\Theta(V_0(u), V_j(u)) = \Theta(\hat{V}_0(u), \hat{V}_j(u)) \geq \arccos\left(1 - \frac{(\delta_1\epsilon_1)^2}{2}\right) \in (0, \pi).
\end{align*}
For the bound on relative size, let $(\ell, \ell') \in \{(0, j), (j, 0)\}$ such that $h(v_\ell(u_0)) \leq h(v_{\ell'}(u_0))$ for some $u_0 \in \hat{D}_p$. Let $l = 1$ if $(\ell, \ell') = (0, j)$ and $l = 2$ if $(\ell, \ell') = (0, j)$. Recalling that $\rho_b$ is a unitary representation, by \cref{lem:FrameFlowHTrappedByh}, we have
\begin{align*}
\frac{\|V_\ell(u)\|_2}{\|V_{\ell'}(u)\|_2} &= \frac{e^{f_{\alpha_\ell}^{(a)}(v_\ell(u))}\|H(v_\ell(u))\|_2}{e^{f_{\alpha_{\ell'}}^{(a)}(v_{\ell'}(u))}\|H(v_{\ell'}(u))\|_2} \leq \frac{4e^{f_{\alpha_\ell}^{(a)}(v_\ell(u)) - f_{\alpha_{\ell'}}^{(a)}(v_{\ell'}(u))}h(v_{\ell}(u))}{h(v_{\ell'}(u))} \\
&\leq \frac{16e^{2m_2 T_0}h(v_{\ell}(u_0))}{h(v_{\ell'}(u_0))} \leq 16e^{2m_2 T_0}
\end{align*}
for all $u \in \hat{D}_p$, which is the required bound on relative size. Now using \cref{lem:StrongTriangleInequality}, \cref{eqn:Constantmu}, and $\|H\| \leq h$ on $\|V_\ell(u) + V_{\ell'}(u)\|_2$ gives $\chi_l^j[\xi, \rho, H, h](u) \leq 1$ for all $u \in \hat{D}_p$.
\end{proof}

\begin{lemma}
For all $\xi \in \mathbb C$ with $|a| < a_0'$, if $(b, \rho) \in \widehat{M}_0(b_0)$, and if $H \in \mathcal{V}_\rho(\tilde{U})$ and $h \in K_{E\|\rho_b\|}(\tilde{U})$ satisfy \cref{itm:FrameFlowDominatedByh,itm:FrameFlowLogLipschitzh} in \cref{thm:FrameFlowDolgopyat}, then there exists $J \in \mathcal{J}(b, \rho)$ such that
\begin{align*}
\big\|\tilde{\mathcal{M}}_{\xi, \rho}^m(H)(u)\big\|_2 \leq \mathcal{N}_{a, J}^H(h)(u) \qquad \text{for all $u \in \tilde{U}$}.
\end{align*}
\end{lemma}

\begin{proof}
Let $\xi \in \mathbb C$ with $|a| < a_0'$ and suppose $(b, \rho) \in \widehat{M}_0(b_0)$. Suppose $H \in \mathcal{V}_\rho(\tilde{U})$ and $h \in K_{E\|\rho_b\|}(\tilde{U})$ satisfy \cref{itm:FrameFlowDominatedByh,itm:FrameFlowLogLipschitzh} in \cref{thm:FrameFlowDolgopyat}. We drop superscripts $(b, \rho)$ and $H$ to simply notation. For all $(k, r) \in \Xi_1(b, \rho)$, there exists $(p_{k, r}, l_{k, r}) \in \Xi_2$ as guaranteed by \cref{lem:chiLessThan1}. Let $J_0 = \{(k, r, p_{k, r}, l_{k, r}) \in \Xi(b, \rho): (k, r) \in \Xi_1(b, \rho)\} \subset \Xi(b, \rho)$ which is then dense by construction and so $J_0 \in \mathcal{J}(b, \rho)$. Now, we make necessary modifications to $J_0$ to define $J \in \mathcal{J}(b, \rho)$. Recall the notations from the proof of \cref{lem:NumberOfIntersectingBallsLessThanOrEqualToN}. For all equivalence classes $[D_{k, r, p}] \in D_\cup^{\mathrm{conn}}$, we do the following. Choose any representative, say $D_{k, r, p} \in [D_{k, r, p}]$ and make the modification $j_{k', r'} = j_{k, r}$ and $l_{k', r'} = l_{k, r}$ for all $(k', r') \in \Xi_1(b, \rho)$ with $D_{k', r', p'} \in [D_{k, r, p}]$ for some $p' \in \{1, 2\}$. Define $J \in \mathcal{J}(b, \rho)$ by $J = \{(k, r, p_{k, r}, l_{k, r}) \in \Xi(b, \rho): (k, r) \in \Xi_1(b, \rho)\} \subset \Xi(b, \rho)$. Now let $u \in \tilde{U}$. If $u \notin D_{k, r, p}$ for all $(k, r, p, l) \in J$, then $\beta_J^H(v) = 1$ for all branches $v = \sigma^{-\alpha}(u)$ where $\alpha$ is an admissible sequence with $\len(\alpha) = m_2$. Hence $\big\|\tilde{\mathcal{M}}_{\xi, \rho}^{m_2}(H)(u)\big\|_2 \leq \tilde{\mathcal{L}}_a^{m_2}\big(\beta_J^H h\big)(u)$ follows trivially from definitions. Otherwise, by construction, there exist $(k, r), (k_0, r_0) \in \Xi_1(b, \rho)$ such that $u \in D_{k, r, p_{k, r}} \in [D_{k_0, r_0, p_{k_0, r_0}}]$ corresponding to $(k, r, p_{k, r}, l_{k, r}) \in J$, and such that $j_{k', r'} = j_{k_0, r_0}$ and $l_{k', r'} = l_{k_0, r_0}$ for all $D_{k', r', p_{k', r'}} \in [D_{k_0, r_0, p_{k_0, r_0}}]$. Denote $j_{k_0, r_0}$ by $j_0$ and $l_{k_0, r_0}$ by $l_0$. Let $(\ell, \ell') = (0, j_0)$ if $l_0 = 1$ and $(\ell, \ell') = (j_0, 0)$ if $l_0 = 2$. Then by construction of $J$, we have $\chi_{l_0}^{j_0}[\xi, \rho, H, h](u) \leq 1$, $\beta_J^H(v_\ell(u)) \geq 1 - N\mu$, and $\beta_J^H(v_j(u)) = 1$ for all $0 \leq j \leq j_{\mathrm{m}}$ with $j \neq \ell$. Hence, we compute that
\begin{align*}
&\big\|\tilde{\mathcal{M}}_{\xi, \rho}^{m_2}(H)(u)\big\|_2 \\
={}&\left\|\sum_{\substack{\alpha: \len(\alpha) = m_2\\ v = \sigma^{-\alpha}(u)}} e^{f_\alpha^{(a)}(v)} \rho_b(\Phi^\alpha(v)^{-1}) H(v)\right\|_2 \\
\leq{}&\sum_{\substack{\alpha: \len(\alpha) = m_2\\ v = \sigma^{-\alpha}(u) \notin \{v_0(u), v_{j_0}(u)\}}} \left\|e^{f_\alpha^{(a)}(v)} \rho_b(\Phi^\alpha(v)^{-1}) H(v)\right\|_2 \\
{}&+ \left\|e^{f_{\alpha_\ell}^{(a)}(v_\ell(u))} \rho_b(\Phi^{\alpha_\ell}(v_\ell(u))^{-1}) H(v_\ell(u))\right. \\
&{}+ \left.e^{f_{\alpha_{\ell'}}^{(a)}(v_{\ell'}(u))} \rho_b(\Phi^{\alpha_{\ell'}}(v_{\ell'}(u))^{-1}) H(v_{\ell'}(u))\right\|_2 \\
\leq{}&\sum_{\substack{\alpha: \len(\alpha) = m_2\\ v = \sigma^{-\alpha}(u) \notin \{v_0(u), v_{j_0}(u)\}}} e^{f_{\alpha}^{(a)}(v)} h(v) + \left((1 - N\mu)e^{f_{\alpha_\ell}^{(a)}(v_\ell(u))}h(v_\ell(u))\right. \\
&{}+ \left.e^{f_{\alpha_{\ell'}}^{(a)}(v_{\ell'}(u))}h(v_{\ell'}(u))\right) \\
\leq{}&\tilde{\mathcal{L}}_a^{m_2}\big(\beta_J^H h\big)(u).
\end{align*}
Thus, we have
\begin{align*}
\big\|\tilde{\mathcal{M}}_{\xi, \rho}^m(H)(u)\big\|_2 &\leq \left\|\Big(\tilde{\mathcal{M}}_{\xi, \rho}^{m_1} \circ \tilde{\mathcal{M}}_{\xi, \rho}^{m_2}\Big)(H)(u)\right\|_2 \leq \tilde{\mathcal{L}}_a^{m_1}\big\|\tilde{\mathcal{M}}_{\xi, \rho}^{m_2}(H)\big\|(u) \\
&\leq \tilde{\mathcal{L}}_a^{m_1}\Big(\tilde{\mathcal{L}}_a^{m_2}\big(\beta_J^H h\big)\Big)(u) = \tilde{\mathcal{L}}_a^m\big(\beta_J^H h\big)(u) = \mathcal{N}_{a, J}^H(h)(u)
\end{align*}
for all $u \in \tilde{U}$.
\end{proof}

\section{Exponential mixing of the frame flow}
\label{sec:ExponentialMixingOfTheFrameFlow}
The aim of this section is to prove \cref{thm:TheoremExponentialMixingOfFrameFlow} using the proven spectral bounds in \cref{thm:TheoremFrameFlow}. We will use techniques originally due to Pollicott \cite{Pol85} to write the correlation function in terms of transfer operators with holonomy and then apply Paley--Wiener theory.

Similar to $R^{\vartheta, \tau}$, consider the suspension space $U^{\vartheta, \tau} = (U \times M \times \mathbb R_{\geq 0})/\mathord{\sim}$ where $\sim$ is the equivalence relation on $U \times M \times \mathbb R_{\geq 0}$ defined by $(u, m, t + \tau(u)) \sim (\sigma(u), \vartheta(u)^{-1}m, t)$ for all $(u, m, t) \in U \times M \times \mathbb R_{\geq 0}$. Also consider $\tilde{U}^{\vartheta, \tau} = \{(u, m, t) \in \tilde{U} \times M \times \mathbb R_{\geq 0}: t \in [0, \tau(u))\}$. For simplicity, we say that $\tilde{\phi} \in C^1(\tilde{U}^{\vartheta, \tau}, \mathbb R)$ is an \emph{extension} of $\phi \in C(U^{\vartheta, \tau}, \mathbb R)$ whenever $\tilde{\phi}(u, m, t) = \phi(u, m, t)$ for all $t \in [0, \tau(u))$, $m \in M$, and $u \in U$.

Let $\phi \in C(U^{\vartheta, \tau}, \mathbb R)$ and $\xi \in \mathbb C$. Define $\hat{\phi}_\xi \in B(U, L^2(M, \mathbb C))$ by
\begin{align*}
\hat{\phi}_\xi(u)(m) = \int_0^{\tau(u)} \phi(u, m, t)e^{-\xi t} \, dt \qquad \text{for all $m \in M$ and $u \in U$}.
\end{align*}
We can decompose it further as $\hat{\phi}_\xi(u) = \sum_{\rho \in \widehat{M}} \hat{\phi}_{\xi, \rho}(u) \in \operatorname*{\widehat{\bigoplus}}_{\rho \in \widehat{M}} V_\rho^{\oplus \dim(\rho)}$ for all $u \in U$. Let $\rho \in \widehat{M}$. Defining $\phi_\rho \in C(U^{\vartheta, \tau}, \mathbb C)$ by the projection $\phi_\rho(u, \cdot, t) = [\phi(u, \cdot, t)]_\rho \in V_\rho^{\oplus \dim(\rho)}$ for all $u \in U$ and $t \in \mathbb R_{\geq 0}$, we have
\begin{align*}
\hat{\phi}_{\xi, \rho}(u)(m) = \int_0^{\tau(u)} \phi_\rho(u, m, t) e^{-\xi t} \, dt \qquad \text{for all $m \in M$ and $u \in U$}.
\end{align*}

\begin{remark}
Let $\phi \in C(U^{\vartheta, \tau}, \mathbb R)$ and $\xi \in \mathbb C$. Because of $\tau$ involved in the definition of $\hat{\phi}_{\xi, \rho}$, it is not Lipschitz. However, in \cref{lem:ExtractNormOfLaplaceTransformDecay} we will see that $\mathcal{M}_{\xi, \rho}\big(\hat{\phi}_{-\overline{\xi}, \rho}\big) \in \mathcal{V}_\rho(U)$ with an extension $\mathcal{M}_{\xi, \rho}\big(\hat{\phi}_{-\overline{\xi}, \rho}\big)^\sim \in \mathcal{V}_\rho(\tilde{U})$.
\end{remark}

\subsection{Correlation function and its Laplace transform}
Let $\phi, \psi \in C(U^{\vartheta, \tau}, \mathbb R)$. Define $\Upsilon_{\phi, \psi} \in L^\infty(\mathbb R_{\geq 0}, \mathbb R)$ by
\begin{align*}
\Upsilon_{\phi, \psi}(t) = \int_U \int_M \int_0^{\tau(u)} \phi(u, m, r + t) \psi(u, m, r) \, dr \, dm \, d\nu_U(u)
\end{align*}
for all $t \in \mathbb R_{\geq 0}$. We can decompose this into $\Upsilon_{\phi, \psi} = \Upsilon_{\phi, \psi}^0 + \Upsilon_{\phi, \psi}^1$ where we define
\begin{align*}
\Upsilon_{\phi, \psi}^0(t) &= \int_U \int_M \int_{\max(0, \tau(u) - t)}^{\tau(u)} \phi(u, m, r + t) \psi(u, m, r) \, dr \, dm \, d\nu_U(u); \\
\Upsilon_{\phi, \psi}^1(t) &= \int_U \int_M \int_0^{\max(0, \tau(u) - t)} \phi(u, m, r + t) \psi(u, m, r) \, dr \, dm \, d\nu_U(u)
\end{align*}
for all $t \in \mathbb R_{\geq 0}$. Recall that the Laplace transform $\hat{\Upsilon}_{\phi, \psi}^0: \{\xi \in \mathbb C: \Re(\xi) > 0\} \to \mathbb C$ is given by
\begin{align*}
\hat{\Upsilon}_{\phi, \psi}^0(\xi) = \int_0^\infty \Upsilon_{\phi, \psi}^0(t) e^{-\xi t} \, dt \qquad \text{for all $\xi \in \mathbb C$ with $a > 0$}.
\end{align*}
The above decomposition is useful because of \cref{lem:LaplaceTransformOfSymbolicCodingCorrelationFunctionInTermsOfTransferOperatorFrameFlow} while $\Upsilon_{\phi, \psi}(t) = \Upsilon_{\phi, \psi}^0(t)$ for all $t \geq \overline{\tau}$. The proof of \cref{lem:LaplaceTransformOfSymbolicCodingCorrelationFunctionInTermsOfTransferOperatorFrameFlow} is similar to \cite[Lemma 5.2]{OW16}.

\begin{lemma}
\label{lem:LaplaceTransformOfSymbolicCodingCorrelationFunctionInTermsOfTransferOperatorFrameFlow}
For all $\phi, \psi \in C(U^{\vartheta, \tau}, \mathbb R)$ and $\xi \in \mathbb C$ with $a > 0$, we have
\begin{align*}
\hat{\Upsilon}_{\phi, \psi}^0(\xi) = \sum_{k = 1}^\infty \sum_{\rho \in \widehat{M}} \lambda_a^k \left\langle \hat{\phi}_{\xi, \rho}, \mathcal{M}_{\xi, \rho}^k\big(\hat{\psi}_{-\overline{\xi}, \rho}\big) \right\rangle.
\end{align*}
\end{lemma}

\subsection{Exponential decay of the correlation function}
\begin{lemma}
\label{lem:ExtractNormOfLaplaceTransformDecay}
There exists $C > 0$ such that for all $\rho \in \widehat{M}$, and $\phi, \psi \in C(U^{\vartheta, \tau}, \mathbb R)$ with some extensions $\tilde{\phi} \in C^{r + 1}(\tilde{U}^{\vartheta, \tau}, \mathbb R)$ for some $r \in \mathbb Z_{\geq 0}$ and $\tilde{\psi} \in C^1(\tilde{U}^{\vartheta, \tau}, \mathbb R)$, and $\xi \in \mathbb C$ with $|a| \leq a_0'$, we have that $\mathcal{M}_{\xi, \rho}\big(\hat{\psi}_{-\overline{\xi}, \rho}\big) \in \mathcal{V}_\rho(U)$ has an extension $\mathcal{M}_{\xi, \rho}\big(\hat{\psi}_{-\overline{\xi}, \rho}\big)^\sim \in \mathcal{V}_\rho(\tilde{U})$ and
\begin{align*}
\sup_{u \in U} \big\|\hat{\phi}_\xi(u)\big\|_{C^r} &\leq C \frac{\|\tilde{\phi}\|_{C^{r + 1}}}{\max(1, |b|)}; & \big\|\mathcal{M}_{\xi, \rho}\big(\hat{\psi}_{-\overline{\xi}, \rho}\big)^\sim\big\|_{1, \|\rho_b\|} &\leq C \frac{\|\tilde{\psi}\|_{C^1}}{\max(1, |b|)}.
\end{align*}
\end{lemma}

\begin{proof}
Fix $C_1 = (2 + \overline{\tau})e^{a_0'\overline{\tau}}$ and $C = Ne^{T_0}C_1 + \frac{4NT_0e^{T_0}C_1}{c_0 \kappa_2} + \frac{Ne^{T_0}e^{a_0'\overline{\tau}}(\overline{\tau} + T_0)}{c_0 \kappa_2}$. Let $\rho \in \widehat{M}$, and $\phi, \psi \in C(U^{\vartheta, \tau}, \mathbb R)$ with some extensions $\tilde{\phi} \in C^{r + 1}(\tilde{U}^{\vartheta, \tau}, \mathbb R)$ for some $r \in \mathbb Z_{\geq 0}$ and $\tilde{\psi} \in C^1(\tilde{U}^{\vartheta, \tau}, \mathbb R)$, and $\xi \in \mathbb C$ with $|a| \leq a_0'$. We show the first inequality. If $|b| \leq 1$, from the definition of $\hat{\phi}_\xi(u)$ we have
\begin{align*}
\sup_{u \in U} \big\|\hat{\phi}_\xi(u)\big\|_{C^r} \leq \overline{\tau}e^{a_0'\overline{\tau}}\|\tilde{\phi}\|_{C^r} \leq C_1 \|\tilde{\phi}\|_{C^{r + 1}}.
\end{align*}
If $|b| \geq 1$, integrating by parts gives
\begin{align*}
\nabla^k \hat{\phi}_\xi(u) ={}&\int_0^{\tau(u)} \nabla^k \phi(u, \cdot, t)e^{-\xi t} \, dt \\
={}&\left[-\frac{1}{\xi} \nabla^k \phi(u, \cdot, t)e^{-\xi t}\right]_{t = 0}^{t \nearrow \tau(u)} + \frac{1}{\xi}\int_0^{\tau(u)} \left.\frac{d}{dt'}\right|_{t' = t} \nabla^k \phi(u, \cdot, t') \cdot e^{-\xi t} \, dt
\end{align*}
for all $u \in U$ and $0 \leq k \leq r$. Hence
\begin{align*}
\sup_{u \in U} \big\|\hat{\phi}_\xi(u)\big\|_{C^r} &\leq \left(\frac{2}{|b|}\|\tilde{\phi}\|_{C^r} e^{a_0'\overline{\tau}} + \frac{\overline{\tau}}{|b|}\|\tilde{\phi}\|_{C^{r + 1}} e^{a_0'\overline{\tau}}\right) \leq C_1 \frac{\|\tilde{\phi}\|_{C^{r + 1}}}{|b|}.
\end{align*}

Now we show the second inequality. For all admissible pairs $(j, k)$, define $\hat{\psi}_{-\overline{\xi}, \rho}^{(j, k)} \in C^1\bigl(\tilde{\mathtt{C}}[j, k], V_\rho^{\oplus \dim(\rho)}\bigr)$ by
\begin{align*}
\hat{\psi}_{-\overline{\xi}, \rho}^{(j, k)}(u)(m) = \int_0^{\tau_{(j, k)}(u)} \tilde{\psi}_\rho(u, m, t)e^{\overline{\xi} t} \, dt \qquad \text{for all $m \in M$ and $u \in \tilde{\mathtt{C}}[j, k]$}.
\end{align*}
For all admissible pairs $(j, k)$, define $\hat{\psi}_{-\overline{\xi}}^{(j, k)} \in C^1(\tilde{\mathtt{C}}[j, k], L^2(M, \mathbb C))$ in a similar fashion. Then, $\hat{\psi}_{-\overline{\xi}, \rho}^{(j, k)}$ and $\hat{\psi}_{-\overline{\xi}}^{(j, k)}$ are extensions of $\hat{\psi}_{-\overline{\xi}, \rho}\big|_{\mathtt{C}[j, k]}$ and $\hat{\psi}_{-\overline{\xi}}\big|_{\mathtt{C}[j, k]}$ respectively, for all admissible pairs $(j, k)$. Define $\mathcal{M}_{\xi, \rho}\big(\hat{\psi}_{-\overline{\xi}, \rho}\big)^\sim \in \mathcal{V}_\rho(\tilde{U})$ by
\begin{align*}
\mathcal{M}_{\xi, \rho}\big(\hat{\psi}_{-\overline{\xi}, \rho}\big)^\sim(u) = \sum_{\substack{(j, k)\\ u' = \sigma^{-(j, k)}(u)}} e^{f_{(j, k)}^{(a)}(u')} \rho_b(\Phi^{(j, k)}(u')^{-1}) \hat{\psi}_{-\overline{\xi}, \rho}^{(j, k)}(u')
\end{align*}
for all $u \in \tilde{U}$, which is then an extension of $\mathcal{M}_{\xi, \rho}\big(\hat{\psi}_{-\overline{\xi}, \rho}\big)$. Now, we first bound the $L^\infty$ norm. Using similar estimates as for the first proven inequality, we have
\begin{align}
\label{eqn:L2BoundFor_psi_hat}
\left\|\hat{\psi}_{-\overline{\xi}, \rho}^{(j, k)}(u)\right\|_2 \leq \left\|\hat{\psi}_{-\overline{\xi}}^{(j, k)}(u)\right\|_2 \leq \left\|\hat{\psi}_{-\overline{\xi}}^{(j, k)}(u)\right\|_\infty \leq C_1 \frac{\|\tilde{\psi}\|_{C^1}}{\max(1, |b|)}
\end{align}
for all $u \in \tilde{\mathtt{C}}[j, k]$, and admissible pairs $(j, k)$. So, by unitarity of $\rho_b$, we have
\begin{align*}
\big\|\mathcal{M}_{\xi, \rho}\big(\hat{\psi}_{-\overline{\xi}, \rho}\big)\big\|_\infty \leq Ne^{T_0}C_1 \frac{\|\tilde{\psi}\|_{C^1}}{\max(1, |b|)}.
\end{align*}
Now, we deal with the $C^1$ norm. Let $u \in \tilde{U}$ and $z \in \T_u(\tilde{U})$ with $\|z\| = 1$. We have a similar formula for $\big(d\mathcal{M}_{\xi, \rho}\big(\hat{\psi}_{-\overline{\xi}, \rho}\big)^\sim\big)_u(z)$ as in \cref{eqn:DifferentialOfCongruenceTransferOperatorOfType_rho} except that the summations are over admissible pairs $(j, k)$ and $H$ is replaced by $\hat{\psi}_{-\overline{\xi}, \rho}^{(j, k)}$. We use the same notation $K_1$, $-K_2$, and $K_3$ for the terms. Using \cref{eqn:L2BoundFor_psi_hat}, the first two terms can be bounded as
\begin{align*}
\|K_1\|_2 &\leq \frac{NT_0e^{T_0}}{c_0\kappa_2} \cdot C_1 \frac{\|\tilde{\psi}\|_{C^1}}{\max(1, |b|)}; & \|K_2\|_2 &\leq \frac{2NT_0e^{T_0}}{c_0\kappa_2} \|\rho_b\| \cdot C_1 \frac{\|\tilde{\psi}\|_{C^1}}{\max(1, |b|)}.
\end{align*}
Now we bound the third term. First, we have
\begin{align*}
d\left(\hat{\psi}_{-\overline{\xi}}^{(j, k)}(\cdot)(m)\right)_u ={}&\int_0^{\tau_{(j, k)}(u)} d\big(\tilde{\psi}(\cdot, m, t)\big)_u \cdot e^{\overline{\xi} t} \, dt \\
&{}+ \tilde{\psi}(u, m, \tau_{(j, k)}(u))e^{\overline{\xi} \tau_{(j, k)}(u)} \cdot \big(d\tau_{(j, k)}\big)_u
\end{align*}
for all $m \in M$ and $u \in \tilde{\mathtt{C}}[j, k]$. Thus $\left|\hat{\psi}_{-\overline{\xi}}^{(j, k)}\right|_{C^1} \leq e^{a_0'\overline{\tau}}(\overline{\tau} + T_0)\|\tilde{\psi}\|_{C^1}$ and so
\begin{align*}
\|K_3\|_2 \leq \frac{Ne^{T_0}}{c_0\kappa_2} \sup_{(j, k)} \left|\hat{\psi}_{-\overline{\xi}}^{(j, k)}\right|_{C^1} \leq \frac{Ne^{T_0}e^{a_0'\overline{\tau}}(\overline{\tau} + T_0)}{c_0\kappa_2} \|\tilde{\psi}\|_{C^1}.
\end{align*}
Using definitions and $\frac{1 + \|\rho_b\|}{\max(1, \|\rho_b\|)} \leq 2$ and $\frac{1}{\max(1, \|\rho_b\|)} \leq \frac{1}{\max(1, |b|)}$, we have
\begin{align*}
&\big\|\mathcal{M}_{\xi, \rho}\big(\hat{\psi}_{-\overline{\xi}, \rho}\big)^\sim\big\|_{1, \|\rho_b\|} \\
\leq{}&Ne^{T_0}C_1 \frac{\|\tilde{\psi}\|_{C^1}}{\max(1, |b|)} + \frac{2NT_0e^{T_0}}{c_0\kappa_2} \cdot \frac{1 + \|\rho_b\|}{\max(1, \|\rho_b\|)} \cdot C_1 \frac{\|\tilde{\psi}\|_{C^1}}{\max(1, |b|)} \\
&{}+ \frac{Ne^{T_0}e^{a_0'\overline{\tau}}(\overline{\tau} + T_0)}{c_0\kappa_2} \cdot \frac{\|\tilde{\psi}\|_{C^1}}{\max(1, \|\rho_b\|)} \\
\leq{}&C \frac{\|\tilde{\psi}\|_{C^1}}{\max(1, |b|)}.
\end{align*}
\end{proof}

\begin{remark}
Unlike the geodesic flow case, in the frame flow case, we have to correctly estimate $L^2$ norms and also take care of its convergence over all $\rho \in \widehat{M}$ in \cref{lem:DecayOfSymbolicCodingCorrelationFunction}. However, this is not a problem due to \cref{lem:ExtractNormOfLaplaceTransformDecay} and \cite[Lemmas 4.4.2.2 and 4.4.2.3]{War72}. For all $\rho \in \widehat{M}$, the number $\lambda_{1 + \varsigma(\rho)} > 1$ which appear in \cref{lem:DecayOfSymbolicCodingCorrelationFunction} is in fact the eigenvalue of $1 + \varsigma(\rho) \in Z(\mathfrak{m})$ where $\varsigma(\rho)$ is the negative Casimir operator as defined in \cref{lem:maActionLowerBound}. \cite[Lemma 4.4.2.3]{War72} states that $\sum_{\rho \in \widehat{M}} \frac{\dim(\rho)^2}{(\lambda_{1 + \varsigma(\rho)})^s} < \infty$ for some $s \in \mathbb N$ which is essentially a direct consequence of the Weyl dimension formula.
\end{remark}

\begin{lemma}
\label{lem:DecayOfSymbolicCodingCorrelationFunction}
There exist $C > 0$, $\eta > 0$, and $r \in \mathbb N$ such that for all $\phi, \psi \in C(U^{\vartheta, \tau}, \mathbb R)$ with some extensions $\tilde{\phi} \in C^{r + 1}(\tilde{U}^{\vartheta, \tau}, \mathbb R)$ and $\tilde{\psi} \in C^1(\tilde{U}^{\vartheta, \tau}, \mathbb R)$ such that $\int_M \tilde{\psi}(u, m, t) \, dm = 0$ for all $(u, t) \in \tilde{U}^\tau$, for all $t > 0$, we have
\begin{align*}
|\Upsilon_{\phi, \psi}(t)| \leq C e^{-\eta t} \|\tilde{\phi}\|_{C^{r + 1}} \|\tilde{\psi}\|_{C^1}.
\end{align*}
\end{lemma}

\begin{proof}
Fix $C_1 \geq 1$, $\tilde{\eta} > 0$, and $a_0'' > 0$ be the $C$, $\eta$, and $a_0$ from \cref{thm:TheoremFrameFlow} and $C_2 > 0$ be the $C$ from \cref{lem:ExtractNormOfLaplaceTransformDecay}. Fix $\eta = a_0 \in \left(0, \frac{1}{2}\min(a_0', a_0'')\right)$ such that $\sup_{|a| \leq 2a_0} \log(\lambda_a) \leq \frac{\tilde{\eta}}{2}$. Fix $C_3 = 2e^{\tilde{\eta}}C_1C_2^2$. By \cite[Lemma 4.4.2.3]{War72}, there exists $s \in \mathbb N$ such that $\sum_{\rho \in \widehat{M}} \frac{\dim(\rho)^2}{(\lambda_{1 + \varsigma(\rho)})^s} < \infty$ where $\lambda_{1 + \varsigma(\rho)} > 1$ are constants in the lemma corresponding to each $\rho \in \widehat{M}$. Fix $r = 2s$, $C_4 = C_3\sum_{\rho \in \widehat{M}} \frac{\dim(\rho)^2}{(\lambda_{1 + \varsigma(\rho)})^s}$, and $C = \max\left(C_4\sum_{k = 1}^\infty e^{-\frac{\tilde{\eta}}{2}k}, 2\overline{\tau}e^{\eta \overline{\tau}}\right)$. Let $\phi, \psi \in C(U^{\vartheta, \tau}, \mathbb R)$ with some extensions $\tilde{\phi} \in C^{r + 1}(\tilde{U}^{\vartheta, \tau}, \mathbb R)$ and $\tilde{\psi} \in C^1(\tilde{U}^{\vartheta, \tau}, \mathbb R)$ such that $\int_M \tilde{\psi}(u, m, t) \, dm = 0$ for all $(u, t) \in \tilde{U}^\tau$. By \cref{lem:LaplaceTransformOfSymbolicCodingCorrelationFunctionInTermsOfTransferOperatorFrameFlow}, we have
\begin{align*}
\hat{\Upsilon}_{\phi, \psi}^0(\xi) = \sum_{k = 1}^\infty \sum_{\rho \in \widehat{M}_0} \lambda_a^k \left\langle \hat{\phi}_{\xi, \rho}, \mathcal{M}_{\xi, \rho}^k\big(\hat{\psi}_{-\overline{\xi}, \rho}\big) \right\rangle \qquad \text{for all $\xi \in \mathbb C$ with $a > 0$}.
\end{align*}
Note that for all $k \geq 1$ and $\rho \in \widehat{M}_0$, the map $\xi \mapsto \lambda_a^k \left\langle \hat{\phi}_{\xi, \rho}, \mathcal{M}_{\xi, \rho}^k\big(\hat{\psi}_{-\overline{\xi}, \rho}\big) \right\rangle$ is entire. Hence, to show that $\hat{\Upsilon}_{\phi, \psi}^0$ has a holomorphic extension to the half plane $\{\xi \in \mathbb C: \Re(\xi) > -2a_0\}$, it suffices to show that the above sum is absolutely convergent for all $\xi \in \mathbb C$ with $|a| < 2a_0$. Recall that $\mathcal{M}_{\xi, \rho}\big(\hat{\psi}_{-\overline{\xi}, \rho}\big) \in \mathcal{V}_\rho(U)$ and by \cref{lem:ExtractNormOfLaplaceTransformDecay} there is an extension $\mathcal{M}_{\xi, \rho}\big(\hat{\psi}_{-\overline{\xi}, \rho}\big)^\sim \in \mathcal{V}_\rho(\tilde{U})$ for all $\rho \in \widehat{M}_0$. Using \cite[Lemma 4.4.2.2]{War72}, we first calculate
\begin{align*}
\big\|\hat{\phi}_{\xi, \rho}\big\|_2 &\leq \left(\int_U \big\|\hat{\phi}_{\xi, \rho}(u)\big\|_2^2 \, d\nu_U(u)\right)^{\frac{1}{2}} \leq \left(\int_U \big\|\hat{\phi}_{\xi, \rho}(u)\big\|_\infty^2 \, d\nu_U(u)\right)^{\frac{1}{2}} \\
&\leq \left(\int_U \frac{\dim(\rho)^4}{(\lambda_{1 + \varsigma(\rho)})^{2s}}\big\|\hat{\phi}_\xi(u)\big\|_{C^r}^2 \, d\nu_U(u)\right)^{\frac{1}{2}} \leq \frac{\dim(\rho)^2}{(\lambda_{1 + \varsigma(\rho)})^s} \sup_{u \in U} \big\|\hat{\phi}_\xi(u)\big\|_{C^r}.
\end{align*}
Also noting $\frac{1}{\max(1, |b|)^2} \leq \frac{2}{1 + b^2}$, we use \cref{thm:TheoremFrameFlow,lem:ExtractNormOfLaplaceTransformDecay} to get
\begin{align*}
&\left|\lambda_a^k \left\langle \hat{\phi}_{\xi, \rho}, \mathcal{M}_{\xi, \rho}^k\big(\hat{\psi}_{-\overline{\xi}, \rho}\big) \right\rangle\right| \\
\leq{}&\lambda_a^k \big\|\hat{\phi}_{\xi, \rho}\big\|_2 \cdot \big\|\tilde{\mathcal{M}}_{\xi, \rho}^{k - 1}\big(\mathcal{M}_{\xi, \rho}\big(\hat{\psi}_{-\overline{\xi}, \rho}\big)^\sim\big)\big\|_2 \\
\leq{}&\lambda_a^k \frac{\dim(\rho)^2}{(\lambda_{1 + \varsigma(\rho)})^s} \sup_{u \in U} \big\|\hat{\phi}_\xi(u)\big\|_{C^r} \cdot C_1e^{-\tilde{\eta}(k - 1)} \big\|\mathcal{M}_{\xi, \rho}\big(\hat{\psi}_{-\overline{\xi}, \rho}\big)^\sim\big\|_{1, \|\rho_b\|} \\
\leq{}&\lambda_a^k \frac{\dim(\rho)^2}{(\lambda_{1 + \varsigma(\rho)})^s} C_2\frac{\|\tilde{\phi}\|_{C^{r + 1}}}{\max(1, |b|)} \cdot C_1e^{-\tilde{\eta}(k - 1)} \cdot C_2\frac{\|\tilde{\psi}\|_{C^1}}{\max(1, |b|)} \\
\leq{}&\frac{C_3e^{-\frac{\tilde{\eta}}{2}k}}{1 + b^2} \cdot  \frac{\dim(\rho)^2}{(\lambda_{1 + \varsigma(\rho)})^s} \|\tilde{\phi}\|_{C^{r + 1}} \|\tilde{\psi}\|_{C^1}.
\end{align*}
Now summing over all $\rho \in \widehat{M}_0$, we have
\begin{align*}
\sum_{\rho \in \widehat{M}_0} \left|\lambda_a^k \left\langle \hat{\phi}_{\xi, \rho}, \mathcal{M}_{\xi, \rho}^k\big(\hat{\psi}_{-\overline{\xi}, \rho}\big) \right\rangle\right| &\leq \frac{C_3e^{-\frac{\tilde{\eta}}{2}k}}{1 + b^2} \|\tilde{\phi}\|_{C^{r + 1}} \|\tilde{\psi}\|_{C^1} \sum_{\rho \in \widehat{M}_0} \frac{\dim(\rho)^2}{(\lambda_{1 + \varsigma(\rho)})^s} \\
&\leq \frac{C_4e^{-\frac{\tilde{\eta}}{2}k}}{1 + b^2} \|\tilde{\phi}\|_{C^{r + 1}} \|\tilde{\psi}\|_{C^1}
\end{align*}
for all $\xi \in \mathbb C$ with $|a| < 2a_0$, whose sum over $k \geq 1$ converges as desired. The above calculation also gives the important bound $|\hat{\Upsilon}_{\phi, \psi}^0(\xi)| \leq \frac{C}{1 + b^2} \|\tilde{\phi}\|_{C^{r + 1}} \|\tilde{\psi}\|_{C^1}$ for all $\xi \in \mathbb C$ with $|a| < 2a_0$. Since $\Upsilon_{\phi, \psi}^0$ is continuous and in $L^\infty(\mathbb R_{\geq 0}, \mathbb R)$, we use the holomorphic extension and the inverse Laplace transform formula along the line $\{\xi \in \mathbb C: \Re(\xi) = -a_0\}$ to obtain
\begin{align*}
\Upsilon_{\phi, \psi}^0(t) = \frac{1}{2\pi i} \lim_{B \to \infty} \int_{-a_0 - iB}^{-a_0 + iB} \hat{\Upsilon}_{\phi, \psi}^0(\xi) e^{\xi t} \, d\xi = \frac{1}{2\pi} \int_{-\infty}^\infty \hat{\Upsilon}_{\phi, \psi}^0(-a_0 + ib) e^{(-a_0 + ib) t} \, db
\end{align*}
for all $t > 0$. Using the above bound, we have
\begin{align*}
\left|\Upsilon_{\phi, \psi}^0(t)\right| &\leq \frac{1}{2\pi} e^{-a_0 t} \int_{-\infty}^\infty \bigl|\hat{\Upsilon}_{\phi, \psi}^0(-a_0 + ib)\bigr| \, db \\
&\leq \frac{1}{2\pi} e^{-a_0 t} \int_{-\infty}^\infty \frac{C}{1 + b^2} \|\tilde{\phi}\|_{C^{r + 1}} \|\tilde{\psi}\|_{C^1} \, db \\
&\leq \frac{C}{2} e^{-\eta t} \|\tilde{\phi}\|_{C^{r + 1}} \|\tilde{\psi}\|_{C^1}
\end{align*}
for all $t > 0$. Now, $\Upsilon_{\phi, \psi}(t) = \Upsilon_{\phi, \psi}^0(t)$ for all $t \geq \overline{\tau}$ while
\begin{align*}
\left|\Upsilon_{\phi, \psi}^1(t)\right| \leq \overline{\tau} \|\tilde{\phi}\|_{C^{r + 1}} \|\tilde{\psi}\|_{C^1} \leq \frac{C}{2} e^{-\eta t} \|\tilde{\phi}\|_{C^{r + 1}} \|\tilde{\psi}\|_{C^1}
\end{align*}
for all $t \in [0, \overline{\tau}]$ and hence $\left|\Upsilon_{\phi, \psi}(t)\right| \leq C e^{-\eta t} \|\tilde{\phi}\|_{C^{r + 1}} \|\tilde{\psi}\|_{C^1}$.
\end{proof}

\subsection{Integrating out the strong stable direction and the proof of \texorpdfstring{\cref{thm:TheoremExponentialMixingOfFrameFlow}}{\autoref{thm:TheoremExponentialMixingOfFrameFlow}}}
Given a $\phi \in C^1(\Gamma \backslash G, \mathbb R)$, we can convert it to a function in $C(U^{\vartheta, \tau}, \mathbb R)$. By Rokhlin's disintegration theorem with respect to the projection $\proj_U: R \to U$, the probability measure $\nu_R$ disintegrates to give the set of conditional probability measures $\{\nu_u: u \in U\}$. For all $j \in \mathcal{A}$ and $u \in U_j$, the measure $\nu_u$ is actually define on the fiber $\proj_U^{-1}(u) = [u, S_j]$ but we can push forward via the map $[u, S_j] \to S_j$ defined by $[u, s] \mapsto s$ to think of $\nu_u$ as a measure on $S_j$. For all $t \geq 0$, we define $\phi_t \in C(U^{\vartheta, \tau}, \mathbb R)$ by
\begin{align*}
\phi_t(u, m, r) = \int_{S_j} \phi(F([u, s])a_{t + r}m) \, d\nu_u(s)
\end{align*}
for all $r \in [0, \tau(u))$, $m \in M$, $u \in U_j$, and $j \in \mathcal{A}$, and in order to ensure that indeed $\phi_t \in C(U^{\vartheta, \tau}, \mathbb R)$, we must define $\phi_t(u, m, r) = \phi_t(\sigma^k(u), \vartheta^k(u)^{-1}m, r - \tau_k(u))$ for all $r \in [\tau_k(u), \tau_{k + 1}(u))$ and $k \geq 0$.

Let $\phi \in C^k(\Gamma \backslash G, \mathbb R)$ for some $k \in \mathbb N$ and $t > 0$. Then $\phi_t \in C(U^{\vartheta, \tau}, \mathbb R)$ has a natural extension $\tilde{\phi}_t \in C^k(\tilde{U}^{\vartheta, \tau}, \mathbb R)$ defined by
\begin{align*}
\tilde{\phi}_t(u, m, r) = \int_{S_j} \phi(F([u, s])a_{t + r}m) \, d\nu_u(s)
\end{align*}
for all $r \in [0, \tau(u))$, $m \in M$, $u \in U_j$, and $j \in \mathcal{A}$, where we clarify the notation $\nu_u$ in the following remark. This justifies using \cref{lem:EstimateFunctionByIntegratingOverS} later.

\begin{remark}
We need to deal with some technicalities. Let $j \in \mathcal{A}$. By smoothness of the strong unstable and strong stable foliations, there exists $C_1 > 1$ such that $d([u, s], [u', s]) \leq C_1d(u, u')$ for all $u, u' \in \tilde{U}_j$ and $s \in S_j$. Now, for all $u \in \tilde{U}_j$, the Patterson--Sullivan density induces the measure $d\mu^{\mathrm{PS}}_{[u, S_j]}([u, s]) = e^{\delta_\Gamma \beta_{[\tilde{u}, \tilde{s}]^-}(o, [\tilde{u}, \tilde{s}])} \, d\mu^{\mathrm{PS}}_o([\tilde{u}, \tilde{s}]^-)$ on $[u, S_j]$ and pushing forward via the map $[u, S_j] \to S_j$ mentioned above gives the measure $\mu^{\mathrm{PS}}_u(s) = e^{\delta_\Gamma \beta_{(\tilde{s})^-}(o, [\tilde{u}, \tilde{s}])} d\mu^{\mathrm{PS}}_o((\tilde{s})^-)$ on $S_j$, using $[\tilde{u}, \tilde{s}]^- = (\tilde{s})^-$. In fact, from the definition of the BMS measure, we have $\nu_u = \frac{\mu^{\mathrm{PS}}_u}{\mu^{\mathrm{PS}}_u(S_j)}$ for all $u \in U_j$. Hence, we use this as the definition of $\nu_u$ for any $u \in \tilde{U}_j$. The nontrivial consequence is that $\frac{d\nu_{u'}}{d\nu_u}(s) \propto e^{\delta_\Gamma \beta_{(\tilde{s})^-}([\tilde{u}', \tilde{s}], [\tilde{u}, \tilde{s}])}$ for all $s \in S_j$ and smooth in $u, u' \in \tilde{U}_j$. Hence, there exists $C_2 > 0$ such that $\left|1 - \frac{d\nu_{u'}}{d\nu_u}(s)\right| \leq C_2d(u, u')$ for all $u, u' \in \tilde{U}_j$ and $s \in S_j$. An easy computation using the two derived inequalities shows that $\|\tilde{\phi}_t\|_{C^1} \leq C\|\phi\|_{C^1}$ for some $C > 1$, for all $\phi \in C_{\mathrm{c}}^1(\Gamma \backslash G, \mathbb R)$ and $t \geq 0$.
\end{remark}

\begin{lemma}
\label{lem:EstimateFunctionByIntegratingOverS}
There exist $C > 0$ and $\eta > 0$ such that for all $\phi \in C_{\mathrm{c}}^1(\Gamma \backslash G, \mathbb R)$, we have
\begin{align*}
|\phi(F([u, s])a_{2t + r}m) - \phi_t(u, m, t + r)| \leq Ce^{-\eta t}\|\phi\|_{C^1}
\end{align*}
for all $[u, s] \in R$, $m \in M$, and $t, r \geq 0$.
\end{lemma}

\begin{proof}
There exists $C_F > 0$ such that $d(F(w_1), F(w_2)) < C_F$ for all $w_1, w_2 \in R_j$ and $j \in \mathcal{A}$. Fix $C = C_{\mathrm{Ano}}^2C_F$ and $\eta = 1$. Let $\phi \in C_{\mathrm{c}}^1(\Gamma \backslash G, \mathbb R)$. Let $j \in \mathcal{A}, [u, s] \in R_j$, and $t, r \geq 0$. Let $l \in \mathbb Z_{\geq 0}$ such that $t + r \in [\tau_l(u), \tau_{l + 1}(u))$ and let $u_1 = \sigma^l(u) \in U_k$ for some $k \in \mathcal{A}$. Let $s_1 \in S_k$. Then, $[u, s]a_{\tau_l(u)} = \mathcal{P}^l([u, s])$ and $[u_1, s_1]$ are both in $[u_1, S_k]$. Noting that $\tau_l(u) \leq t + r$, we use the above property and recall \cref{lem:HolonomyConstantOnStrongStableLeaves} to get
\begin{align*}
&d(F([u, s])a_{2t + r}m, F([u_1, s_1])a_{2t + r - \tau_l(u)}\vartheta^l(u)^{-1}m) \\
\leq{}&C_{\mathrm{Ano}}^2 e^{-(2t + r - \tau_l(u))}d(F([u, s])a_{\tau_l(u)}\vartheta^l(u), F([u_1, s_1])) \\
={}&C_{\mathrm{Ano}}^2 e^{-(2t + r - \tau_l(u))}d(F(\mathcal{P}^l([u, s])), F([u_1, s_1])) \\
\leq{}&C_{\mathrm{Ano}}^2 C_F e^{-(2t + r - \tau_m(u))} \leq Ce^{-t}.
\end{align*}
Thus
\begin{align*}
|\phi(F([u, s])a_{2t + r}m) - \phi(F([u_1, s_1])a_{2t + r - \tau_l(u)}\vartheta^l(u)^{-1}m)| \leq Ce^{-\eta t}\|\phi\|_{C^1}.
\end{align*}
The lemma follows by integrating over $s_1 \in S_k$ with respect to the probability measure $\nu_{u_1}$.
\end{proof}

\begin{corollary}
\label{cor:ApproximatingMixingOfFrameFlowBySymbolicCodingCorrelationFunction}
There exist  $C > 0$ and $\eta > 0$ such that for all $\phi, \psi \in C_{\mathrm{c}}^1(\Gamma \backslash G, \mathbb R)$, we have
\begin{align*}
\left|\int_{\Gamma \backslash G} \phi(xa_{2t})\psi(x) \, d\mathsf{m}(x) - \frac{1}{\nu_U(\tau)}\Upsilon_{\phi_t, \psi_0}(t)\right| \leq C e^{-\eta t} \|\phi\|_{C^1} \|\psi\|_{C^1}.
\end{align*}
\end{corollary}

\begin{proof}[Proof of \cref{thm:TheoremExponentialMixingOfFrameFlow}]
Recall the remark before \cref{lem:EstimateFunctionByIntegratingOverS} and fix $C_4 > 1$ to be the $C$ described there. Fix $C_1, \eta_1, C_2, \eta_2 > 0$ to be the $C$ and $\eta$ from \cref{cor:ApproximatingMixingOfFrameFlowBySymbolicCodingCorrelationFunction,lem:DecayOfSymbolicCodingCorrelationFunction} respectively. Fix $r' \in \mathbb N$ to be the $r$ from \cref{lem:DecayOfSymbolicCodingCorrelationFunction}. Fix $r = r' + 1$, $C_3 = \frac{C_2}{\nu_U(\tau)}$, $\eta = \frac{1}{2}\min(\eta_1, \eta_2)$, and $C = C_1 + C_3C_4^2$. Let $\phi \in C_{\mathrm{c}}^r(\Gamma \backslash G, \mathbb R)$ and $\psi \in C_{\mathrm{c}}^1(\Gamma \backslash G, \mathbb R)$. Consider the decomposition $\psi = \psi^M + \psi^0$ where $\psi^M$, defined by $\psi^M(x) = \int_M \psi(xm) \, dm$ for all $x \in \Gamma \backslash G$, is $M$-invariant and consequently $\psi^0$ satisfies $\int_M \psi^0(xm) \, dm = 0$ for all $x \in \Gamma \backslash G$. We also use similar notations for $\phi$. Then
\begin{align*}
&\left|\int_{\Gamma \backslash G} \phi(xa_t)\psi(x) \, d\mathsf{m}(x) - \int_{\Gamma \backslash G} \phi \, d\mathsf{m} \cdot \int_{\Gamma \backslash G} \psi \, d\mathsf{m}\right| \\
\leq{}&\left|\int_{\Gamma \backslash G} \phi^M(xa_t)\psi^M(x) \, d\mathsf{m}(x) - \int_{\Gamma \backslash G} \phi^M \, d\mathsf{m} \cdot \int_{\Gamma \backslash G} \psi^M \, d\mathsf{m}\right| \\
{}&+ \left|\int_{\Gamma \backslash G} \phi(xa_t)\psi^0(x) \, d\mathsf{m}(x)\right|.
\end{align*}
Exponential mixing of the geodesic flow has been established by Stoyanov \cite{Sto11} also using Dolgopyat's method. Thus, we know the first term is bounded by $C'e^{-\eta't}\|\phi\|_{C^1} \|\psi\|_{C^1}$ for some $C', \eta' > 0$. Now it suffices to assume that $\psi = \psi^0$, i.e., $\int_M \psi(xm) \, dm = 0$ for all $x \in \Gamma \backslash G$. Thus, we have corresponding functions $\phi_t, \psi_0 \in C(U^{\vartheta, \tau}, \mathbb R)$ with some extensions $\tilde{\phi}_t \in C^r(\tilde{U}^{\vartheta, \tau}, \mathbb R)$ and $\tilde{\psi}_0 \in C^1(\tilde{U}^{\vartheta, \tau}, \mathbb R)$ with $\int_M \tilde{\psi}_0(u, m, t) \, dm = 0$ for all $(u, t) \in \tilde{U}^\tau$ and $\|\tilde{\phi}_t\|_{C^r} \leq C_4\|\phi\|_{C^r}$ for all $t \geq 0$ and $\|\tilde{\psi}_0\|_{C^1} \leq C_4\|\psi\|_{C^1}$. Hence by \cref{cor:ApproximatingMixingOfFrameFlowBySymbolicCodingCorrelationFunction,lem:DecayOfSymbolicCodingCorrelationFunction}, for all $t > 0$, letting $t' = \frac{t}{2}$, we have
\begin{align*}
\left|\int_{\Gamma \backslash G} \phi(xa_t)\psi(x) \, d\mathsf{m}(x)\right| \leq{}&\frac{1}{\nu_U(\tau)}\left|\Upsilon_{\phi_{t'}, \psi_0}(t')\right| + C_1 e^{- \eta_1 t'} \|\phi\|_{C^1} \|\psi\|_{C^1} \\
\leq{}&C_3C_4^2 e^{- \eta_2 t'} \|\phi\|_{C^r} \|\psi\|_{C^1} + C_1 e^{- \eta_1 t'} \|\phi\|_{C^1} \|\psi\|_{C^1} \\
\leq{}&C e^{-\eta t} \|\phi\|_{C^r} \|\psi\|_{C^1}.
\end{align*}
\end{proof}

\appendix
\section{Ruelle--Perron--Frobenius theorem for the smooth setting}
\label{sec:RPF_Theorem}
We need to work on compact sets and hence, without loss of generality, we assume that $\tilde{U}$ from \cref{subsec:ModifiedConstructionsUsingTheSmoothStructureOnG} is closed by taking the closures $\overline{\tilde{U}_j}$ for all $j \in \mathcal{A}$. We note that the maps $\sigma^{-(j, k)}$ and $\tau_{(j, k)}$ can be extended to the closures and they are smooth, for all admissible pairs $(j, k)$. Recall that we have a Riemannian metric on $\tilde{U}$ which is induced from the one chosen on $G$. By \cref{lem:SigmaHyperbolicity}, the inverse maps $\sigma^{-\alpha}$ are \emph{eventually} contracting for all admissible sequences $\alpha$. We now need to slightly modify the Riemannian metric on $\tilde{U}$ to ensure that they are \emph{strictly} contracting. Such a Riemannian metric is called an \emph{adapted metric} and it can be constructed by a technique which involves averaging the original Riemannian metric on $\T^1(X)$ over sufficiently long orbits of the forward and backward geodesic flow. This is a standard trick which is originally due to Mather \cite{Mat68} in the case of diffeomorphisms but the flow version is similar and it can be found in \cite[Lemma 2.2]{Man98} for example. Then the new Riemannian metric induces the desired modified Riemannian metric on $\tilde{U}$ and we use this henceforth. Now we can assume $c_0 = 1$ in \cref{lem:SigmaHyperbolicity}, i.e., for all $j \in \mathbb Z_{\geq 0}$ and admissible sequences $\alpha = (\alpha_0, \alpha_1, \dotsc, \alpha_j)$, we have
\begin{align*}
\frac{1}{\kappa_1^j} \leq \|(d\sigma^{-\alpha})_u\|_{\mathrm{op}} \leq \frac{1}{\kappa_2^j} < 1 \qquad \text{for all $u \in \tilde{U}_{\alpha_j}$}.
\end{align*}

We denote by $\nabla: \Gamma(\T(\tilde{U})) \to \Gamma(\T(\tilde{U}) \otimes \T^*(\tilde{U}))$ the Levi--Civita connection corresponding to the Riemannian metric on $\tilde{U}$. This extends to the connection $\nabla: \Gamma(E) \to \Gamma(E \otimes \T^*(\tilde{U}))$ for any tensor bundle $E$ over $\tilde{U}$. Let $|\cdot|$ denote the pointwise norm of tensors on $\tilde{U}$ so that applying it gives nonnegative functions in $C^\infty(\tilde{U}, \mathbb R)$. We first have the following \cref{lem:TransferOperatorPreservesCone} which show that transfer operators preserve a certain cone.

\begin{lemma}
\label{lem:TransferOperatorPreservesCone}
For all $a \in \mathbb R$, there exists $\{C_k > 0: k \in \mathbb Z_{\geq 0}\}$ with $C_0 = 1$ such that $\tilde{\mathcal{L}}_{a\tau}(\Lambda) \subset \Lambda$ for the cone
\begin{align*}
\Lambda = \{h \in C^\infty(\tilde{U}, \mathbb R): |\nabla^k h| \leq C_k h \text{ for all } k \in \mathbb Z_{\geq 0}\}.
\end{align*}
\end{lemma}

\begin{proof}
Let $a \in \mathbb R$. It suffices to inductively construct $\{C_k > 0: k \in \mathbb Z_{\geq 0}\}$ such that for all $k \in \mathbb Z_{\geq 0}$ and $h \in C^\infty(\tilde{U}, \mathbb R)$, if $|\nabla^{k'} h| \leq C_{k'} h$ for all $0 \leq k' \leq k$, then $\big|\nabla^k \big(\tilde{\mathcal{L}}_{a\tau}(h)\big)\big| \leq C_k \tilde{\mathcal{L}}_{a\tau}(h)$. The base case $k = 0$ is trivial by choosing $C_0 = 1$. Now assume we have chosen appropriate $C_0, C_1, \dotsc, C_{k - 1} > 0$ for some $k \in \mathbb N$. Let $C_k > 0$ be some constant. Suppose $h \in C^\infty(\tilde{U}, \mathbb R)$ such that $|\nabla^{k'} h| \leq C_{k'} h$ for all $0 \leq k' \leq k$. We will show that we also have $\big|\nabla^k \big(\tilde{\mathcal{L}}_{a\tau}(h)\big)\big| \leq C_k \tilde{\mathcal{L}}_{a\tau}(h)$ if $C_k$ is sufficiently large. It can be computed using coordinate charts and by induction on $k$ that the $k$\textsuperscript{th} covariant derivative of $\tilde{\mathcal{L}}_{a\tau}(h)$ is the tensor of the form
\begin{align*}
\nabla^k \big(\tilde{\mathcal{L}}_{a\tau}(h)\big) = \sum_{\alpha: \len(\alpha) = 1} e^{a (\tau_\alpha \circ \sigma^{-\alpha})} \left(T_{\alpha, k} + \nabla^k h \circ \underbrace{d\sigma^{-\alpha} \otimes d\sigma^{-\alpha} \otimes \dotsb \otimes d\sigma^{-\alpha}}_{k}\right)
\end{align*}
where $T_{\alpha, k}$ is a tensor which is a sum of terms composed of $a$, covariant derivatives of $\tau_\alpha$, various derivatives of $\sigma^{-\alpha}$, and covariant derivatives of $h$ of orders \emph{strictly} less than $k$. Moreover, each term has exactly one factor of the covariant derivative $\nabla^{k'}h$ for some $k' < k$. Note that all orders of derivatives of $\tau_\alpha$ and $\sigma^{-\alpha}$ are bounded on the compact set $\tilde{U}$ for all admissible sequences $\alpha$. Hence, taking the norm and using the induction hypothesis, we see that there exists a constant $C > 0$ such that
\begin{align*}
\big|\nabla^k \big(\tilde{\mathcal{L}}_{a\tau}(h)\big)\big| \leq{}&\sum_{\alpha: \len(\alpha) = 1} e^{a (\tau_\alpha \circ \sigma^{-\alpha})} \left(C(h \circ \sigma^{-\alpha})\rule{0cm}{0.7cm}\right. \\
&\left.{}+ (|\nabla^k h| \circ \sigma^{-\alpha}) \cdot \underbrace{|\sigma^{-\alpha}|_{C^1} \cdot |\sigma^{-\alpha}|_{C^1} \dotsb |\sigma^{-\alpha}|_{C^1}}_{k}\right) \\
\leq{}&\sum_{\alpha: \len(\alpha) = 1} e^{a (\tau_\alpha \circ \sigma^{-\alpha})} \left(C\left(h \circ \sigma^{-\alpha}\right) + C_k \left(h \circ \sigma^{-\alpha}\right) \cdot \frac{1}{\kappa_2^k}\right) \\
\leq{}&\left(C + \frac{C_k}{\kappa_2^k}\right) \sum_{\alpha: \len(\alpha) = 1} e^{a (\tau_\alpha \circ \sigma^{-\alpha})} \left(h \circ \sigma^{-\alpha}\right) \\
\leq{}&\left(C + \frac{C_k}{\kappa_2^k}\right) \tilde{\mathcal{L}}_{a\tau}(h).
\end{align*}
Thus, we have $\big|\nabla^k \big(\tilde{\mathcal{L}}_{a\tau}(h)\big)\big| \leq C_k \tilde{\mathcal{L}}_{a\tau}(h)$ if $C + \frac{C_k}{\kappa_2^k} \leq C_k$. Since $\frac{1}{\kappa_2^k} \in (0, 1)$, this is possible so long as $C_k \geq C\left(1 - \frac{1}{\kappa_2^k}\right)^{-1}$.
\end{proof}

The following \cref{thm:SmoothRPF} shows that the eigenvector $h_a \in C^{\Lip(d)}(U, \mathbb R)$ for $\mathcal{L}_{-(\delta_\Gamma + a)\tau}$ corresponding to its maximal simple eigenvalue can indeed be extended to a \emph{smooth} eigenvector $h_a \in C^\infty(\tilde{U}, \mathbb R)$ for $\tilde{\mathcal{L}}_{-(\delta_\Gamma + a)\tau}$. To ensure positivity of $h_a$, we assume $U$ was enlarged to $\tilde{U}$ using a sufficiently small $\delta > 0$. We follow the proof of \cite[Theorem 2.2]{PP90}.

\begin{theorem}
\label{thm:SmoothRPF}
For all $a \in \mathbb R$, the operator $\tilde{\mathcal{L}}_{a\tau}$ has a positive eigenvector $h \in C^\infty(\tilde{U}, \mathbb R)$ corresponding to its maximal eigenvalue which coincides with that of $\mathcal{L}_{a\tau}$.
\end{theorem}

\begin{proof}
Let $a \in \mathbb R$. Let $\{C_k > 0: k \in \mathbb N\}$ be the corresponding set of constants provided by \cref{lem:TransferOperatorPreservesCone}. Consider the convex set $\Lambda \subset C(\tilde{U}, \mathbb R)$ defined by
\begin{align*}
\Lambda = \{h \in C^\infty(\tilde{U}, \mathbb R): 0 \leq h \leq 1, |\nabla^k h| \leq C_k h \text{ for all } k \in \mathbb N\}.
\end{align*}
Note that all covariant derivatives are uniformly bounded over all $h \in \Lambda$ by virtue of the scaling $0 \leq h \leq 1$. Then $\Lambda$ is equicontinuous and uniformly bounded and hence by Arzel\`{a}--Ascoli, $\Lambda \subset C(\tilde{U}, \mathbb R)$ is precompact. Thus, to show that $\Lambda \subset C(\tilde{U}, \mathbb R)$ is compact, it suffices to show that $\Lambda \subset C(\tilde{U}, \mathbb R)$ is closed. To show this, let $\{\phi_j\}_{j \in \mathbb N} \subset \Lambda$ be a sequence which converges in $C(\tilde{U}, \mathbb R)$ to some $\phi$. Then $\{\phi_j\}_{j \in \mathbb N}$ converges uniformly to $\phi$. Now, using coordinate charts and the Landau--Kolmogorov inequality, we can deduce that $\{\partial_\alpha \phi_j\}_{j \in \mathbb N}$ is also uniformly convergent for all multi-indices $\alpha$. This implies that in fact $\phi \in C^\infty(\tilde{U}, \mathbb R)$ and $\{\partial_\alpha \phi_j\}_{j \in \mathbb N}$ converges uniformly to $\partial_\alpha \phi$ for all multi-indices $\alpha$. It is then easy to see that we also have $0 \leq \phi \leq 1$ and $|\nabla^k \phi| \leq C_k \phi$ for all $k \in \mathbb N$ which implies $\phi \in \Lambda$. So $\Lambda \subset C(\tilde{U}, \mathbb R)$ is closed and hence compact.

Let $j \in \mathbb N$. Define the map $\tilde{\mathcal{L}}_{a\tau, j}: \Lambda \to C^\infty(\tilde{U}, \mathbb R)$ by $\tilde{\mathcal{L}}_{a\tau, j}(h) = \frac{\tilde{\mathcal{L}}_{a\tau}(h + \frac{1}{j})}{\left\|\tilde{\mathcal{L}}_{a\tau}(h + \frac{1}{j})\right\|_\infty}$ for all $h \in \Lambda$. Then, $\tilde{\mathcal{L}}_{a\tau, j}(\Lambda) \subset \Lambda$ by \cref{lem:TransferOperatorPreservesCone} where $\Lambda \subset C(\tilde{U}, \mathbb R)$ is a compact convex set. Now, by Schauder--Tychonoff fixed point theorem, we obtain $h_j \in \Lambda$ such that $\tilde{\mathcal{L}}_{a\tau, j}(h_j) = h_j$ which implies $\tilde{\mathcal{L}}_{a\tau}\big(h_j + \frac{1}{j}\big) = \big\|\tilde{\mathcal{L}}_{a\tau}\big(h_j + \frac{1}{j}\big)\big\|_\infty h_j$. By compactness of $\Lambda$, we can choose $h \in \Lambda$ to be any limit point of the sequence $\{h_j\}_{j \in \mathbb N}$. Then by continuity, we have $\tilde{\mathcal{L}}_{a\tau}(h) = \big\|\tilde{\mathcal{L}}_{a\tau}(h)\big\|_\infty h$ which shows that $h \in C^\infty(\tilde{U}, \mathbb R)$ is an eigenvector of $\tilde{\mathcal{L}}_{a\tau}$. From the proof of \cite[Theorem 2.2]{PP90}, we see that $h|_U \in C^{\Lip(d)}(U, \mathbb R)$ is a positive eigenvector for $\mathcal{L}_{a\tau}$ corresponding to its maximal simple eigenvalue which must coincide with $\big\|\tilde{\mathcal{L}}_{a\tau}(h)\big\|_\infty$. The eigenvalue is also maximal for $\tilde{\mathcal{L}}_{a\tau}$ because any eigenvector of $\tilde{\mathcal{L}}_{a\tau}$ restricts via $|_U: C(\tilde{U}, \mathbb R) \to C(U, \mathbb R)$ to an eigenvector of $\mathcal{L}_{a\tau}$. Now, ensuring that $\delta > 0$ used to enlarge $U$ to $\tilde{U}$ in \cref{subsec:ModifiedConstructionsUsingTheSmoothStructureOnG} is sufficiently small, we can guarantee that $h \in C^\infty(\tilde{U}, \mathbb R)$ is also positive by uniform continuity.
\end{proof}

%\newpage
\nocite{*}
\bibliographystyle{alpha_name-year-title}
\bibliography{References}

\begin{thebibliography}{DFSU20}

\bibitem[AGY06]{AGY06}
Artur Avila, S\'{e}bastien Gou\"{e}zel, and Jean-Christophe Yoccoz.
\newblock Exponential mixing for the {T}eichm\"{u}ller flow.
\newblock {\em Publ. Math. Inst. Hautes \'{E}tudes Sci.}, (104):143--211, 2006.

\bibitem[Bab02]{Bab02}
Martine Babillot.
\newblock On the mixing property for hyperbolic systems.
\newblock {\em Israel J. Math.}, 129:61--76, 2002.

\bibitem[BO12]{BO12}
Yves Benoist and Hee Oh.
\newblock Effective equidistribution of {$S$}-integral points on symmetric
  varieties.
\newblock {\em Ann. Inst. Fourier (Grenoble)}, 62(5):1889--1942, 2012.

\bibitem[Bor16]{Bor16}
David Borthwick.
\newblock {\em Spectral theory of infinite-area hyperbolic surfaces}, volume
  318 of {\em Progress in Mathematics}.
\newblock Birkh\"{a}user/Springer, [Cham], second edition, 2016.

\bibitem[BKS10]{BKS10}
Jean Bourgain, Alex Kontorovich, and Peter Sarnak.
\newblock Sector estimates for hyperbolic isometries.
\newblock {\em Geom. Funct. Anal.}, 20(5):1175--1200, 2010.

\bibitem[Bow93]{Bow93}
B.~H. Bowditch.
\newblock Geometrical finiteness for hyperbolic groups.
\newblock {\em J. Funct. Anal.}, 113(2):245--317, 1993.

\bibitem[Bow95]{Bow95}
B.~H. Bowditch.
\newblock Geometrical finiteness with variable negative curvature.
\newblock {\em Duke Math. J.}, 77(1):229--274, 1995.

\bibitem[Bow70]{Bow70}
Rufus Bowen.
\newblock Markov partitions for {A}xiom {${\rm A}$} diffeomorphisms.
\newblock {\em Amer. J. Math.}, 92:725--747, 1970.

\bibitem[Bow71]{Bow71}
Rufus Bowen.
\newblock Periodic points and measures for {A}xiom {$A$} diffeomorphisms.
\newblock {\em Trans. Amer. Math. Soc.}, 154:377--397, 1971.

\bibitem[Bow73]{Bow73}
Rufus Bowen.
\newblock Symbolic dynamics for hyperbolic flows.
\newblock {\em Amer. J. Math.}, 95:429--460, 1973.

\bibitem[Bow08]{Bow08}
Rufus Bowen.
\newblock {\em Equilibrium states and the ergodic theory of {A}nosov
  diffeomorphisms}, volume 470 of {\em Lecture Notes in Mathematics}.
\newblock Springer-Verlag, Berlin, revised edition, 2008.
\newblock With a preface by David Ruelle. Edited by Jean-Ren\'{e} Chazottes.

\bibitem[Bri82]{Bri82}
M.~Brin.
\newblock Ergodic theory of frame flows.
\newblock In {\em Ergodic theory and dynamical systems, {II} ({C}ollege {P}ark,
  {M}d., 1979/1980)}, volume~21 of {\em Progr. Math.}, pages 163--183.
  Birkh\"{a}user, Boston, Mass., 1982.

\bibitem[BP74]{BP74}
M.~I. Brin and Ja.~B. Pesin.
\newblock Partially hyperbolic dynamical systems.
\newblock {\em Math. USSR Izv.}, 8:177--218, 1974.

\bibitem[Che02]{Che02}
N.~Chernov.
\newblock Invariant measures for hyperbolic dynamical systems.
\newblock In {\em Handbook of dynamical systems, {V}ol. 1{A}}, pages 321--407.
  North-Holland, Amsterdam, 2002.

\bibitem[CdV85]{Col85}
Yves Colin~de Verdi\`ere.
\newblock Th\'{e}orie spectrale des surfaces de {R}iemann d'aire infinie.
\newblock Number 132, pages 259--275. 1985.
\newblock Colloquium in honor of Laurent Schwartz, Vol. 2 (Palaiseau, 1983).

\bibitem[DFSU20]{DFSU20}
Tushar Das, Lior Fishman, David Simmons, and Mariusz Urba\'{n}ski.
\newblock Extremality and dynamically defined measures, part {II}: {M}easures
  from conformal dynamical systems.
\newblock {\em Ergodic Theory Dynam. Systems}, pages 1--38, 2020.
\newblock Ahead of Print.

\bibitem[Dol98]{Dol98}
Dmitry Dolgopyat.
\newblock On decay of correlations in {A}nosov flows.
\newblock {\em Ann. of Math. (2)}, 147(2):357--390, 1998.

\bibitem[Dol02]{Dol02}
Dmitry Dolgopyat.
\newblock On mixing properties of compact group extensions of hyperbolic
  systems.
\newblock {\em Israel J. Math.}, 130:157--205, 2002.

\bibitem[DRS93]{DRS93}
W.~Duke, Z.~Rudnick, and P.~Sarnak.
\newblock Density of integer points on affine homogeneous varieties.
\newblock {\em Duke Math. J.}, 71(1):143--179, 1993.

\bibitem[EO21]{EO21}
Samuel~C. Edwards and Hee Oh.
\newblock Spectral gap and exponential mixing on geometrically finite
  hyperbolic manifolds.
\newblock {\em Duke Math. J.}, arXiv:2001.03377, 2021.
\newblock To appear.

\bibitem[EM93]{EM93}
Alex Eskin and Curt McMullen.
\newblock Mixing, counting, and equidistribution in {L}ie groups.
\newblock {\em Duke Math. J.}, 71(1):181--209, 1993.

\bibitem[Gam02]{Gam02}
Alex Gamburd.
\newblock On the spectral gap for infinite index ``congruence'' subgroups of
  {${\rm SL}_2(\bold Z)$}.
\newblock {\em Israel J. Math.}, 127:157--200, 2002.

\bibitem[GS14]{GS14}
Alexander Gorodnik and Ralf Spatzier.
\newblock Exponential mixing of nilmanifold automorphisms.
\newblock {\em J. Anal. Math.}, 123:355--396, 2014.

\bibitem[GLZ04]{GLZ04}
Laurent Guillop\'{e}, Kevin~K. Lin, and Maciej Zworski.
\newblock The {S}elberg zeta function for convex co-compact {S}chottky groups.
\newblock {\em Comm. Math. Phys.}, 245(1):149--176, 2004.

\bibitem[HM79]{HM79}
Roger~E. Howe and Calvin~C. Moore.
\newblock Asymptotic properties of unitary representations.
\newblock {\em J. Functional Analysis}, 32(1):72--96, 1979.

\bibitem[Kai90]{Kai90}
Vadim~A. Kaimanovich.
\newblock Invariant measures of the geodesic flow and measures at infinity on
  negatively curved manifolds.
\newblock {\em Ann. Inst. H. Poincar\'{e} Phys. Th\'{e}or.}, 53(4):361--393,
  1990.
\newblock Hyperbolic behaviour of dynamical systems (Paris, 1990).

\bibitem[Kai91]{Kai91}
Vadim~A. Kaimanovich.
\newblock Bowen-{M}argulis and {P}atterson measures on negatively curved
  compact manifolds.
\newblock In {\em Dynamical systems and related topics ({N}agoya, 1990)},
  volume~9 of {\em Adv. Ser. Dynam. Systems}, pages 223--232. World Sci. Publ.,
  River Edge, NJ, 1991.

\bibitem[Kat95]{Kat95}
Tosio Kato.
\newblock {\em Perturbation theory for linear operators}.
\newblock Classics in Mathematics. Springer-Verlag, Berlin, 1995.
\newblock Reprint of the 1980 edition.

\bibitem[KO21]{KO21}
Dubi Kelmer and Hee Oh.
\newblock Shrinking targets for geodesic flow on geometrically finite
  hyperbolic manifolds.
\newblock {\em J. Mod. Dyn.}, arXiv:1812.05251, 2021.
\newblock To appear.

\bibitem[KM96]{KM96}
D.~Y. Kleinbock and G.~A. Margulis.
\newblock Bounded orbits of nonquasiunipotent flows on homogeneous spaces.
\newblock In {\em Sina\u{\i}'s {M}oscow {S}eminar on {D}ynamical {S}ystems},
  volume 171 of {\em Amer. Math. Soc. Transl. Ser. 2}, pages 141--172. Amer.
  Math. Soc., Providence, RI, 1996.

\bibitem[KO11]{KO11}
Alex Kontorovich and Hee Oh.
\newblock Apollonian circle packings and closed horospheres on hyperbolic
  3-manifolds.
\newblock {\em J. Amer. Math. Soc.}, 24(3):603--648, 2011.
\newblock With an appendix by Oh and Nimish Shah.

\bibitem[Lal89]{Lal89}
Steven~P. Lalley.
\newblock Renewal theorems in symbolic dynamics, with applications to geodesic
  flows, non-{E}uclidean tessellations and their fractal limits.
\newblock {\em Acta Math.}, 163(1-2):1--55, 1989.

\bibitem[LY73]{LY73}
A.~Lasota and James~A. Yorke.
\newblock On the existence of invariant measures for piecewise monotonic
  transformations.
\newblock {\em Trans. Amer. Math. Soc.}, 186:481--488 (1974), 1973.

\bibitem[LP82]{LP82}
Peter~D. Lax and Ralph~S. Phillips.
\newblock The asymptotic distribution of lattice points in {E}uclidean and
  non-{E}uclidean spaces.
\newblock {\em J. Funct. Anal.}, 46(3):280--350, 1982.

\bibitem[LO13]{LO13}
Min Lee and Hee Oh.
\newblock Effective circle count for {A}pollonian packings and closed
  horospheres.
\newblock {\em Geom. Funct. Anal.}, 23(2):580--621, 2013.

\bibitem[Mag15]{Mag15}
Michael Magee.
\newblock Quantitative spectral gap for thin groups of hyperbolic isometries.
\newblock {\em J. Eur. Math. Soc. (JEMS)}, 17(1):151--187, 2015.

\bibitem[Man98]{Man98}
Brian~S. Mangum.
\newblock Incompressible surfaces and pseudo-{A}nosov flows.
\newblock {\em Topology Appl.}, 87(1):29--51, 1998.

\bibitem[MMO14]{MMO14}
Gregory Margulis, Amir Mohammadi, and Hee Oh.
\newblock Closed geodesics and holonomies for {K}leinian manifolds.
\newblock {\em Geom. Funct. Anal.}, 24(5):1608--1636, 2014.

\bibitem[Mar04]{Mar04}
Grigoriy~A. Margulis.
\newblock {\em On some aspects of the theory of {A}nosov systems}.
\newblock Springer Monographs in Mathematics. Springer-Verlag, Berlin, 2004.
\newblock With a survey by Richard Sharp: Periodic orbits of hyperbolic flows,
  Translated from the Russian by Valentina Vladimirovna Szulikowska.

\bibitem[Mat68]{Mat68}
John~N. Mather.
\newblock Characterization of {A}nosov diffeomorphisms.
\newblock {\em Nederl. Akad. Wetensch. Proc. Ser. A 71 = Indag. Math.},
  30:479--483, 1968.

\bibitem[MM87]{MM87}
Rafe~R. Mazzeo and Richard~B. Melrose.
\newblock Meromorphic extension of the resolvent on complete spaces with
  asymptotically constant negative curvature.
\newblock {\em J. Funct. Anal.}, 75(2):260--310, 1987.

\bibitem[MO15]{MO15}
Amir Mohammadi and Hee Oh.
\newblock Matrix coefficients, counting and primes for orbits of geometrically
  finite groups.
\newblock {\em J. Eur. Math. Soc. (JEMS)}, 17(4):837--897, 2015.

\bibitem[Mok78]{Mok78}
Kam~Ping Mok.
\newblock On the differential geometry of frame bundles of {R}iemannian
  manifolds.
\newblock {\em J. Reine Angew. Math.}, 302:16--31, 1978.

\bibitem[Moo87]{Moo87}
Calvin~C. Moore.
\newblock Exponential decay of correlation coefficients for geodesic flows.
\newblock In {\em Group representations, ergodic theory, operator algebras, and
  mathematical physics ({B}erkeley, {C}alif., 1984)}, volume~6 of {\em Math.
  Sci. Res. Inst. Publ.}, pages 163--181. Springer, New York, 1987.

\bibitem[Nau05]{Nau05}
Fr\'{e}d\'{e}ric Naud.
\newblock Expanding maps on {C}antor sets and analytic continuation of zeta
  functions.
\newblock {\em Ann. Sci. \'{E}cole Norm. Sup. (4)}, 38(1):116--153, 2005.

\bibitem[OS13]{OS13}
Hee Oh and Nimish~A. Shah.
\newblock Equidistribution and counting for orbits of geometrically finite
  hyperbolic groups.
\newblock {\em J. Amer. Math. Soc.}, 26(2):511--562, 2013.

\bibitem[OW16]{OW16}
Hee Oh and Dale Winter.
\newblock Uniform exponential mixing and resonance free regions for convex
  cocompact congruence subgroups of {${\rm SL}_2(\Bbb{Z})$}.
\newblock {\em J. Amer. Math. Soc.}, 29(4):1069--1115, 2016.

\bibitem[OW17]{OW17}
Hee Oh and Dale Winter.
\newblock Prime number theorems and holonomies for hyperbolic rational maps.
\newblock {\em Invent. Math.}, 208(2):401--440, 2017.

\bibitem[PP90]{PP90}
William Parry and Mark Pollicott.
\newblock Zeta functions and the periodic orbit structure of hyperbolic
  dynamics.
\newblock {\em Ast\'{e}risque}, (187-188):268, 1990.

\bibitem[Pat76]{Pat76}
S.~J. Patterson.
\newblock The limit set of a {F}uchsian group.
\newblock {\em Acta Math.}, 136(3-4):241--273, 1976.

\bibitem[Pat88]{Pat88}
S.~J. Patterson.
\newblock On a lattice-point problem in hyperbolic space and related questions
  in spectral theory.
\newblock {\em Ark. Mat.}, 26(1):167--172, 1988.

\bibitem[PPS15]{PPS15}
Fr\'{e}d\'{e}ric Paulin, Mark Pollicott, and Barbara Schapira.
\newblock Equilibrium states in negative curvature.
\newblock {\em Ast\'{e}risque}, (373):viii+281, 2015.

\bibitem[PS16]{PS16}
Vesselin Petkov and Luchezar Stoyanov.
\newblock Ruelle transfer operators with two complex parameters and
  applications.
\newblock {\em Discrete Contin. Dyn. Syst.}, 36(11):6413--6451, 2016.

\bibitem[Pol85]{Pol85}
Mark Pollicott.
\newblock On the rate of mixing of {A}xiom {A} flows.
\newblock {\em Invent. Math.}, 81(3):413--426, 1985.

\bibitem[Rat73]{Rat73}
M.~Ratner.
\newblock Markov partitions for {A}nosov flows on {$n$}-dimensional manifolds.
\newblock {\em Israel J. Math.}, 15:92--114, 1973.

\bibitem[Rat87]{Rat87}
Marina Ratner.
\newblock The rate of mixing for geodesic and horocycle flows.
\newblock {\em Ergodic Theory Dynam. Systems}, 7(2):267--288, 1987.

\bibitem[Rob03]{Rob03}
Thomas Roblin.
\newblock Ergodicit\'{e} et \'{e}quidistribution en courbure n\'{e}gative.
\newblock {\em M\'{e}m. Soc. Math. Fr. (N.S.)}, (95):vi+96, 2003.

\bibitem[Rud82]{Rud82}
Daniel~J. Rudolph.
\newblock Ergodic behaviour of {S}ullivan's geometric measure on a
  geometrically finite hyperbolic manifold.
\newblock {\em Ergodic Theory Dynam. Systems}, 2(3-4):491--512 (1983), 1982.

\bibitem[Rue89]{Rue89}
David Ruelle.
\newblock The thermodynamic formalism for expanding maps.
\newblock {\em Comm. Math. Phys.}, 125(2):239--262, 1989.

\bibitem[SW99]{SW99}
Peter Sarnak and Masato Wakayama.
\newblock Equidistribution of holonomy about closed geodesics.
\newblock {\em Duke Math. J.}, 100(1):1--57, 1999.

\bibitem[Sas58]{Sas58}
Shigeo Sasaki.
\newblock On the differential geometry of tangent bundles of {R}iemannian
  manifolds.
\newblock {\em T\^{o}hoku Math. J. (2)}, 10:338--354, 1958.

\bibitem[Sch05]{Sch05}
Barbara Schapira.
\newblock Equidistribution of the horocycles of a geometrically finite surface.
\newblock {\em Int. Math. Res. Not.}, (40):2447--2471, 2005.

\bibitem[Sto05]{Sto05}
Luchezar Stoyanov.
\newblock On the {R}uelle-{P}erron-{F}robenius theorem.
\newblock {\em Asymptot. Anal.}, 43(1-2):131--150, 2005.

\bibitem[Sto11]{Sto11}
Luchezar Stoyanov.
\newblock Spectra of {R}uelle transfer operators for axiom {A} flows.
\newblock {\em Nonlinearity}, 24(4):1089--1120, 2011.

\bibitem[Sul79]{Sul79}
Dennis Sullivan.
\newblock The density at infinity of a discrete group of hyperbolic motions.
\newblock {\em Inst. Hautes \'{E}tudes Sci. Publ. Math.}, (50):171--202, 1979.

\bibitem[Sul84]{Sul84}
Dennis Sullivan.
\newblock Entropy, {H}ausdorff measures old and new, and limit sets of
  geometrically finite {K}leinian groups.
\newblock {\em Acta Math.}, 153(3-4):259--277, 1984.

\bibitem[War72]{War72}
Garth Warner.
\newblock {\em Harmonic analysis on semi-simple {L}ie groups. {I}}.
\newblock Springer-Verlag, New York-Heidelberg, 1972.
\newblock Die Grundlehren der mathematischen Wissenschaften, Band 188.

\bibitem[Win15]{Win15}
Dale Winter.
\newblock Mixing of frame flow for rank one locally symmetric spaces and
  measure classification.
\newblock {\em Israel J. Math.}, 210(1):467--507, 2015.

\end{thebibliography}
\end{document}